\documentclass[11pt]{amsart}
\usepackage{amssymb, amsmath,latexsym,amsfonts,amsbsy, amsthm,mathtools,graphicx,color}
\usepackage{float}
\usepackage{hyperref}
\numberwithin{equation}{section}
\allowdisplaybreaks
\usepackage{mathrsfs,enumerate,extarrows}
\let\al=\alpha

\let\g=\gamma

\let\la=\lambda

\let\wt=\widetilde
\let\wh=\widehat

\let\th=\theta
\let\pa=\partial

\def\cC{{\mathcal C}}
\def\cD{{\mathcal D}}

\def\cL{{\mathcal L}}
\def\cM{{\mathcal M}}
\def\cN{{\mathcal N}}
\def\cO{{\mathcal O}}

\def\cR{{\mathcal R}}
\def\cS{{\mathcal S}}

\def\cU{{\mathcal U}}
\def\cV{{\mathcal V}}
\def\cW{{\mathcal W}}

\def\R{\mathbb R}
\def\T{\mathbb T}
\def\Z{\mathbb Z}
\def\N{\mathbb N}

\def\e{\mathrm e}

\def\dive{\mathop{\rm div}\nolimits}

\newcommand{\beq}{\begin{equation}}
	\newcommand{\eeq}{\end{equation}}
\newcommand{\ben}{\begin{eqnarray}}
	\newcommand{\een}{\end{eqnarray}}
\newcommand{\beno}{\begin{eqnarray*}}
	\newcommand{\eeno}{\end{eqnarray*}}

\newtheorem{problem}{Problem}

\newtheorem{theorem}{Theorem}[section]

\newtheorem{lemma}[theorem]{Lemma}
\newtheorem{proposition}[theorem]{Proposition}
\newtheorem{corollary}[theorem]{Corollary}
\theoremstyle{remark}
\newtheorem{remark}[theorem]{Remark}
\newtheorem{Theorem}{Theorem}[section]

\newtheorem{Remark}[Theorem]{Remark}

\begin{document}

\title[Blow-up of compressible Navier-Stokes equations]
{Blow-up of the 3-D compressible Navier-Stokes equations for monatomic gases}

\author[F. Shao]{Feng Shao}
\address{School of Mathematical Sciences, Peking University, Beijing 100871,  China}
\email{fshao@stu.pku.edu.cn}

\author[S. Wang]{Shumao Wang}
\address{School of Mathematical Sciences, Peking University, Beijing 100871,  China}
\email{shumaowang@pku.edu.cn}

\author[D. Wei]{Dongyi Wei}
\address{School of Mathematical Sciences, Peking University, Beijing 100871,  China}
\email{jnwdyi@pku.edu.cn}

\author[Z. Zhang]{Zhifei Zhang}
\address{School of Mathematical Sciences, Peking University, Beijing 100871, China}
\email{zfzhang@math.pku.edu.cn}

\date{\today}

\begin{abstract}
In this paper, we prove the blow-up of the $3$-D isentropic compressible Navier-Stokes equations for the adiabatic exponent $\gamma=5/3$, which corresponds to the law of monatomic gases.
This is the degenerate case in the sense of Merle, Raphaël, Rodnianski, and Szeftel \cite{MRRJ2,MRRJ3}. Motivated by these breakthrough works, we first establish the existence of a sequence of smooth, self-similar imploding solutions to the compressible Euler equations for $\gamma=5/3$.  Subsequently, we utilize these self-similar profiles to construct smooth, asymptotically self-similar blow-up solutions to the compressible Navier-Stokes equations for monatomic gases.
\end{abstract}

\maketitle
\tableofcontents

\section{Introduction}

In this paper, we consider the isentropic compressible Euler equations and compressible Navier-Stokes equations
\begin{equation}\label{CNS}
	\begin{cases}
		\rho\pa_t\mathbf u+\rho \mathbf u\cdot\nabla \mathbf u=-\frac1\g\nabla\left(\rho^\g\right)+\nu\Delta \mathbf u,\quad (t,y)\in\R^+\times\R^d\ (\text{or}\ \R^+\times\T^d),\\
		\pa_t\rho+\dive(\rho \mathbf u)=0.
	\end{cases}
\end{equation}
Here $\rho=\rho(t,y)>0$ is the density, $\mathbf u=\mathbf u(t, y)\in\R^d$ is the velocity, $d\in\Z\cap[1,+\infty)$ is the physical dimension, $\T^d$ is the $d$-dimensional torus, $\g>1$ is the adiabatic exponent, and $\nu\in \{0,1\}$ is the viscosity, where $\nu=0$ for Euler and $\nu=1$ for Navier-Stokes.

\if0
The solution to (\ref{CNS}) satisfies the following conservation laws:
\begin{itemize}
\item mass
\[
\int_{\mathbb{R}^{d}}\rho(t,y)\,\mathrm dy=\int_{\mathbb{R}^{d}}\rho_{0}(y)\,\mathrm dy,
\]
\item momentum
\[
\int_{\mathbb{R}^{d}}\rho(t,y)\mathbf{u}(t,y)\,\mathrm dy=\int_{\mathbb{R}^{d}}\rho_{0}(y)\mathbf{u}_{0}(y)\,\mathrm dy,
\]
\item energy
\begin{align*}
 & \int_{\mathbb{R}^{d}}\left(\frac{1}{2}\rho(t,y)|\mathbf{u}(t,y)|^{2}+\frac{1}{\gamma(\gamma-1)}\rho(t,y)^{\gamma}\right)\,\mathrm dy+\int_{0}^{t}\nu||\nabla\mathbf{u}(s,\cdot)||_{L^{2}(\mathbb{R}^{d})}^{2}\,\mathrm ds\\
= & \int_{\mathbb{R}^{d}}\left(\frac{1}{2}\rho_{0}(y)|\mathbf{u}_{0}(y)|^{2}+\frac{1}{\gamma(\gamma-1)}\rho_{0}(y)^{\gamma}\right)\,\mathrm dy.
\end{align*}
\end{itemize}\fi

\subsection{Historical background}

\subsubsection{Compressible Euler}

Local well-posedness of smooth solutions to the compressible Euler equations ($\nu=0$ in \eqref{CNS}) dates back to the classical works of Lax \cite{Lax1973}, Kato \cite{Kato1975}, and Majda \cite{Majda1984} on the theory of quasi-linear symmetric hyperbolic systems. They were able to establish the local Cauchy theory for the non-vacuum case, where $\inf \rho_0>0$, through a symmetrization of the compressible Euler equations. Another symmetrization was proposed later by Makino, Ukai, and Kawashima \cite{MUK1987}, where they considered the vacuum case; see also \cite{Chemin1990}.

Due to the nature of being a hyperbolic conservation system, a typical phenomenon observed in compressible Euler equations is the development of singularities. The first rigorous result was obtained by Lax \cite{Lax1964}, who proved the existence of 1-D shocks using Riemann invariants and the method of characteristics. The existence of finite-time singularities in 2-D and 3-D was established by the pioneering work of Sideris \cite{Sideris1985} via a virial-type argument. In two space dimensions, Alinhac \cite{Alinhac1999,Alinhac2001} was the first to give a detailed description of shock formation for a class of quasilinear wave equations. He proved that small initial data with compact support can lead to the formation of shock waves in finite time due to the intersection of characteristic lines. Based on \cite{Lebaud1994}, Yin \cite{Yin2004} proved shock formation and development in 3-D under spherical symmetry; see also \cite{Chris-Lisibach2016}.
In a series of monographs \cite{Christ2007,Christodoulou2019,Christ2014}, Christodoulou, and Christodoulou and Miao gave the first proof of shock formation and shock development in the absence of symmetry within the regime of irrotational fluids, both in relativistic and non-relativistic cases. These results were extended by Luk-Speck \cite{Luk2018,Luk2024} and Abbrescia-Speck \cite{Speck2022} to the setting involving non-zero vorticity and variable entropy. 
A new approach has been developed by Buckmaster, Shkoller, and Vicol by considering a perturbation of a self-similar Burgers-type shock. This approach was first applied in the 2-D setting with azimuthal symmetry \cite{Buck-Shko-Vicol1} (see also \cite{Buck-D-Shko-Vicol, Buck-Iyer}),  where they constructed shocks for the 2-D isentropic Euler equations and characterized the shock profile as an asymptotically self-similar, stable 1-D blowup profile. They then extended this approach to the 3-D setting, considering the presence of vorticity and entropy \cite{Buck-Shko-Vicol2, Buck-Shko-Vicol3}, and proved for the first time that the 3-D isentropic Euler equations generically form a stable point shock, even in the presence of vorticity . Their recent work \cite{SV} makes significant contributions to understanding the geometry of maximal development and shock formation. We refer readers to the survey \cite{BDSV-survey} for more references in this direction.

Nonetheless, shocks are not the only form of singularity arising from regular initial data. Another blow-up mechanism, known as implosion, where the density and velocity themselves (not only their derivatives) both become infinite at the blow-up time, has been constructed. The classical work by Guderley \cite{Guderley1942} and Sedov \cite{Sedov1959} gave a family of self-similar imploding singularities, although their solutions are not smooth. Motivated by Guderley's paper, in a series of breakthrough works \cite{MRRJ2,MRRJ3,MRRJ4}, Merle, Rapha\"el, Rodnianski and Szeftel rigorously proved the existence of a sequence of smooth imploding solutions to the compressible Euler equations. They then used these profiles to construct radially symmetric smooth blow-up solutions to both 3-D compressible Navier-Stokes equations and high-dimensional energy supercritical defocusing Schrödinger equations. Their result on the nonlinear Schrödinger equation solves a long-standing problem in the field, which was proposed by Bourgain \cite{Bourgain2000}. Their result on compressible Navier-Stokes is the first construction of blow-up solutions to \eqref{CNS} (with $\nu=1$), covering the range $1<\gamma<1+\frac{2}{\sqrt{3}}$ (where $d=3$) except for at most countably many points. However, the case $\g=5/3$, which corresponds to a monatomic gas and thus is physically relevant, is not included in their analysis due to a triple-point degeneration.

In \cite{BCLGS}, Buckmaster, Cao-Labora and G\'omez-Serrano constructed self-similar smooth imploding solutions to the 3-D compressible Euler equations for all $\g>1$ and proved the blow-up of the 3-D compressible Navier-Stokes equations for $\g=7/5$ (the law of a diatomic gas). The non-radial extensions of these results have been established in \cite{CLGSSS} for the compressible Euler and Navier-Stokes equations, and in \cite{CLGSSS-2024} for energy supercritical defocusing Schrödinger equations. These extensions are significant as they move beyond radial symmetry and address more general initial data.  See also \cite{Chen2024,CCSV2024} for further developments extending the work of \cite{MRRJ2} to include the presence of vorticity.

Motivated by the pioneering works on front compression mechanisms \cite{MRRJ2,MRRJ3,MRRJ4}, Shao, Wei and Zhang \cite{SWZ2024_2}  constructed complex-valued blow-up solutions of supercritical defocusing wave equations
\[-\pa_t^2u+\Delta_xu=|u|^{p-1}u\]
for $d=4, p\geq 29$ and $d\geq 5, p\geq 17$. This was achieved by converting the problem to the construction of self-similar smooth imploding solutions to the relativistic Euler equations, which was verified in \cite{SWZ2024_1}. This result has been extended recently by Buckmaster and Chen \cite{Buckmaster-Chen} to $(d,p)=(4,7)$, which is the endpoint of front-compression mechanism (for integer value $p$). They also rely on the construction of self-similar profiles for the relativistic Euler equations.%,  leveraging a similar front compression mechanism.

\subsubsection{Compressible Navier-Stokes}

The local theory of strong solutions to the compressible Navier-Stokes equations has been well developed by Serrin \cite{Serrin1959}, Nash \cite{Nash1962}, Itaya \cite{Itaya1976}, and Danchin \cite{Danchin2001}. In this series of works, they deal with regular initial data with the density bounded away from zero. For general initial density allowing vacuum, the local theory could be found in the papers by Cho, Choe, and Kim \cite{CCK2004,CK2003,CK2006}. Lions \cite{Lions1998} first proved the global existence of weak solutions to the 3-D compressible Navier-Stokes equations arising from large initial data for $\gamma \geq 9/5$. His result was extended to $\gamma > 3/2$ by Feireisl, Novotn\'y and Petzeltov\'a \cite{Feireisl2001}. For spherically
symmetric or axisymmetric initial data, Jiang and Zhang \cite{JZ2001,JZ2003} proved the global existence of weak solutions for the whole range $\gamma>1$.

 In a similar spirit to \cite{Sideris1985}, Xin \cite{Xin1998} proved the non-existence of global solutions in the class $C^1([0,+\infty); H^m(\R^d))$  for sufficiently large $m$ and compactly supported $\rho_0$, by employing a virial-type argument. Rozanova \cite{Rozanova2008} extended this result by replacing the compact-support condition with some rapid-decay conditions. However, it was proved by Li, Wang and Xin \cite{LWX2019} that if the initial density has compact support, then the compressible Navier-Stokes equations are NOT well-posed in the inhomogeneous Sobolev spaces, even \emph{locally}. To the best of our knowledge, the following problem has so far remained open.

\begin{problem}\label{Problem}
Provide a virial type argument to prove the non-existence of global smooth solutions to 2-D and 3-D isentropic compressible Navier-Stokes \eqref{CNS} in suitable functional spaces where the local Cauchy theory is available.
\end{problem}

We also remark that, for the full compressible Navier-Stokes equations involving temperature,  Xin and Yan \cite{XY2013} proved the blow-up of classical solutions if the initial density has compact support (or more generally, has an isolated mass group), by utilizing the extra structure provided by the equation for temperature, and a virial type argument. It is important to note that the blow-up results in \cite{XY2013} are independent of the functional spaces where the solutions may lie in and whether the initial data are large or small. \smallskip

In the seminal paper \cite{Nash1958}, Nash attempted to study the continuity of solutions to the equations governing compressible viscous fluids and proposed what has since become known as the {\bf conditional regularity conjecture}:
\smallskip

{\it  Probably one should first try to prove a conditional existence and uniqueness
theorem for flow equations. This should give existence, smoothness, and unique
continuation (in time) of flows, conditional on the non-appearance of certain
gross types of singularity, such as infinities of temperature or density.}
\smallskip

In   \cite{SWZ2011}, Sun, Wang and Zhang (the fourth author) confirmed Nash's conjecture and showed that  {\it  a smooth solution $(\rho, \mathbf u)$ to 3-D compressible Navier-Stokes equations blows up in finite time $T^*$ if and only if
\[\limsup_{T\uparrow T^*}\left\|\rho(t)\right\|_{L^\infty(0, T; L^\infty(\R^3))}=+\infty.\]

}

\noindent  To the best of our knowledge, the first blow-up result for the 3-D isentropic compressible Navier-Stokes equations was obtained by Merle, Rapha\"el, Rodnianski and Szeftel \cite{MRRJ2,MRRJ3}, as we mentioned in the previous sub-subsection. The main result in \cite{MRRJ3} states as follows. \smallskip

\noindent{\it  There exists a (possibly empty) exceptional countable sequence $\{\g_n\}_{n\in\N}$ whose accumulation points can only belong to $\{1, 5/3, +\infty\}$ such that for all $\g\in(1, 1+2/\sqrt3)\setminus(\{5/3\}\cup\{\g_n\}_{n\in\N})$, there exists a discrete sequence of blow-up speeds $\{r_k>1\}_{k\in\N}$ such that for each $k$, there is a finite co-dimensional manifold of smooth spherically symmetric initial data $(\rho_0, \mathbf u_0)\in H_\text{rad}^\infty(\R^3)$, for which the corresponding solution $(\rho, \mathbf u)$ to 3-D isentropic compressible Navier-Stokes equations blows up in finite time $0<T^*<+\infty$, with
\[\|\mathbf u(t,\cdot)\|_{L^\infty}=\frac{c_{\mathbf u}\left(1+o_{t\uparrow T^*}(1)\right)}{(T^*-t)^{(r_k-1)/r_k}},\quad\left\|\rho(t,\cdot)\right\|_{L^\infty}=\frac{c_\rho\left(1+o_{t\uparrow T^*}(1)\right)}{(T^*-t)^{\ell(r_k-1)/r_k}},\]
where $\ell=2/(\g-1)$. }

Here, the blow-up mechanism (i.e., implosion) is compatible  with Nash's conditional regularity conjecture.
The discrete sequence $\{\g_n\}$ of possibly non-admissible equations of state is related to the existence of smooth self-similar imploding solutions 
to the compressible Euler equations, as shown in \cite{MRRJ2}. The sequence $\{\g_n\}$ is given by the zeroes of a specific but complicated series denoted by 
$S_\infty(3,\g)$, which is analytic with respect to $\g$. The authors in \cite{MRRJ2} checked numerically that $S_\infty(3,\cdot)$ is not 
identically zero. Nonetheless, given a particular $\g\in(1, 1+2/\sqrt3)$ (let's say $\g=7/5$), it is challenging to determine that whether 
this specific $\g$ belongs to the sequence $\{\g_n\}$ or not. It was proved in \cite{BCLGS} that
$\g=7/5\not\in\{\g_n\}$, which corresponds to a diatomic gas.
A new proof of Merle-Rapha\"el-Rodnianski-Szeftel's theorem is also provided in \cite{BCLGS}. Recently, the extension to the non-radial case has been established in \cite{CLGSSS} for $\gamma$ values the same as those in \cite{MRRJ3}. The authors in \cite{CLGSSS} proved an abstract theorem (\cite[Theorem 1.2]{CLGSSS}) which converts the blow-up of the 3-D compressible Navier-Stokes equations to the construction of self-similar profiles solving the compressible Euler equations and satisfying some repulsivity properties. This result will be used later in the present paper.

Unfortunately, these results fail for $\g=5/3$ (the law of a monatomic gas), in which case a triple-point degeneracy appears and the series $S_\infty(3,\g)$ is not defined for $\g=5/3$. Due to both the mathematical difficulties caused by the degeneracy and the physical relevance, our current paper focuses primarily on this case. Roughly speaking, we prove that for $\g=5/3$, the compressible Navier-Stokes equations have infinitely many finite-time blow-up solutions. This is achieved by demonstrating that for $\g=5/3$ there exists a variation of $S_\infty$, denoted by $\cS_\infty$ (see \eqref{Eq.S_infty}), such that $\cS_\infty>0$ implies the existence of self-similar imploding profiles to the compressible Euler equations.

\subsubsection{Self-similar solutions}
It is noticed that all the imploding solutions discussed above are of the self-similar type. Now, we review some recent results on the construction of self-similar singularities for related fluid PDEs.

First, there have been some significant developments in the study of blow-up phenomena for the 3-D incompressible Euler equations. These advancements have been achieved by considering low regularity solutions \cite{Chen2021,Elgindi2021-1,Elgindi2021-2}, or by imposing boundary conditions \cite{Chen2021, Chen2022-1, Chen2022-2,HL2014}. For further details, see also the review article \cite{Drivas-Elgindi}. Similar results have been obtained for other models related to the 3-D incompressible Euler equations; see \cite{Chen2023,ChenHouHuang2021,HQWW2023,HQWW2024,HTW2023} and the references therein.

Recently, significant progress has been made in the mathematical theory of gravitational collapse in the field of astrophysics, which refers to the process of star implosion. The corresponding PDE models are the compressible Euler-Poisson equations and the Einstein-Euler equations. For the construction of self-similar radially symmetric imploding solutions, see the series of remarkable works \cite{AHS2023,GHJ2021-1,GHJ2021-2,GHJ2023,GHJS2022}, where the key ingredient is solving some non-autonomous ODEs that possess sonic points.

\subsection{Main results}
Before stating our theorems, we set up the problem following \cite{CLGSSS}. Let $\ell=\frac2{\g-1}>0$ and $T>0$. We introduce the self-similar change of variables\footnote{For simplicity, we let the domain to be $\R^d$. The $\T^d$ case is similar.}
\begin{align*}
	\mathbf u(t,y)&=\frac{(T-t)^{\frac1r-1}}{r}\mathbf U\left(-\frac{\ln(T-t)}{r}, \frac y{(T-t)^{\frac1r}}\right),\\
	 \rho(t,y)^{\frac1\ell}&=\frac{(T-t)^{\frac1r-1}}{\ell\cdot r}S\left(-\frac{\ln(T-t)}{r}, \frac y{(T-t)^{\frac1r}}\right),
\end{align*}
where $r>1$ is the self-similar parameter, $t\in[0, T)$ and $y\in\R^d$. Under the self-similar variables $(s, z)$ defined by
\begin{align*}
	s:=-\frac{\ln(T-t)}{r},\quad z:=\frac y{(T-t)^{\frac1r}}=\e^sy,
\end{align*}
the system \eqref{CNS} is reduced to
\begin{equation}\label{NS_sseq}
	\begin{cases}
		\pa_s\mathbf U=-(r-1)\mathbf U-(z+\mathbf U)\cdot\nabla\mathbf U-\frac1\ell S\nabla S+\nu C_{\text{dis}}\e^{-\delta_\text{dis}s}\frac{\Delta \mathbf U}{S^\ell},\\
		\pa_sS=-(r-1)S-(z+\mathbf U)\cdot\nabla S-\frac1\ell S\dive\mathbf U,
	\end{cases}
\end{equation}
where
\begin{equation*}
	C_{\text{dis}}:=r^{1+\ell}\ell^\ell,\quad \delta_\text{dis}:=\ell(r-1)+r-2.
\end{equation*}

We restrict the parameter $r$ to the range
\begin{equation}\label{Para_rangeEuler}
	1<r<\begin{cases}
		1+\frac{d-1}{\left(1+\sqrt{\frac{2}{\g-1}}\right)^2},&\text{for }1<\g<1+\frac2d,\\
		\frac{d\g+2-d}{2+\sqrt d(\g-1)}, &\text{for }\g\geq1+\frac2d.
	\end{cases}
\end{equation}

In order to use the profile of Euler equations to construct finite-time blow-up solutions of Navier-Stokes equations, we need an extra condition on the parameters
\begin{equation}\label{Para_rangeNS}
	\delta_\text{dis}=\ell(r-1)+r-2>0.
\end{equation}

From \cite{BCLGS}, we know that there exist smooth and radially symmetric profiles $(\mathbf U_E, S_E)$ solving \eqref{NS_sseq} for $\nu=0$:
\begin{equation}\label{profile_Euler}
	\begin{cases}
		(r-1)\mathbf U_E+(z+\mathbf U_E)\cdot\nabla \mathbf U_E+\frac1\ell S_E\nabla S_E=0,\\ (r-1)S_E+(z+\mathbf U_E)\cdot\nabla S_E+\frac1\ell S_E\dive\mathbf U_E=0,
	\end{cases}
\end{equation}
with some $r$ in the range \eqref{Para_rangeEuler}, for $d=3$ and all $\g>1$. In the pioneering work \cite{MRRJ2}, Merle, Rapha\"el, Rodnianski and Szeftel proved that for each $d\in\Z_{\geq 2}$, there exist a set $\mathcal O_d\subset (0,+\infty)\setminus \{d\}$ and a function $S_\infty=S_\infty(d,\ell):\Z_{\geq 2}\times\mathcal O_d\to\R$ such that for any $\ell\in\mathcal O_d$ obeying $S_\infty(d,\ell)\neq0$, there is a sequence of smooth and radially symmetric profiles $(\mathbf U_E, S_E)$ solving \eqref{profile_Euler}, with the corresponding parameter $r$ satisfying both \eqref{Para_rangeEuler} and \eqref{Para_rangeNS}. It is shown in \cite{MRRJ2} that for $d\in\{2,3\}$, there holds $\mathcal O_d=(0,+\infty)\setminus \{d\}$, and the authors proved with computational assistance that $S_\infty(d,\ell)\neq 0$ for all $\ell>0$ but possibly countably many numbers $\{\ell_k\}$. Nonetheless, given a pair of explicit $d$ and $\g>1$, it is challenging to check that whether $\ell\in\{\ell_k\}$ or not. From \cite{BCLGS}, we know the existence of smooth and radially symmetric profiles $(\mathbf U_E, S_E)$ solving \eqref{profile_Euler} with parameter $r$ satisfying both \eqref{Para_rangeEuler} and \eqref{Para_rangeNS} for one explicit pair $(d=3,\g=7/5)$.

It is noted that the methods in \cite{BCLGS,MRRJ2} fail for the case $d=\ell$, which is considered degenerate according to \cite[page 581]{MRRJ2}. Specifically, taking $d=\ell=3$ yields $\g=5/3$, corresponding to the law of monatomic gases. Due to its physical significance and the mathematical challenges posed by degeneracy, it is essential to investigate this degenerate case further, which serves as our primary motivation. From now on, we will focus primarily on the case  $d=3$ and $\g=5/3$. Most of our analysis remains valid for the general case $\ell=d$ (or equivalently $\g=1+2/d$) with $d\in\Z_{\geq 2}$, for the compressible Euler equations.

For radially symmetric $(\mathbf U_E, S_E)$ solving \eqref{profile_Euler}, if we write
\begin{equation}\label{Eq.radial_transform}
	Z:=|z|,\quad \mathbf U_E(z)=U_E(Z)\frac{z}{Z},\quad S_E(z)=S_E(Z),
\end{equation}
then $(U_E, S_E): (0,+\infty)\to\R\times(0,+\infty)$ solves the following system
\begin{equation}\label{Eq.radial_eq}
	\begin{cases}
		(r-1)U_E+(Z+U_E)\pa_ZU_E+\frac1\ell S_E\pa_ZS_E=0,\\
		(r-1)S_E+(Z+U_E)\pa_ZS_E+\frac1\ell\left(\pa_ZU_E+\frac{d-1}{Z}U_E\right)S_E=0.
	\end{cases}
\end{equation}

Our main result is stated as follows.

\begin{theorem}\label{Mainthm}
	Let $d=3$ and $\g=5/3$. There exists a discrete sequence $\{r_n\}$ satisfying
	\[r_n<1+\frac2{(1+\sqrt3)^2}=3-\sqrt3, \quad \left|r_n-(3-\sqrt3)\right|\ll1,\quad \lim_{n\to\infty}r_n=3-\sqrt3\]
	such that \eqref{profile_Euler} with $r=r_n$ admits a global $C^\infty$ radially symmetric solution $(\mathbf U_E, S_E)$ possessing the following properties:
	\begin{align}
		S_E&\gtrsim_r\langle Z\rangle^{-r+1},\label{Coercive1}\\
		\left|\nabla^j U_E\right|+\left|\nabla^j S_E\right|&\lesssim_{r,j} \langle Z\rangle^{-(r-1)-j},\quad\forall\ j\in\Z_{\geq 0},\label{Coercive2}\\
		1+\pa_Z U_E-\frac1\ell\left|\pa_ZS_E\right|&>\eta,\label{Coercive3}\\
		1+\frac{U_E}Z-\frac1\ell\left|\pa_ZS_E\right|&>\eta,\label{Coercive4}
	\end{align}
	for some constant $\eta>0$ depending on $r$, recalling \eqref{Eq.radial_transform}.
\end{theorem}

To prove blow-up of the compressible Navier-Stokes equations, we use the following abstract result from \cite{CLGSSS}.

\begin{theorem}[{\cite[Theorem 1.3]{CLGSSS}}]\label{Thm.stability}
	Let $\nu=1$. Let $(\mathbf U_E, S_E)$ be a smooth self-similar profile solving \eqref{profile_Euler} for some $r$ in the range \eqref{Para_rangeEuler}, \eqref{Para_rangeNS} and satisfying \eqref{Coercive1}$-$\eqref{Coercive4}. Let $T>0$ be sufficiently small and $c>0$ be sufficiently small.
	
	Then there exists $C^\infty$ (non-radially symmetric) initial data $(\mathbf u_0,\rho_0)$ with $\rho_0>c$, for which the Navier-Stokes equations \eqref{CNS} on $\R^3$ (or $\T^3$, in which case $\rho_0>0$) blow up at time $T$ in a self-similar manner. More concretely, for any fixed $z\in\R^3$, we have
	\begin{align}
		\lim_{t\uparrow T}r(T-t)^{1-\frac1r}\mathbf u\left(t, (T-t)^{1/r}z\right)&=\mathbf U_E(z),\label{Asymptotic1}\\
		\lim_{t\uparrow T}\left(\ell r(T-t)^{1-\frac1r}\right)^\ell\rho\left(t, (T-t)^{1/r}z\right)&=S_E(z)^\ell.\label{Asymptotic2}
	\end{align}
	Moreover, there exists a finite co-dimensional set of initial data satisfying the above conclusions.
\end{theorem}

In Theorem \ref{Mainthm}, the parameters $r_n$ satisfy \eqref{Para_rangeEuler} and
\[\ell(r_n-1)+r_n-2=4r_n-5\to 7-4\sqrt3>0,\quad\text{as}\quad n\to\infty.\]
Hence $r_n$ lies in the range \eqref{Para_rangeNS} for $n$ sufficiently large. Moreover, the smooth self-similar profiles constructed in Theorem \ref{Mainthm} satisfy the assumptions in Theorem \ref{Thm.stability}. Combining Theorem \ref{Mainthm} and Theorem \ref{Thm.stability}, we establish

\begin{corollary}\label{cor:NS}
	Let $d=3,\g=5/3$ and $\nu=1$. There exists a discrete sequence $\{r_n\}$ satisfying
	\[r_n<3-\sqrt3, \quad \left|r_n-(3-\sqrt3)\right|\ll1,\quad \lim_{n\to\infty}r_n=3-\sqrt3\]
	such that for each $r\in\{r_n\}$ the following facts hold for any sufficiently small $T>0$:
	\begin{itemize}
	\item There exists a smooth non-radially symmetric initial data $(\rho_0, \mathbf u_0)$ with $\inf\rho_0>0$ (or $\rho_0>0$ for $\T^3$), for which the corresponding solution to  Navier-Stokes equations \eqref{CNS} blows up at time $T$ with the asymptotic behavior \eqref{Asymptotic1} and \eqref{Asymptotic2}, where $(\mathbf U_E, S_E)$ is the self-similar profile corresponding to this parameter $r$ obtained from Theorem \ref{Mainthm};
	\item There exists a finite co-dimensional set of initial data satisfying the above item.
	\end{itemize}
\end{corollary}

Several remarks are in order.

\begin{enumerate}[1.]
	\item In the proof of Theorem \ref{Mainthm}, one of the key points is to verify the non-degeneracy of a number $\cS_\infty$, which is the limit of a sequence defined by an explicit but rather complicated recurrence relation without any parameter. We verify that $\cS_\infty>0$ with the help of a computer. This is the only instance in this paper where we use computational assistance.
	
	\item It is reasonable to believe that the proof of Theorem \ref{Mainthm} can be extended to the general case $d=\ell$ for all $d\in\Z_{\geq 2}$. In this general case, we would also need to assume the non-degeneracy of a limit similar to $\cS_\infty$, the validity of which can similarly be verified with computational assistance. Nevertheless, we note that Theorem \ref{Thm.stability} requires $d=3$, as stated in \cite{CLGSSS}. Due to this limitation, at this stage, we can prove the blow-up of the compressible Navier-Stokes equations with $d=\ell$ only for $d=3$.
	
	\item The abstract Theorem  \ref{Thm.stability} should hold for the general diffusion $-\mu \Delta \mathbf u-(\lambda+\mu)\nabla\text{div}\mathbf u$ with $\mu>0$ and $2\mu+3\lambda>0$. Consequently, Corollary \ref{cor:NS} should also hold.
	
	\item Due to the scaling invariance, Theorem \ref{Mainthm} holds for any 3-D torus of any size. Moreover, thanks to \cite[Remark 1.6]{CLGSSS}, given the local behavior of the singularity, we can also construct solutions that blow up at $m$ points.
	
	\item It is noticed that for $d=2$, the ranges \eqref{Para_rangeEuler} and \eqref{Para_rangeNS} have an empty intersection. This indicates that the front compression mechanism fails in the construction of blow-up solutions to the 2-D isentropic compressible Navier-Stokes equations. Consequently, it remains unknown whether smooth solutions to the 2-D isentropic compressible Navier-Stokes equations blow up or not.
	\end{enumerate}

Let's conclude the introduction by notations. 
\begin{itemize}
	\item We denote $A\lesssim B$ (or $B\gtrsim A$) if $A\leq CB$ for some absolute constant $C>0$, and denote $A\asymp B$ if $A\lesssim B$ and $B\lesssim A$.
	\item We denote $A\lesssim_pB$ (or $B\gtrsim A$) if $A\leq C(p)B$ for some constant $C(p)>0$ depending on $p$, and define $A\asymp_pB$ similarly.
	\item Given two functions $f=f(x)$ and $g=g(x)\neq 0$, we say ``$f\sim g$ as $x\to x_0$" if $\lim_{x\to x_0}f(x)/g(x)=1$. We denote $f=\cO(g)$ if $f/g$ is bounded.
	\item For $r\in\R$, we denote $\Z_{\geq r}:=\Z\cap[r,+\infty)$. We let $\N:=\Z_{\geq 0}$, $\Z_+:=\Z_{\geq 1}$, $\sqrt\N:=\{\sqrt n: n\in\N\}$. We also denote $2\Z+1$ the set consisting of all odd integers.
	\item For $R\in\R$, we denote $\lceil R \rceil=\inf(\Z\cap[R,+\infty))$.
\end{itemize}

\section{The autonomous ODE system and sketch of the proof}
In this section, we first convert \eqref{profile_Euler} into an autonomous ODE system. We then translate Theorem \ref{Mainthm} into two results: Theorem \ref{Thm.ODE} and Proposition \ref{Prop.coercive}, which address the existence and repulsivity of global smooth solutions to the autonomous ODE system, respectively.

\subsection{The autonomous ODE system and phase portraits}\label{Subsec.ODE_portrait}
We seek smooth, radially symmetric solutions $(\mathbf U_E, S_E)$ to \eqref{profile_Euler}, which is equivalent to solving \eqref{Eq.radial_eq}. Following \cite{MRRJ2}, we introduce the Emden transform
\begin{equation}\label{Eq.Emden_transform}
    S_E(Z)=\ell Z\sigma(x),\quad U_E(Z)=-Zw(x),\quad Z=\e^x.
\end{equation}
Then \eqref{Eq.radial_eq} is equivalent to the following autonomous ODE system
\begin{equation}\label{autonomousODE}
	\begin{cases}
		(w-1)\frac{\mathrm dw}{\mathrm dx}+\ell\sigma\frac{\mathrm d\sigma}{\mathrm dx}+w^2-rw+\ell\sigma^2=0,\\
		\frac{\sigma}{\ell}\frac{\mathrm dw}{\mathrm dx}+(w-1)\frac{\mathrm d\sigma}{\mathrm dx}+\sigma\left[w\left(1+\frac d\ell\right)-r\right]=0,
	\end{cases}\Longleftrightarrow\begin{cases}
	\Delta\frac{\mathrm dw}{\mathrm dx}=-\Delta_1,\\
	\Delta\frac{\mathrm d\sigma}{\mathrm dx}=-\Delta_2,
	\end{cases}
\end{equation}
where
\begin{align}
	\Delta&=\Delta(\sigma,w)=(w-1)^2-\sigma^2,\label{Eq.Delta}\\
	\Delta_1&=\Delta_1(\sigma,w)=w(w-1)(w-r)-\big(dw-\ell(r-1)\big)\sigma^2,\label{Eq.Delta_1}\\
	\Delta_2&=\Delta_2(\sigma,w)=\frac{\sigma}{\ell}\left[(\ell+d-1)w^2-(\ell+d+\ell r-r)w+\ell r-\ell\sigma^2\right].\label{Eq.Delta_2}
\end{align}

Then Theorem \ref{Mainthm} is a consequence of Theorem \ref{Thm.ODE} and Proposition \ref{Prop.coercive} as below.

\begin{theorem}[Existence]\label{Thm.ODE}
	Let $d=\ell=3$. There exists a discrete sequence $\{r_n\}$ satisfying
	\begin{equation}\label{Eq.r_n}
	r_n<3-\sqrt3, \quad \left|r_n-(3-\sqrt3)\right|\ll1,\quad \lim_{n\to\infty}r_n=3-\sqrt3
	\end{equation}
	such that \eqref{autonomousODE} with $r=r_n$ admits a global $C^\infty$ solution $(\sigma,w)=(\sigma,w)(x)$ with the following properties
	\begin{itemize}
		\item $\Delta(\sigma(x), w(x))<0$, $\Delta_1(\sigma(x), w(x))>0$ and $\Delta_2(\sigma(x), w(x))<0$ for all $x<0$;
		\item $(\sigma(0), w(0))=(1-w_-, w_-)$;
		\item The map $x\mapsto\Delta_1(\sigma(x), w(x))$ has only two solutions $x=0$ and $x=x_A>0$;
		\item $\Delta(\sigma(x), w(x))>0$ and $\Delta_2(\sigma(x), w(x))>0$ for $x>0$;
		\item $\Delta_1(\sigma(x), w(x))<0$ for $x\in(0, x_A)$ and $\Delta_1(\sigma(x), w(x))>0$ for $x>x_A$;
		\item $w(x)>a(1+a)\sigma(x)^2$ for all $x>0$.
	\end{itemize}
	Here $w_-=w_-(r)=(r-\sqrt{r^2-6r+6})/2$ and $a=w_-/(1-w_-)$. See Figure \ref{Fig.Phase_Portrait_w_sigma}. Moreover, $(\mathbf U_E, S_E)$ defined by \eqref{Eq.radial_transform} and \eqref{Eq.Emden_transform} is a smooth self-similar solution to \eqref{profile_Euler} on $\R^3$.
\end{theorem}

\begin{proposition}[Repulsivity]\label{Prop.coercive}
	Let $d=\ell=3$ and $\{r_n\}$ be given by Theorem \ref{Thm.ODE}. For each $r\in\{r_n\}$, let $(\sigma,w)$ be the solution to \eqref{autonomousODE} given by Theorem \ref{Thm.ODE}. Then we have	
	\begin{align}
		\sigma(x)&\gtrsim_r \min\left\{\e^{-rx},\e^{-x}\right\}  \label{repuls 1},\\
		\left|w^{(j)}(x)\right|+\left|\sigma^{(j)}(x)\right|&\lesssim_{r,j} \min\left\{\e^{-rx},\e^{-x}\right\},\quad\forall\
         j\in\Z_{\geq 0},\label{repuls 2}\\
		1-\left(w(x)+w'(x)\right)-\left|\sigma(x)+\sigma'(x)\right|&>\eta\label{repuls 3},\\
		1-w(x)-\left|\sigma(x)+\sigma'(x)\right|&>\eta\label{repuls 4},
		\end{align}
	for some constant $\eta>0$ depending on $r$. Here the derivatives are taken with respect to $x\in\R$.
\end{proposition}

\begin{figure}
	\centering
	\includegraphics[width=1\textwidth]{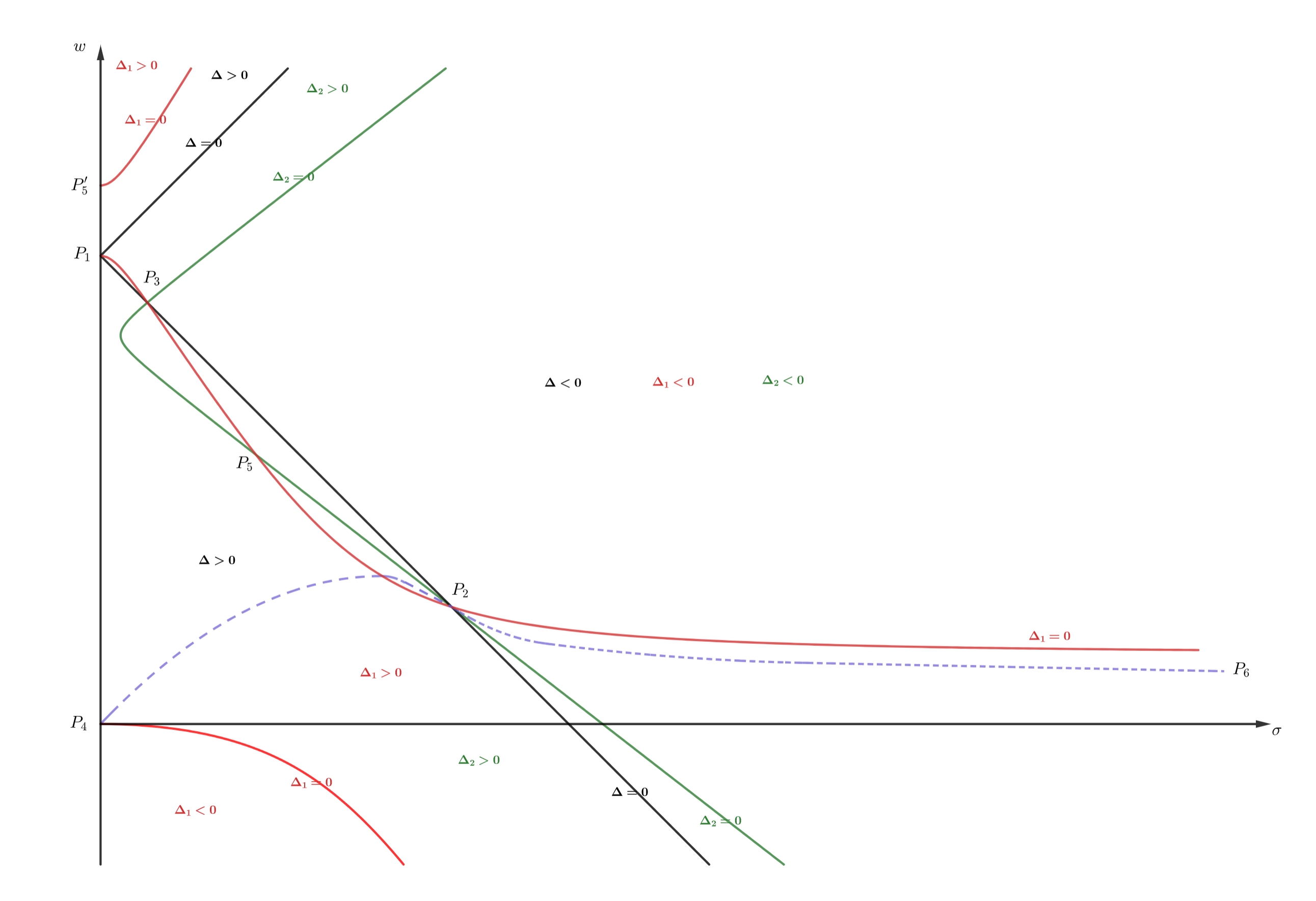}
	\caption{Phase portrait for the $\sigma-w$ system \eqref{autonomousODE}: Dashed curve is the trajectory of the solution constructed in Theorem \ref{Thm.ODE}.}
	\label{Fig.Phase_Portrait_w_sigma}
\end{figure}

The system \eqref{autonomousODE} can be analyzed through the phase portrait in the $\sigma-w$ plane, see Figure \ref{Fig.Phase_Portrait_w_sigma}. The solution curves corresponding to $\Delta=0$, $\Delta_1=0$ and $\Delta_2=0$, along with the intersection points of these curves, play a crucial role in our analysis of the system \eqref{autonomousODE}.

To simplify calculations, we take $d=\ell=3$. Recall that we require the parameter $r$ to lie in the range \eqref{Para_rangeEuler}.
The solutions to $\Delta=0$ are $w-1=\pm\sigma$, which are called \emph{sonic lines}. According to \cite[Lemma 2.3]{MRRJ2}, the roots of $\Delta_1=0$ are given by three curves $w_1(\sigma)<w_2(\sigma)<w_3(\sigma)$ with the following properties:
\begin{align}
	w_1(\sigma)\leq0<r-1<w_2(\sigma)\leq 1<r\leq w_3(\sigma),\quad&\forall\ \sigma\geq 0;\nonumber\\
	w_1'(\sigma)<0,\quad w_2'(\sigma)<0,\quad w_3'(\sigma)>0,\quad&\forall\ \sigma>0;\label{Eq.w_j_monotonicity}\\
	w_1(\sigma)=-\frac{3(r-1)}{r}\sigma^2+\mathcal O(\sigma^3),\ w_2(\sigma)=1+\cO(\sigma^2),\ w_3(\sigma)=r+\cO(\sigma^2),\quad&\text{as }\sigma\downarrow0;\nonumber\\
	w_1(\sigma)=-\sqrt3\sigma+\mathcal O(1),\ w_2(\sigma)=r-1+\cO(\sigma^{-2}),\ w_3(\sigma)=\sqrt3\sigma+\cO(1),\quad&\text{as }\sigma\uparrow+\infty.\nonumber
\end{align}
The solutions to $\Delta_2=0$ are given by $\sigma=0$ and another two curves $w_2^\pm=w_2^\pm(\sigma)$, where
\[w_2^\pm(\sigma)=\frac{r+3\pm\sqrt{r^2-9r+9+15\sigma^2}}{5},\quad \forall\ \sigma\geq \sigma_2^{(0)},\]
and
\[\sigma_2^{(0)}:=\begin{cases}
	\sqrt{\frac{-r^2+9r-9}{15}}&\text{if } \frac{9-3\sqrt{5}}{2}<r<3-\sqrt3,\\
	0&\text{if }1<r\leq \frac{9-3\sqrt{5}}{2}.
\end{cases}\]

By \cite[Lemma 2.4]{MRRJ2}, the solutions to $\Delta_1=\Delta_2=0$ are
\begin{align}
	&P_1=(0,1),\quad P_2=(\sigma_2=1-w_-, w_-),\quad P_3=(1-w_+, w_+),\label{Eq.P_1_2}\\
	&P_4=(0,0),\quad P_5=\left(\frac{\sqrt3 r}{6}, \frac r2\right),\quad P_5'=(0,r),\nonumber
\end{align}
where
\begin{equation}\label{Eq.w+-}
    w_\pm=\frac{r\pm\sqrt{r^2-6r+6}}{2}\in(r-1,1).
\end{equation}
Here we also note that $r^2-6r+6>0$ due to $r<3-\sqrt3$. It is easy to see that
\begin{itemize}
	\item $P_5$ lies between $P_2$ and $P_3$;
	\item $P_5$ and $P_2$ lies in the curve $\sigma\mapsto(\sigma,w_2^-(\sigma))$;\footnote{It seems that $P_3$ lies in the curve $\sigma\mapsto(\sigma,w_2^+(\sigma))$ from our Figure \ref{Fig.Phase_Portrait_w_sigma}. Unfortunately, this is not true for $r$ being sufficiently close to $3-\sqrt3$. Indeed, there holds
	\[w_+=\frac{r+\sqrt{r^2-6r+6}}{2}<\frac{r+3}{5}=\min w_2^+\quad\text{for}\quad \frac{57-5\sqrt{57}}{16}<r<3-\sqrt3.\]
	Therefore, if we let $r$ be sufficiently close to $3-\sqrt3$, then $P_3$ will lie in the curve $\sigma\mapsto(\sigma,w_2^-(\sigma))$, 
not exactly as shown in our Figure \ref{Fig.Phase_Portrait_w_sigma}. Nevertheless, this inconsistency does not affect our proof.}

	\item $P_3$ and $P_2$ lies in the sonic line $w=1-\sigma$ and
	\[\sigma+w_2(\sigma)\begin{cases}
		>1&\text{for}\quad 0<\sigma<1-w_+=\sigma(P_3),\\
		<1&\text{for}\quad \sigma(P_3)=1-w_+<\sigma<1-w_-=\sigma(P_2)=\sigma_2,\\
		>1&\text{for}\quad \sigma>1-w_-=\sigma(P_2)=\sigma_2.
	\end{cases}\]
	\item As $r\uparrow 3-\sqrt3$, we have
	\begin{equation}\label{Eq.degeneracy}
		w_\pm(r)\to \frac{3-\sqrt3}{2},\quad w(P_5(r))\to \frac{3-\sqrt3}{2},
	\end{equation}
	in which case the phase portrait has a triple point degeneracy.
\end{itemize}

We note that $P_2$ solves $\Delta=\Delta_1=\Delta_2=0$, where the system \eqref{autonomousODE} degenerates and the classical ODE theory fails at $P_2$. In the literature, $P_2$ is commonly referred to as a \emph{sonic point}.

\subsection{Sketch of the proof: existence}\label{Subsec.sketch_existence}

In this subsection, we outline  the proof of Theorem \ref{Thm.ODE}. Due to the appearance of the sonic point $P_2$,  the proof is rather complex and intricate.\smallskip

\textbf{Step 1.} \underline{Near the origin $Z=0$: $P_6-P_2$ solution curve.} Eliminating the implicit variable $x$ in \eqref{autonomousODE} gives the ODE of $w=w(\sigma)$:
\begin{equation}\label{ODE_w_sigma}
	\frac{\mathrm dw}{\mathrm d\sigma}=\frac{\Delta_1(\sigma,w)}{\Delta_2(\sigma,w)}.
\end{equation}
Following \cite{MRRJ2}, the first step is to construct a local solution to \eqref{ODE_w_sigma} emerging from $P_6(\sigma=+\infty, w=r-1)$ and prove that it can be extended until reaching $P_2$, see Lemma 3.1 in \cite{MRRJ2}. We remark that the proof of \cite[Lemma 3.1]{MRRJ2} works well in our setting, 
even if $\ell= d$. Moreover, \cite[Lemma 3.1]{MRRJ2} also proves the smoothness of $(\mathbf U_E, S_E)$ (defined by \eqref{Eq.radial_transform} and \eqref{Eq.Emden_transform}) at the spatial origin.
Now we obtain a smooth function $w=w_F(\sigma;r)$ defined on $(\sigma,r)\in(\sigma_2, +\infty)\times(1, 3-\sqrt3)$, such that for each $r\in(1, 3-\sqrt3)$, $w_F(\cdot;r)$ solves \eqref{ODE_w_sigma} and
\begin{align*}
	w_F(\sigma;r)=r-1+\frac{(r-1)(2-r)}{5}&\frac1{\sigma^2}+\mathcal O\left(\frac1{\sigma^4}\right)\quad\text{as}\quad\sigma\uparrow+\infty,\\
	\max\left\{r-1, w_2^-(\sigma)\right\}<w_F(\sigma;r)&<w_2(\sigma)\quad \forall\ \sigma_2<\sigma<+\infty,\\
	\lim_{\sigma\downarrow\sigma_2}w_F(\sigma;r&)=w_-=w(P_2).
\end{align*}
This solution depends smoothly on the parameter $r\in (1, 3-\sqrt3)$.

\medskip

{\textbf{Step 2.} \underline{Near the sonic point: re-normalization and local smooth solutions.} \, One of the main difficulties in proving Theorem \ref{Thm.ODE} lies in extending $w_F(\cdot; r)$ to pass through the sonic point $P_2$ smoothly. However, this is not possible for some parameter $r$. To find the desired parameter, the strategy is to construct a local smooth solution $w_L(\sigma; r)$ near $P_2$, defined by a power series $w_L(\sigma;r)=\sum_{n=0}^\infty \widetilde{w}_n(r)(\sigma-\sigma_2)^n$, and then show that $w_L(\sigma; r_n)=w_F(\sigma;r_n)$ for some well-chosen parameters $\{r_n\}$. In this strategy, understanding the local solution $w_L$ in a quantitative manner is crucial, hence obtaining the sharp estimates of the coefficients $\{\widetilde{w}_n\}$ is essential.

Unfortunately, it turns out that the sequence $\{\widetilde{w}_n\}$ obeys a rather complicated recurrence relation. To simplify the analysis, we introduce the following renormalization
\begin{equation}\label{Eq.renormalization}
	a:=\frac{w_-}{1-w_-},\quad\alpha:=\frac{a(a+1)}{a+3},\quad u=(1+a)^2\sigma^2,\quad \tau=-(1+a)(w-w_-),
\end{equation}
then \eqref{ODE_w_sigma} is mapped to
\begin{equation}\label{ODE_u_tau}
	\Delta_\tau\mathrm du-\Delta_u\mathrm d\tau=0,
\end{equation}
where
\begin{align}
	\Delta_u(\tau,u):&=-2u^2+2u-\frac43(\al-4)\tau u+\frac{10}3\tau^2u,\label{Eq.Delta_u}\\
	\Delta_\tau(\tau,u):&=3(\al-\tau)u-3\al -(4\al-1)\tau-(\al-2)\tau^2+\tau^3.\label{Eq.Delta_tau}
\end{align}

\begin{Remark}
The re-normalization \eqref{Eq.renormalization}
reduces the power for $\sigma$ in the equation \eqref{ODE_w_sigma}. The introduction of the parameter $a$ helps us avoid many computations involving radicals. This significantly simplifies the computations. We emphasize that part of this renormalization procedure is not constrained by the condition $d=\ell=3$. For parameters belonging to other ranges, a similar renormalization is available, which could potentially provide a simpler proof of Merle-Raphaël-Rodnianski-Szeftel's theorem \cite{MRRJ2}. 
\end{Remark}

Similarly with the $\sigma-w$ ODE \eqref{ODE_w_sigma}, solutions of $\Delta_u=0$ and $\Delta_\tau=0$ play a crucial role in our analysis of the $\tau-u$ ODE. It is a direct computation to obtain the solution of $\Delta_u(\tau, u)=0$:
\begin{equation}\label{Eq.u_g}
	u=u_\text g(\tau):=1-\frac23(\al-4)\tau+\frac53\tau^2,
\end{equation}
see the green curve in Figure \ref{Fig.Phase_Portrait_u_tau}; and the solution of $\Delta_\tau(\tau, u)=0$:
\begin{equation}\label{Eq.u_b}
	u=u_\text b(\tau):=\frac{-\tau^3+(\al-2)\tau^2+(4\al-1)\tau+3\al}{3(\al-\tau)},
\end{equation}
see the blue curve in Figure \ref{Fig.Phase_Portrait_u_tau}.

The above $\al$ is the parameter and the following maps are all strictly increasing and bijective:
\begin{equation}\label{maps-parameter}
	r\in(1, 3-\sqrt3)\mapsto w_-\in\left(0,\frac{3-\sqrt3}2\right)\mapsto a\in(0,\sqrt3)\mapsto \al\in(0,1).
\end{equation}
Under the transformation $(\sigma,w)\mapsto(\tau, u)$, $P_6$ and $P_2$ are mapped into $Q_6$ and $Q_2(\tau=0,u=1)$, respectively; the solution $w=w_F(\sigma)$ of \eqref{ODE_w_sigma} connecting $P_6$ and $P_2$ is converted to a solution $u=u_F(\tau)$ of \eqref{ODE_u_tau} connecting $Q_6$ and $Q_2$,
and $u_F$ depends smoothly on the parameter $\al$. Moreover,
\begin{align}
	\lim_{\tau\uparrow\tau(Q_6)}u_F(\tau)=+\infty,\quad\tau(Q_6)=\frac{w_--r+1}{1-w_-}=\al,\label{Eq.tau(Q6)=al}\\
	u_\text g(\tau)<u_F(\tau)<u_\text b(\tau),\quad\forall\ 0<\tau<\tau(Q_6)=\al.\label{Eq.u_F_property}
\end{align}  See also Figure \ref{Fig.Phase_Portrait_u_tau}.

\begin{figure}[htbp]
	\centering
	\includegraphics[width=1\textwidth]{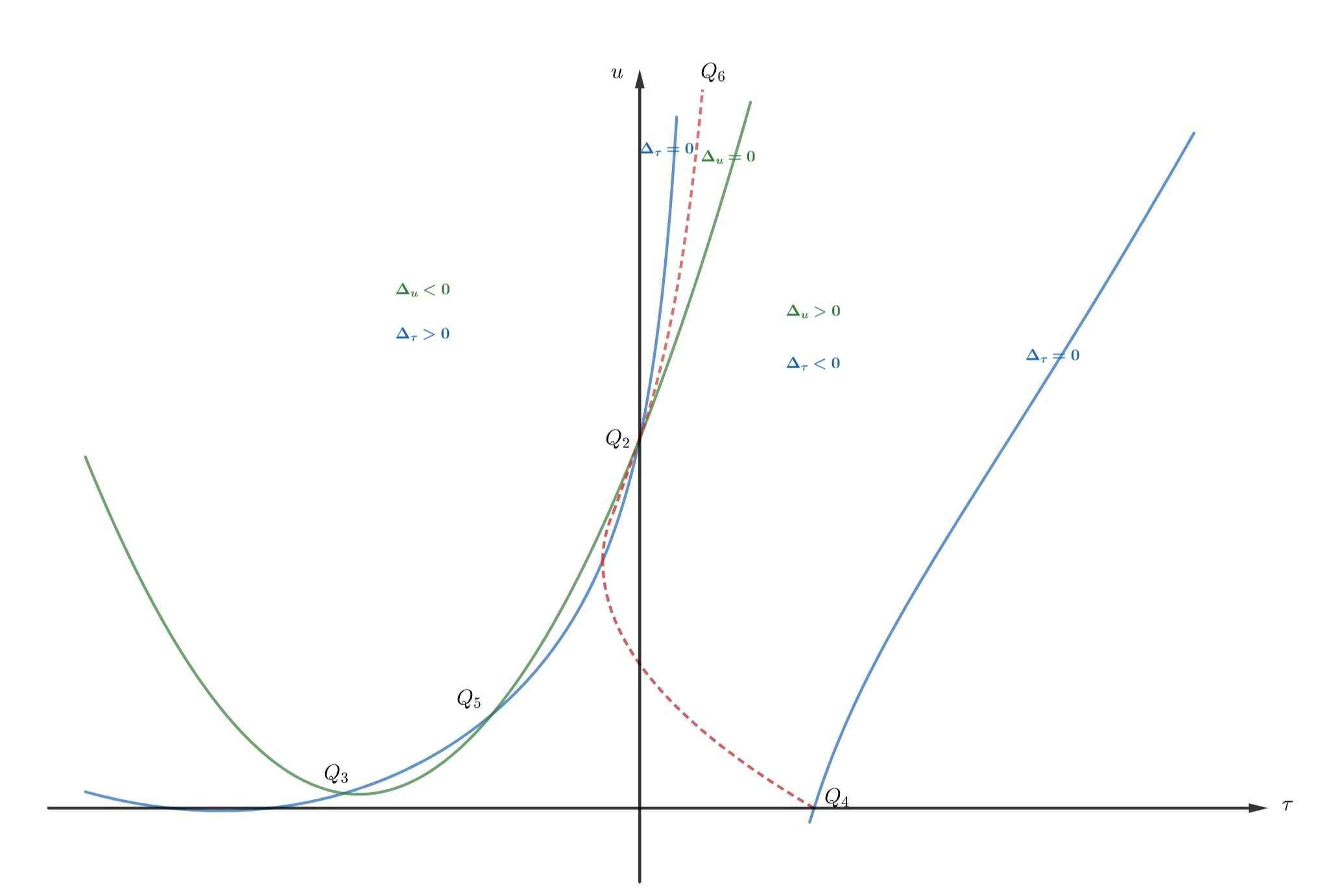}
	\caption{Phase portrait for the $\tau-u$ system \eqref{ODE_u_tau}: Dashed curve is the trajectory of the solution.}
	\label{Fig.Phase_Portrait_u_tau}
\end{figure}

Thanks to the re-normalization \eqref{Eq.renormalization}, the analysis of local solutions to \eqref{ODE_w_sigma} near $P_2$ is equivalent to the analysis of local solutions to \eqref{ODE_u_tau} near $Q_2(0,1)$. As in \cite{MRRJ2,SWZ2024_1}, the eigen-system of \eqref{ODE_u_tau} near $Q_2$ introduces a parameter that plays a crucial role in our analysis.
Let
\begin{equation}\label{Eq.eigen_system}
	\mathcal M:=\begin{pmatrix}
		\pa_u\Delta_u & \pa_\tau\Delta_u\\
		\pa_u\Delta_\tau &  \pa_\tau\Delta_\tau
	\end{pmatrix}\Bigg|_{(\tau, u)=Q_2}=\begin{pmatrix}
		-2 & -\frac43(\al-4)\\
		3\al & -2-4\al
	\end{pmatrix}.
\end{equation}
The eigenvalues of $\cM$ are given by
\begin{align*}
	\la^2+4(\al+1)\la+4(\al-1)^2=0\Longrightarrow
\la_\pm=-2(\al+1\mp2\sqrt\al)=-2(\sqrt\al\mp1)^2<0.
\end{align*}
We introduce the parameter
\begin{equation}\label{Eq.R}
	R:=\frac{\la_-}{\la_+}=\frac{(\sqrt\al+1)^2}{(\sqrt\al-1)^2}>1,
\end{equation}
which will play a central role in the coefficients of the power series expansion of $u$ at the sonic point $Q_2$. Note that the map $$\al\in(0,1)\mapsto\frac{1+\sqrt\al}{1-\sqrt\al}\in(1,+\infty)$$
is smooth, strictly increasing and bijective. Combining this with \eqref{maps-parameter}, we see that\footnote{The strict increase of the map $r\mapsto R$ holds for general $d$ and $\ell$, which can be checked similarly. Note that this fact has also been shown in \cite[Lemma 2.1]{BCLGS}, 
via a computer-assisted proof. Here we give a simple proof.}
\begin{equation}\label{Eq.r_to_R_increasing}
	r\in(1,3-\sqrt3)\mapsto R\in(1,+\infty)\quad\text{is smooth, strictly increasing and bijective}.
\end{equation}
Recall that we aim to construct smooth profiles for $r$ sufficiently close to $3-\sqrt3$, which corresponds to sufficiently large $R$.

\if0 Let us convert the solution $w_{F}$ to \eqref{ODE_w_sigma} connecting $P_6$ and $P_2$ into the smooth solution $u=u_F(\tau;R)$ to \eqref{ODE_u_tau}, defined on $\tau\in(0, \tau(Q_6))$ such that $\lim_{\tau\downarrow 0} u_F(\tau)=1$ and $u_F$ depends smoothly on the parameter $R$. Moreover,
\begin{align}
\lim_{\tau\uparrow\tau(Q_6)}u_F(\tau)=+\infty,\quad\tau(Q_6)=\frac{w_--r+1}{1-w_-}=\al,\label{Eq.tau(Q6)=al}\\
	u_\text g(\tau)<u_F(\tau)<u_\text b(\tau),\quad\forall\ 0<\tau<\tau(Q_6)=\al.\label{Eq.u_F_property}
\end{align}\fi

In order to extend $u_F$ smoothly through $Q_2$, our strategy is to construct a local smooth solution $u_L(\tau, R)$ to \eqref{ODE_u_tau} near $Q_2$, and then prove that $u_L(\cdot;R)=u_F(\cdot;R)$ for some well-chosen parameter $R$.

To construct $u_L$, we seek local solutions that can be expressed as a power series
\begin{equation}\label{Eq.u_L_series}
	u_L(\tau)=\sum_{n=0}^\infty a_n\tau^n.
\end{equation}
Plugging the series into \eqref{ODE_u_tau} yields the recurrence relation for the sequence $\{a_n\}$. A simple upper bound estimate of $\{a_n\}$ implies that the series defining $u_L$ in \eqref{Eq.u_L_series} is uniformly convergent in a neighborhood of $\tau=0$, for all $R\in(1,+\infty)\setminus\Z$. Hence, \eqref{Eq.u_L_series} gives a local analytic solution to \eqref{ODE_u_tau} near $Q_2$, and moreover, $u_L$ depends smoothly on the parameter $R\in(1,+\infty)\setminus\Z$. For full details, readers are referred to Section \ref{Sec.local_sol}.
\smallskip

{\textbf{Step 3.} \underline{Sharp estimates of the coefficients $\{a_n\}$.}
The first key point is to rewrite the recurrence relation of $\{a_{n}\}_{n=0}^\infty$ (see \eqref{Eq.a_n_recurrence_relation}) in the following form
\begin{equation}\label{Eq.re-formulation}
	\g_na_n=2p_na_{n-1}+q_na_{n-2}+s_1a_{n-3}+s_2a_{n-4}+\wt\varepsilon_n,\quad\forall\ n\geq 10,
\end{equation}
where both $s_1$ and $s_2$ are independent of $n$, hence the main order terms are given by the combination of $a_{n-1}$ and $a_{n-2}$. This motivates the introduction of the comparison sequence $\{M_n\}$, given by \eqref{Eq.M_n_sequence}. We prove the following sharp estimate of the sequence $\{a_n\}$.
\begin{proposition}\label{Prop.a_n_sharp_est}
	Let the sequence $\{M_n>0\}_{n=0}^{\lceil R \rceil-1}$ be given by \eqref{Eq.M_n_sequence}. Then there exists $N_0\in\N$ such that for all $R\in(N_0,+\infty)\setminus\Z$, we have
	\[a_n\asymp M_n,\quad\forall \ n\in[\sqrt R, R).\]
	Here the implicit constant is independent of $R$ and $n$.
\end{proposition}

It is natural to expect that the behavior of $\{a_n\}$ is approximated by the limiting sequence $\{a_n^\infty\}$ as $R\to+\infty$. However, we observe that $\lim_{R\to+\infty}p_n=0$, and hence the limiting sequence obeys a second-order recurrence relation, where the main order is given by $a_{n-2}^\infty$. This degenerate phenomenon is precisely caused by the triple-point degeneracy \eqref{Eq.degeneracy} of the phase portrait. As a consequence,  $a_n^\infty$ grows like $\sqrt{\Gamma(C+n)}/A_*^n$ for some constants $C\in\R$ and $A_*>0$, where $\Gamma$ is the Gamma function. This behavior differs significantly from the case  $d\neq\ell$, where the limiting sequence grows in the order of $\Gamma(C+n)/A_*^n$; see \cite[proposition 5.1]{MRRJ2}.

Section \ref{Sec.an} is primarily devoted to proving Proposition \ref{Prop.a_n_sharp_est}. Since the proof is intricate and highly technical, we divide it into three subsections:\smallskip

In Subsection \ref{Subsec.re-formulation}, we deduce the re-formulation \eqref{Eq.re-formulation} of the recurrence relation \eqref{Eq.a_n_recurrence_relation}.

In Subsection \ref{Subsec.a_n_upper_bound}, we prove the upper-bound aspect of Proposition \ref{Prop.a_n_sharp_est}, namely, Proposition \ref{Prop.a_n_upper_bound}. Here, we introduce another comparison sequence $\{M_n^*\}$ possessing some convexity properties, which facilitates the estimation of the summation terms in $\wt \varepsilon_n$ (see \eqref{Eq.tilde_epsilon_n}). The introduction of convexity replaces various convolution-type estimates of Gamma functions (see \cite[Appendices B and C]{MRRJ2}). We then prove that $M_n^*$ can be bounded by $M_n$. To overcome the degeneracy of the coefficient of $a_{n-1}$, we consider a new sequence $\{|a_n|+\la_n|a_{n-1}|\}$, for which the degeneracy disappears. The coefficient $\la_n$ is carefully chosen (see Lemma \ref{Lem.la_n_def}) such that the recurrence relation of $\{|a_n|+\la_n|a_{n-1}|\}$ has a simple leading order. We use the comparison sequence $\{\wt M_n\}$ to capture this leading order. We will prove that $\wt M_n$ are comparable with $M_n$, thanks to the well-chosen $\la_n$. Hence, the basic logic is
\[|a_n|\leq |a_n|+\la_n|a_{n-1}|\lesssim M_n^*\lesssim \wt M_n\asymp M_n,\quad\ \forall\ n\in\Z\cap[1, R).\]

In Subsection \ref{Subsec.a_n_lower_bound}, we prove the lower-bound aspect of Proposition \ref{Prop.a_n_sharp_est}, namely, Proposition \ref{Prop.an_lower_bound}, where we consider the limiting sequence $\{a_n^\infty\}$ as $R\to+\infty$. We observe that the sequence
\[\left\{\frac{a_n^\infty+\la_n^\infty a_{n-1}^\infty}{M_n^\infty}\right\}\]
has a positive limit as $n\to+\infty$, see \eqref{Eq.S_infty}. Based on this observation, we are able to prove that $(a_n+\la_na_{n-1})/M_n$ has a positive lower bound for appropriately large $n$ and large $R$, see Lemma \ref{Lem.a_n_lower}. Here, we emphasize that, due to the use of a perturbation argument, we are not able to prove the positive lower bound for all $n\in[N_0, R)$. Instead, we can only get the result for $n\in[N_0, c_1R]$, where $c_1\in(0,1)$ is a small constant independent of $n$ and $R$. The next step is to prove that $a_{n}/M_n-a_{n-1}/M_{n-1}$ is small enough, by considering the recurrence of the ``Wronskian" sequence $\{a_nM_{n-1}-a_{n-1}M_n\}$, for which we have a simple leading order, recorded by $\{L_n\}$; see Lemma \ref{Lem.a_n/M_n_lower}. So far, roughly speaking, we have proved
\[\frac{a_n}{M_n}+\la_n\frac{a_{n-1}}{M_n}>c_0>0\quad\text{and}\quad\la_n\frac{a_{n-1}}{M_n}\approx \frac{a_{n-1}}{M_{n-1}}\approx\frac{a_n}{M_n},\]
from which we immediately find that $a_n/M_n$ has a positive lower bound for $n\in[\sqrt R, R]$.

Proposition \ref{Prop.a_n_sharp_est} then follows from combining Proposition \ref{Prop.a_n_upper_bound} and Proposition \ref{Prop.an_lower_bound}. For full details, see Section \ref{Sec.an}.

\smallskip

{\textbf{Step 4.} \underline{Passing through the sonic point smoothly.} Using the estimates of $a_{n}$ established in Section \ref{Sec.an}, we have
the following two propositions.

\begin{proposition}\label{Prop.R_near_odd}
	There exists $N_0\in\mathbb N$ satisfying the following property: for all $N\in \Z\cap(N_0, +\infty)$, we can find $\th_{1}\in(0,1)$ such that for all $R\in(N, N+\th_{1})$ we have
	\begin{equation}\label{Eq.u_L>u_F}
		u_L(\tau;R)>u_F(\tau;R)\quad \forall\ \tau\in(0,\tau_1(R)]\quad \text{for some}\quad\tau_1(R)>0.
	\end{equation}
\end{proposition}

\begin{proposition}\label{Prop.R_near_even}
	There exists $N_0\in\mathbb N$ satisfying the following property: for all $N\in \Z\cap(N_0, +\infty)$, we can find $\th_{2}\in(0,1)$ such that for all $R\in(N+1-\th_2, N+1)$ we have
	\begin{equation}\label{Eq.u_L<u_F}
		u_L(\tau;R)<u_F(\tau;R)\quad \forall\ \tau\in(0,\tau_2(R)]\quad \text{for some}\quad\tau_2(R)>0.
	\end{equation}
\end{proposition}

With the above two propositions at hand, we let $N_0\in\N$ be sufficiently large so that both conclusions of these two propositions hold for all $N\in \Z\cap(N_0, +\infty)$. Fixing such an $N$, a direct application of the intermediate value theorem implies the existence of $R_N\in(N, N+1)$ such that $u_L(\cdot; R_N)=u_F(\cdot; R_N)$. Hence, $u_F(\cdot; R_N)$ can be extended to pass through $Q_2$ smoothly. Consequently, using the inverse of \eqref{Eq.renormalization}, we extend $w_F(\cdot; R=R_N)$ to the left so that it passes through $P_2$ smoothly.

The proof of Proposition \ref{Prop.R_near_odd} and Proposition \ref{Prop.R_near_even} is based on the barrier function method. For full details, see Section \ref{Solution curve passing through the sonic point}. We emphasize that in Section \ref{Solution curve passing through the sonic point}, the integer $N$ is fixed, and we consider the limits $R\downarrow N$ or $R\uparrow N+1$, where the implicit constants in the estimates are allowed to depend on $N$ (although should not depend on $R$). This approach simplifies our analysis compared to establishing the barrier properties in Section \ref{Global extension of the solution curve}.
\smallskip

\textbf{Step 5.} \underline{Global extension of the solution curve.} After passing the sonic point $P_{2}$, we prove that if $N\in(N_0, +\infty)\cap(2\Z+1)$, then the extension of $w_F(\cdot; R=R_N)$ will exit the region between $P_5$ and $P_2$ by crossing the red curve in Figure \ref{Fig.Phase_Portrait_w_sigma}. Subsequently, it will be extended toward the left until reaching $P_4$. This corresponds to the extension of solution curve $u_F(\cdot; R_N)$ (denoted by $u_E(\cdot; R_N)$) to \eqref{ODE_u_tau} in Figure \ref{Fig.Phase_Portrait_u_tau}. We need to show that it exits the region between $Q_2$ and $Q_5$ by crossing the blue curve, and then it will be extended toward the right until reaching $Q_4$. Our main goal in Section \ref{Global extension of the solution curve} is to achieve this result, primarily by leveraging the observation thatt $u_{(N)}(\tau):=\sum_{n=0}^Na_n\tau^n$ serves a barrier function.
\begin{proposition}\label{Prop.u_N_compare0}
	Let $u_{(N)}(\tau)=a_0+a_1\tau+\cdots+a_N\tau^N$. There exists $N_0\in\Z_+$ such that if $N\in(N_0, +\infty)\cap(2\Z+1)$ and $R\in(N, N+1)$, then
(here $\cL $ is defined in \eqref{Eq.Lu})
	\begin{equation*}
		\cL\left[u_{(N)}\right](\tau)<0,\quad\forall\ \tau\in\left[-4/\sqrt R, 0\right).
	\end{equation*}
\end{proposition}
Proposition \ref{Prop.u_N_compare0} implies that $u_E<u_{(N)}$ for $\tau<0$. On the other hand, it will be shown that $u_{(N)}$ exits the region between $Q_2$ and $Q_5$ by crossing the blue curve (see Lemma \ref{Lem.u_N_zero}), and hence so does $u_E$. We then note that, after crossing the blue curve, $u_E$ will not touch the other blue curve again, unless it reaches $Q_4$. This is a consequence of the observation that the straight segment $U_\text O$ connecting $Q_2$ and $Q_4$ is a barrier function (see Lemma \ref{Lem.U_O_compare}). As a result, in Figure \ref{Fig.Phase_Portrait_w_sigma}, the solution curve will not touch the curve $\sigma\mapsto(\sigma, w_1(\sigma))$ for $\sigma>0$.\footnote{This causes an inconsistency between our Figure \ref{Fig.Phase_Portrait_w_sigma} and figures in \cite{MRRJ2}. The solution curves in \cite{MRRJ2} will intersect the curve $\sigma\mapsto(\sigma, w_1(\sigma))$ for some $\sigma>0$, which will not happen in our situation.}

The proof of Proposition \ref{Prop.u_N_compare0} relies heavily on the estimates of $a_n$ obtained in Section \ref{Sec.an}. Unlike in Section \ref{Solution curve passing through the sonic point}, all estimates here must be uniform with respect to both $R\in(N, N+1)$ and $N\in(N_0, +\infty)\cap(2\Z+1)$. We note that the convexity properties used in Section \ref{Solution curve passing through the sonic point} simplifies the analysis again in Section \ref{Global extension of the solution curve}.

\subsection{Sketch of the proof: repulsivity}\label{Subsec.sketch_repulsivity}

In Section \ref{Section repulsivity property}, we prove Proposition \ref{Prop.coercive}. The first two properties (\eqref{repuls 1} and \eqref{repuls 2}) follow from a standard ODE argument involving applications of Grönwall's inequality; see Subsection \ref{Subsection repuls 12}. The proof of the last two properties (\eqref{repuls 3} and \eqref{repuls 4}) is much more challenging. We introduce two auxiliary points, $P_A$ and $P_B$; see Figure \ref{Fig.w-sigma-repulsivity} for their positions. Using the re-normalization \eqref{Eq.renormalization}, $P_A$ and $P_B$ are mapped to $Q_A$ and $Q_B$; see Figure \ref{Fig.u-tau-replusivity}. We prove \eqref{repuls 3} and \eqref{repuls 4} by considering four segments $P_6-P_2$, $P_2-P_A$, $P_A-P_B$, and $P_B-P_4$ separately. The first two segments are handled by carefully choosing barrier functions under $\tau-u$ coordinates; see Lemma \ref{Lem.repulse1} and Lemma \ref{Lem.repulse2}. The last two segments, $P_A-P_B$ and $P_B-P_4$, are simpler and follow from direct computations; see Lemma \ref{Lem.repulse3} and Lemma \ref{Lem.repulse4}.

\medskip

\noindent{\bf Organization of the paper.} The rest of the paper is organized as follows. In Section \ref{Sec.local_sol}, we present some qualitative properties of $u_{L}(\tau)$ near $Q_{2}$. In Section \ref{Sec.an}, we  establish quantitative properties of $u_{L}(\tau)$ near $Q_{2}$ by obtaining sharp estimates for the sequence $a_{n}$. Section \ref{Solution curve passing through the sonic point} is devoted to proving Propositions \ref{Prop.R_near_odd} and \ref{Prop.R_near_even}, which show that there exists a sequence $\{R_N\}$ such that $u_F(\cdot;R_N)$ passes through the sonic point $Q_2$ smoothly.
In Section \ref{Global extension of the solution curve}, we prove the global extension of the solution curve to the origin, thereby completing the proof of Theorem \ref{Thm.ODE}. Finally, in the last section, we prove Proposition \ref{Prop.coercive}, which provides the repulsivity for the constructed solution curve.

\section{Local smooth solution around the sonic point}\label{Sec.local_sol}

In this section, we construct an analytic local solution $u_L$ to the $\tau-u$ ODE \eqref{ODE_u_tau} that passes through the sonic point $Q_2(\tau=0, u=1)$, for each $R\in(1,+\infty)\setminus\Z$.

\subsection{Fundamental properties of the phase portrait near $Q_2$}
Near the sonic point $Q_2$, we write \eqref{ODE_u_tau} in the form
\begin{equation}\label{Eq.u_tauODE}
	\frac{\mathrm du}{\mathrm d\tau}=\frac{\Delta_u(\tau, u)}{\Delta_\tau(\tau, u)},
\end{equation}
where $\Delta_u$ and $\Delta_\tau$ are given by \eqref{Eq.Delta_u} and \eqref{Eq.Delta_tau}, respectively. For further use, given a $C^1$ function $u=u(\tau)$, we denote
\begin{equation}\label{Eq.Lu}
	\cL[u](\tau):=\Delta_\tau(\tau, u(\tau))\frac{\mathrm du}{\mathrm d\tau}-\Delta_u(\tau, u(\tau)).
\end{equation}

One computes directly that (recalling \eqref{Eq.u_g}, \eqref{Eq.u_b} and $\al\in(0,1)$)
\begin{equation}
	\frac{\mathrm du_{\text g}}{\mathrm d\tau}(0)=\frac23(4-\al)>0,\quad \frac{\mathrm du_{\text b}}{\mathrm d\tau}(0)=\frac{4\al+2}{3\al}>0.
\end{equation}

Assume that $u(\tau)$ is a local smooth solution of \eqref{Eq.u_tauODE} passing through $Q_2$. Let $a_1=\mathrm du/\mathrm d\tau(0)$. By L'H\^opital's rule, \eqref{Eq.u_tauODE} and \eqref{Eq.eigen_system}, we have
\[a_1=\frac{\mathrm du}{\mathrm d\tau}\Big|_{\tau=0}=\frac{\pa_u\Delta_u(Q_2)a_1+\pa_\tau\Delta_u(Q_2)}{\pa_u\Delta_\tau(Q_2)a_1+\pa_\tau\Delta_\tau(Q_2)}=\frac{-2a_1-\frac43(\al-4)}{3\al a_1-2-4\al},\]
i.e.,  $a_1$ solves the quadratic equation $3\al a_1^2-4\al a_1+\frac43(\al-4)=0$, hence (recalling $\al\in(0,1)$)
\[a_1=\frac{2\sqrt\al+4}{3\sqrt\al}>0\quad\text{or}\quad a_1=\frac{2\sqrt\al-4}{3\sqrt\al}<0.\]

Note that $u_{\text g}(\tau)<u_F(\tau)<u_{\text b}(\tau)$ for $\tau>0$. Our aim is to construct a local solution $u_L$ near $Q_2$ such that $u_L=u_F$. Hence ,it is natural to require the slope of $u_L$ at $Q_2$ to be the positive one.  Accordingly, we designate
\begin{equation}\label{Eq.a_1}
	\frac{\mathrm du_L}{\mathrm d\tau}(0)=a_1=\frac{2\sqrt\al+4}{3\sqrt\al}.
\end{equation}
Moreover, one checks easily that
\begin{equation}\label{Eq.slopes}
	\frac{\mathrm du_{\text g}}{\mathrm d\tau}(0)=\frac23(4-\al)<\frac{2\sqrt\al+4}{3\sqrt\al}<\frac{4\al+2}{3\al}=\frac{\mathrm du_{\text b}}{\mathrm d\tau}(0),\quad\forall\ \al\in(0,1).
\end{equation}

\subsection{Analytic solutions near $Q_2$}
In this subsection we construct an analytic local solution to \eqref{Eq.u_tauODE} that passes through $Q_2$ with slope $a_1$ at $Q_2$. We write the Taylor series of a general analytic function $u$ around $Q_2$ with $u(0)=1$ and $u'(0)=a_1$ as
\begin{equation}\label{Eq.u_series}
	u(\tau)=\sum_{n=0}^\infty u_n\tau^n,
\end{equation}
where $u_0=1$ and $u_1=a_1$ is given by \eqref{Eq.a_1}.

\begin{lemma}\label{Lem.recurrence_relation}
	Let $u(\tau)$ be given by \eqref{Eq.u_series} and $\cL[u]$ be defined by \eqref{Eq.Lu}. Then we have
	\begin{equation}
		\cL[u](\tau)=\sum_{n=0}^\infty \cU_n\tau^n,
	\end{equation}
	where $\cU_0=\cU_1=0$ and for each $n\geq 2$.
	\begin{equation}\label{Eq.U_n}
		\cU_n=\delta(R-n)u_n-E_n,
	\end{equation}
	with $\delta:=2(1-\sqrt\al)^2$ and $E_n$ given by
	\begin{equation}
		\begin{aligned}
			E_n:=&\left[(\al-2)(n-1)-\frac43(\al-4)\right]u_{n-1}-\left(n-\frac{16}3\right)u_{n-2}\\
			&\qquad-\frac32\al(n+1)\sum_{j=2}^{n-1}u_ju_{n+1-j}
			+\left(\frac32n-2\right)\sum_{j=1}^{n-1}u_ju_{n-j}.
		\end{aligned}
	\end{equation}
\end{lemma}
\begin{proof}
	For notational simplicity, we set $u_n=0$ for $n<0$ and
	\[(u^2)_n:=\sum_{j=0}^nu_ju_{n-j}\quad\forall\ n\in\Z_{\geq 0}.\]
	With this notation, we have
	\[u^2=\sum_{n=0}^\infty(u^2)_n\tau^n,\quad uu'=\frac12(u^2)'=\frac12\sum_{n=0}^\infty (n+1)(u^2)_{n+1}\tau^n.\]
	Using \eqref{Eq.a_1}, it is a direct computation to find that $\cL[u](\tau)=\sum_{n=0}^\infty\cU_n\tau^n$, with $\cU_0=\cU_1=0$ and for $n\geq 2$,
	\begin{align*}
		\cU_n&=\frac32\al(n+1)(u^2)_{n+1}-\left(\frac32n-2\right)(u^2)_n-3\al(n+1)u_{n+1}\\
		&\qquad-\left[(4\al-1)n+2\right]u_n-\left[(\al-2)(n-1)-\frac43(\al-4)\right]u_{n-1}+\left(n-\frac{16}3\right)u_{n-2}\\
		&=\left[(3\al a_1-4\al-2)n+3\al a_1+2\right]u_n-E_n.
	\end{align*}
	By \eqref{Eq.a_1}, we have
	\[3\al a_1-4\al-2=-2(1-\sqrt\al)^2,\quad 3\al a_1+2=2(1+\sqrt\al)^2,\]
	hence, 
	$$(3\al a_1-4\al-2)n+3\al a_1+2=2(1-\sqrt\al)^2\left(\frac{(1+\sqrt\al)^2}{(1-\sqrt\al)^2}-n\right)=\delta(R-n),$$
	where $R$ is defined in \eqref{Eq.R}. This proves \eqref{Eq.U_n}.
\end{proof}

\begin{lemma}\label{Lem.a_n_bound}
	Let $\{a_n\}_{n=0}^\infty$ be a sequence such that $a_0=1$, $a_1$ is given by \eqref{Eq.a_1} and
	\begin{equation}\label{Eq.a_n_recurrence_relation}
		\begin{aligned}
			\delta(R-n)a_n=\wt E_n:&=\left[(\al-2)(n-1)-\frac43(\al-4)\right]a_{n-1}-\left(n-\frac{16}3\right)a_{n-2}\\
			&\qquad-\frac32\al(n+1)\sum_{j=2}^{n-1}a_ja_{n+1-j}
			+\left(\frac32n-2\right)\sum_{j=1}^{n-1}a_ja_{n-j}
		\end{aligned}
	\end{equation}
	for all $n\geq 2$. Assume that $\beta\in(1,2)$, $\cC\subset(1,+\infty)\setminus\Z$ is compact and $R\in \cC$. Then there exists  a constant $K=K(\cC,\beta)>1$ independent of $R\in\cC$ such that
	\begin{equation}\label{Eq.a_n_bound}
		|a_n|\leq \mathfrak{C}_{n-1}K^{n-\beta}\quad\forall\ n\geq2.
	\end{equation}
	Here $\mathfrak{C}_n$ is te Catalan number.
\end{lemma}
\begin{proof}
	The proof is quite  similar to \cite[Lemma 4.5]{SWZ2024_1}. Assume that for all $R\in\cC$ and $n\in\Z_+$, we have $R\leq R_M$ and $|R-n|\geq R_m>0$. It follows from \eqref{Eq.a_1}, $\delta=2(1-\sqrt\al)^2$ and \eqref{Eq.r_to_R_increasing} that $R\mapsto a_1, R\mapsto \delta$ are smooth functions. Hence, $\delta>\delta_m>0$ for all $R\in\cC$.
	
	It follows directly  from \eqref{Eq.a_n_recurrence_relation} that $R\mapsto a_n$ is smooth in $R\in(1,+\infty)\setminus\{2,3,\cdots,n\}$ and \eqref{Eq.a_n_bound} holds for $2\leq n\leq 4$, provided that we take $K>1$ sufficiently large.
	We assume for induction that $R\mapsto a_n$ is smooth in $R\in(1,+\infty)\setminus\{2,3,\cdots,n\}$ and \eqref{Eq.a_n_bound} holds for $2\leq n\leq N-1$, where $N\geq 5$. We have
	\begin{align*}
		\left|\sum_{j=2}^{N-1}a_ja_{N+1-j}\right|\leq \sum_{j=2}^{N-1}\mathfrak{C}_{j-1}\mathfrak{C}_{N-j}K^{N+1-2\beta}\lesssim \mathfrak{C}_{N-1}K^{N+1-2\beta},
	\end{align*}
	where we have used $\mathfrak{C}_n\sim \mathfrak{C}_{n-1}$; and similarly,
	\begin{align*}
		\left|\sum_{j=1}^{N-1}a_ja_{N-j}\right|&=\left|2a_1a_{N-1}+\sum_{j=2}^{N-2}a_ja_{N-j}\right|\lesssim \mathfrak{C}_{N-2}K^{N-1-\beta}\\
		&\qquad+ \sum_{j=2}^{N-2}\mathfrak{C}_{j-1}\mathfrak{C}_{N-j-1}K^{N-2\beta}\lesssim \mathfrak{C}_{N-1}K^{N-1-\beta};
	\end{align*}
	hence,
	\begin{align*}
		\left|\wt E_N\right|\lesssim(N+1)\left(\left|\sum_{j=2}^{N-1}a_ja_{N+1-j}\right|+\left|\sum_{j=1}^{N-1}a_ja_{N-j}\right|+|a_{N-1}|+|a_{N-2}|\right)
		\lesssim  N \mathfrak{C}_{N-1}K^{N+1-2\beta}.
	\end{align*}
	Therefore, by \eqref{Eq.a_n_recurrence_relation} we know that $R\mapsto a_N$ is smooth in $R\in(1,+\infty)\setminus\{2,3,\cdots,N\}$ and for $R\in\cC$ we have
	\[|a_N|=\frac{\left|\wt E_N\right|}{|R-N|\delta}\leq \frac{C_0(\mathcal C,\beta)}N\left|\wt E_N\right|\leq C_1(\mathcal C,\beta) \mathfrak{C}_{N-1}K^{N+1-2\beta}\leq\mathfrak{C}_{N-1}K^{N-\beta},\]
	if we choose $K>C_1^{1/(\beta-1)}$, and hence we close the induction.
\end{proof}

\begin{remark}
	Let $N\in\Z_{\geq 2}$. From the proof of Lemma \ref{Lem.a_n_bound}, we see that for all $n\in\Z_{\geq 0}$, $a_n$ is continuous in $R\in(1,+\infty)\setminus\Z$, and $a_0, a_1,\cdots, a_{N-1}$ are uniformly bounded for $R\in(1, N]$; furthermore, we have $\left|a_N(R)\right|\lesssim |R-N|^{-1}$ as $R\uparrow N$.
\end{remark}

\begin{proposition}\label{Prop.local_sol}
	Assume that $\cC\subset(1,+\infty)\setminus\Z$ is compact and $R\in \cC$. Let $\{a_n\}_{n=0}^\infty$ be defined as in Lemma \ref{Lem.a_n_bound}. Then there exists $\tau_0=\tau_0(\cC)>0$ such that the Taylor series
	\[u_L(\tau)=u_L(\tau;R)=\sum_{n=0}^\infty a_n\tau^n\]
	converges absolutely for $\tau\in[-\tau_0,\tau_0]$, and it gives the unique analytic solution to the $u-\tau$ ODE \eqref{ODE_u_tau} in $[-\tau_0,\tau_0]$ with $u_L(0)=1$ and $u_L'(0)=a_1$. Moreover, the function $u_L(\tau;R)$ is continuous in the domain $\{(\tau,R):\tau\in[-\tau_0,\tau_0], R\in\cC\}$.
\end{proposition}
\begin{proof}
	Since $\mathfrak{C}_n\leq 4^n$ for all $n\geq 0$, it follows from Lemma \ref{Lem.a_n_bound} that there exists a constant $K=K(\mathcal C)>1$ such that $|a_n|\leq K^n$ for all $n\geq0$ and all $R\in\mathcal C$. Taking $\tau_0=\tau_0(\mathcal C)=\frac{1}{2K}>0$, the power series $u_L(\tau)=\sum_{n=0}^\infty a_n\tau^n$ is uniformly convergent and analytic in $\tau\in[-\tau_0, \tau_0]$. It follows from Lemma \ref{Lem.recurrence_relation} that the function $u_L$ is the unique analytic solution to \eqref{ODE_u_tau} in $[-\tau_0, \tau_0]$ with $u_L(0)=1$ and $u_L'(0)=a_1$. Finally, the continuity of $u_L(\tau; R)$ in the domain $\{(\tau,R):\tau\in[-\tau_0,\tau_0], R\in\cC\}$ follows from the fact that uniform convergence preserves the continuity of $a_n$ in $R$.
\end{proof}

\begin{remark}\label{Rmk.u_F_continuity}
	For further use, we note that the $Q_6-Q_2$ solution  $u_F(\tau; R)$ is continuous in the region 
	\[\left\{(\tau, R):0<\tau<\tau(Q_6)=\al=\al(R), R\in(1,+\infty)\right\},\]
	where $\tau(Q_6)$ is defined in  \eqref{Eq.tau(Q6)=al}.
\end{remark}

\section{Quantitative properties of local analytic solutions}\label{Sec.an}

In Proposition \ref{Prop.local_sol}, we show that for $R\in(1,+\infty)\setminus\Z$, the $u-\tau$ ODE \eqref{ODE_u_tau} has a unique local analytic solution $u_L(\tau)=\sum_{n=0}^\infty a_n\tau^n$ with $u_L(0)=1, \frac{\mathrm du_L}{\mathrm d\tau}(0)=a_1$, where the sequence $\{a_n\}$ is given by Lemma \ref{Lem.a_n_bound}. Sometimes, we denote $u_L(\tau)$ by $u_L(\tau; R)$ to emphasize the dependence on $R$. Recall that we also construct the $Q_6-Q_2$ solution curve $u=u_F$. If $u_L$ and $u_F$ are equal in a neighborhood of $\tau=0$, then we know that $u_F$ is smooth at $Q_2$. However, it turns out that $u_L$ and $u_F$ are not the same for some parameters. Our strategy is to use the intermediate value theorem to find suitable parameters such that $u_L$ and $u_F$ are equal. To this end, we establish some qualitative properties of the coefficients $\{a_n\}$. Our main result of this subsection is Proposition \ref{Prop.a_n_sharp_est}.

Proposition \ref{Prop.a_n_sharp_est} implies, in particular, that for $N\in\N$ sufficiently large and $R\in(N, N+1)$, we have $a_n>0$ for $n\in\Z\cap[\sqrt N+1, N]$. For our further use, we have the following lemma stating that $a_{N+1}<0$.

\begin{lemma}\label{Lem.a_N+1<0}
	There exists $N_0\in\N$ such that for all $N>N_0$ and $R\in(N, N+1)$, we have $a_{N+1}<0$.
\end{lemma}

Recall the sketch of the proof from Subsection \ref{Subsec.sketch_existence}. This section is primarily devoted to proving Proposition \ref{Prop.a_n_sharp_est}. In Subsection \ref{Subsec.re-formulation}, we introduce a re-formulation of the recurrence relation \eqref{Eq.a_n_recurrence_relation}. The key point is to make the coefficients $a_{n-3}$ and $a_{n-4}$ independent of $n$, so that the main order terms are given by the combination of $a_{n-1}$ and $a_{n-2}$. This motivates the definition of the comparison sequence $\{M_n\}$. In Subsection \ref{Subsec.a_n_upper_bound}, we prove the upper-bound aspect of Proposition \ref{Prop.a_n_sharp_est}, namely, Proposition \ref{Prop.a_n_upper_bound}. Subsection \ref{Subsec.a_n_lower_bound} is dedicated to proving the lower-bound, namely, Proposition \ref{Prop.an_lower_bound}, where we consider the limiting sequence $\{a_n^\infty\}$ as $R\to+\infty$. The proof of Lemma \ref{Lem.a_N+1<0} is much simpler and can be found in Subsection \ref{Subsec.a_N+1<0}.

\subsection{Reformulation of the recurrence relation}\label{Subsec.re-formulation}
In this subsection, we derive a new formulation of the recurrence relation $\{a_n\}$, which motivates the definition of $\{M_n\}$. Recall that our goal in this section is to show that $a_n$ and $M_n$ have the same quantitative behavior, at least for $n\in[\sqrt R, R]$.

According to \eqref{Eq.a_n_recurrence_relation}, we can rewrite the recurrence relation for $\{a_n\}$ as follows
\begin{align}
	\gamma_n a_n=\wt E_n&=-(A_1n+B_1)a_{n-1}-(A_2n+B_2)a_{n-2}-(A_3n+B_3)a_{n-3}\label{Eq.a_n_recurrence1}\\
	&\quad-(A_4n+B_4)a_{n-4}-\frac32\al(n+1)\sum_{j=6}^{n-5}a_ja_{n+1-j}+\frac{3n-4}{2}\sum_{j=5}^{n-5}a_ja_{n-j}\nonumber
\end{align}
for all $n\in\Z_{\geq 10}$, where
\begin{align}
	&\gamma_n:=\delta(R-n)=2(1-\sqrt\al)^2(R-n),\label{Eq.gamma_n}\\
	A_1:=3\al a_2&-3a_1-\al+2,\quad B_1:=3\al a_2+4a_1+(7\al-22)/3,\nonumber\\
	A_2&:=3\al  a_3-3a_2+1,\quad B_2:=3\al a_3+4a_2-16/3,\nonumber\\
	A_3:=3\al a_4&-3a_3,\quad B_3:=3\al a_4+4a_3,\quad
	A_4:=3\al a_5-3a_4,\quad B_4:=3\al a_5+4a_4.\nonumber
\end{align}

It is convenient to introduce
\begin{equation}\label{Eq.A_def}
	A:=\sqrt R\in(1,+\infty),\qquad \la:=\sqrt\al\in(0,1).
\end{equation}
Then $A=(1+\sqrt\al)/(1-\sqrt\al)=(1+\la)/(1-\la)$, thus $\la=(A-1)/(A+1)$.
%\[A=\frac{1+\sqrt\al}{1-\sqrt\al}=\frac{1+\la}{1-\la}\Longrightarrow \la=\frac{A-1}{A+1}.\]
Using $a_0=1$, $a_1=(2\la+4)/(3\la)=(6A+2)/(3A-3)$, and \eqref{Eq.a_n_recurrence_relation}, one can check by induction that
\[A\mapsto a_n\quad \text{is analytic for}\quad A\in(1,+\infty]\setminus\left\{\sqrt k: k\in\Z\cap[0, n]\right\},\quad\forall\ n\in\Z_{\geq 0}.\]
Then $\g_n=\g_n(A)=\frac{8}{(A+1)^2}(A^2-n)=:\g_{(1)}n+\g_{(2)}$, and as $A\to+\infty$ we have
\begin{align}
	A_1=-\frac6A-\frac8{A^2}+\cO(A^{-3}),&\quad B_1=8-\frac2{3A}+\cO(A^{-2}),\label{Eq.A_1B_1_asymp}\\
	A_2=-1-\frac8A+\cO(A^{-2}),&\quad B_2=\frac{13}3+\frac{34}A+\cO(A^{-2}),\label{Eq.A_2B_2_asymp}\\
	A_3=-1-\frac{23}{4A}+\cO(A^{-2}),&\quad B_3=6+\frac{547}{12A}+\cO(A^{-2}),\nonumber\\
	A_4=-\frac38-\frac4A+\cO(A^{-2}),&\quad B_4=\frac{103}{24}+\frac{631}{12A}+\cO(A^{-2}).\nonumber
\end{align}

Replacing $n$ in \eqref{Eq.a_n_recurrence1} by $n-1$, we obtain
\begin{equation}\label{Eq.a_n_recurrence2}
	\begin{aligned}
		\g_{n-1}a_{n-1}&=-(A_1n+B_1-A_1)a_{n-2}-(A_2n+B_2-A_2)a_{n-3}-(A_3n+B_3-A_{3})a_{n-4}\\
		&\qquad-\frac32\al n\sum_{j=5}^{n-5}a_ja_{n-j}+\frac{3n-7}{2}\sum_{j=4}^{n-5}a_ja_{n-1-j};
	\end{aligned}
\end{equation}
Replacing $n$ in \eqref{Eq.a_n_recurrence1} by $n-2$ yields that
\begin{equation}\label{Eq.a_n_recurrence3}
	\begin{aligned}
		\g_{n-2}a_{n-2}&=-(A_1n+B_1-2A_1)a_{n-3}-(A_2n+B_2-2A_2)a_{n-4}\\
		&\qquad-\frac32\al (n-1)\sum_{j=4}^{n-5}a_ja_{n-1-j}+\frac{3n-10}{2}\sum_{j=3}^{n-5}a_ja_{n-2-j}.
	\end{aligned}
\end{equation}

Let
\begin{equation*}
	k_1:=\frac{A_1A_4-A_2A_3}{A_2^2-A_1A_3},\quad k_2:=\frac{A_3^2-A_4A_2}{A_2^2-A_1A_3}.
\end{equation*}
Here, we note that $A_2^2-A_1A_3=1+10/A+\cO(A^{-2})$ as $A\to+\infty$. Hence, $A_2^2-A_1A_3>1>0$ for sufficiently large $A$, and thus $k_1, k_2$ are well-defined for sufficiently large $A$.\footnote{Indeed, by definition, $A_1, A_2, A_3$ are explicit functions on $A$, and one can compute the expression of $A_2^2-A_1A_3$, which is the quotient of two polynomials of degree $20$, then one can show that $A_2^2-A_1A_3>1$ for all $A>1$. Nevertheless, the positivity of $A_2^2-A_1A_3$ for large $A$ is enough for our purpose.} For further use, we compute that
\begin{equation}\label{Eq.k_1k_2asymp}
	k_1=-1-\frac3{2A}+\cO(A^{-2}),\quad k_2=\frac58-\frac7{4A}+\cO(A^{-2}),\quad\text{as}\quad A\to+\infty.
\end{equation}

We consider $\eqref{Eq.a_n_recurrence1}+k_1\cdot\eqref{Eq.a_n_recurrence2}+k_2\cdot\eqref{Eq.a_n_recurrence3}$. We then deduce the following version of the recurrence relation:
\begin{equation}\label{Eq.a_n_new_recurrence}
	\g_na_n=2p_na_{n-1}+q_na_{n-2}+s_1a_{n-3}+s_2a_{n-4}+\wt\varepsilon_n,\quad\forall\ n\geq 10,
\end{equation}
where $\g_n$ is given by \eqref{Eq.gamma_n}, and
\begin{align}
	p_n:&=\left(-\frac{A_1}{2}+\frac{4k_1}{(A+1)^2}\right)n-\frac{B_1}{2}-\frac{4k_1(A^2+1)}{(A+1)^2}=:p_{(1)}n+p_{(2)},\label{Eq.p_n}\\
	& p_{(1)}=\frac3A+\cO(A^{-2}),\quad p_{(2)}=-\frac{5}{3A}+\cO(A^{-2}),\quad\text{as}\quad A\to+\infty,\label{Eq.p_n_asymp}\\
	\notag q_n:&=\left(-A_2-A_1k_1+\frac{8k_2}{(A+1)^2}\right)n-B_2+(A_1-B_1)k_1-\frac{8k_2(A^2+2)}{(A+1)^2}\\&=:q_{(1)}n+q_{(2)},\label{Eq.q_n}\\
	&\qquad q_{(1)}=1+\cO(A^{-1}),\quad q_{(2)}=-\frac43+\cO(A^{-1}),\quad\text{as}\quad A\to+\infty,\label{Eq.q_n_asymp}\\
	s_1:&=-B_3+(A_2-B_2)k_1+(2A_1-B_1)k_2=-\frac{17}3+\frac{34}{3A}+\cO(A^{-2}),\label{Eq.s_1asymp}\\
	s_2:&=-B_4+(A_3-B_3)k_1+(2A_2-B_2)k_2=-\frac54-\frac{131}{12A}+\cO(A^{-2}),\label{Eq.s_2asymp}\\
	\wt\varepsilon_n:&=-\frac32\al(n+1)\sum_{j=6}^{n-5}a_ja_{n+1-j}+\left(\frac{3n-4}{2}-\frac32\al k_1n\right)\sum_{j=5}^{n-5}a_ja_{n-j}\label{Eq.tilde_epsilon_n}\\
	&\quad+\left(\frac{3n-7}{2}k_1-\frac32\al k_2(n-1)\right)\sum_{j=4}^{n-5}a_ja_{n-1-j}+\frac{3n-10}{2}k_2\sum_{j=3}^{n-5}a_ja_{n-2-j}\nonumber
\end{align}

In \eqref{Eq.a_n_new_recurrence}, we note that the coefficients of $a_{n-3}$ and $a_{n-4}$ are independent of $n$, which indicates that the term $s_1a_{n-3}+s_2a_{n-4}+\wt\varepsilon_n$ may be viewed as a perturbation for large $n$. Therefore, we expect that $\{a_n\}$ can be modeled by the simpler sequence $\{M_n>0\}_{n=0}^{\lceil A^2\rceil-1}$ defined by
\begin{equation}\label{Eq.M_n_sequence}
	M_0=1,\quad M_1=a_1>0,\quad \gamma_nM_n=2p_nM_{n-1}+q_nM_{n-2}\quad\forall\ n\in\Z\cap[2,A^2).
\end{equation}

\begin{remark}
	We emphasize that the sequence $\{M_n\}$ has only finitely many terms. This is because $\g_n>0$ holds only for $n<A^2$. The positivity of $M_n$ for $n<A^2$ is ensured by $M_0>1$, $M_1>0$, $p_n>0$, $q_n>0$ and $\g_n>0$ (for sufficiently large $A$).
\end{remark}

\begin{remark}\label{Rmk.gamma_p_q_t}
	We note that $n\mapsto \gamma_n$, $n\mapsto p_n$ and $n\mapsto q_n$ are linear functions. We can define
	\[\gamma_t:=\frac8{(A+1)^2}(A^2-t)=\g_{(1)}t+\g_{(2)},\quad p_t:=p_{(1)}t+q_{(1)},\quad q_t:=q_{(1)}t+q_{(2)},\quad\forall\ t\in\R.\]
\end{remark}

\subsection{Upper bound of $\{a_n\}$}\label{Subsec.a_n_upper_bound}
In this subsection, we prove the ``$\lesssim$" part of Proposition \ref{Prop.a_n_sharp_est}. We remark that, in the rest of this section, all the implicit constants in $\lesssim, \gtrsim, \asymp$ are independent of the index $n$ and the parameter $A$. The main result of this subsection is the following proposition.

\begin{proposition}\label{Prop.a_n_upper_bound}
	There exist $A_0>10$ and $C>0$ such that for all $A>A_0$ we have
	\begin{equation}\label{Eq.a_n_upper_bound}
		|a_n|\leq CM_n,\quad\forall\ n\in\Z\cap[1, A^2).
	\end{equation}
\end{proposition}

We first illustrate that for each $n$, $a_n$ is bounded with respect to $A$.

\begin{lemma}\label{Lem.a_n_trivial_bound}
	For each $n\in\Z_{\geq 0}$, there exists $A(n)>0$ such that
	\begin{equation*}
		|a_n|\lesssim_n 1,\quad\forall\ A>A(n).
	\end{equation*}
\end{lemma}
\begin{proof}
	We consider the sequence $\{a_n^\infty\}_{n=0}^\infty$ defined by $a_0^\infty=1$, $a_1^\infty=2$, and
	\begin{equation}\label{Eq.a_n^infty}
		8a_n^\infty=(5-n)a_{n-1}^\infty-\left(n-\frac{16}{3}\right)a_{n-2}^\infty-\frac32(n+1)\sum_{j=2}^{n-1}a_j^\infty a_{n+1-j}^\infty+\left(\frac32n-2\right)\sum_{j=1}^{n-1}a_j^\infty a_{n-j}^\infty
	\end{equation}
	for $n\in\Z_{\geq 2}$. Since $a_0=1$,
	\[\lim_{A\to+\infty}a_1=\lim_{A\to+\infty}\frac{6A+2}{3A-3}=2,\]
	$\lim_{A\to+\infty}\al=1$ and $\lim_{A\to+\infty}\g_n=8$ for each $n\in\Z_{\geq 0}$, one can verify by induction that
	\begin{equation}\label{Eq.a_n_limit}
		\lim_{A\to+\infty}a_n=a_n^\infty\in\R,\quad\forall\ n\in\Z_{\geq 0}.
	\end{equation}
	This completes the proof.
\end{proof}
\begin{remark}\label{Rmk.a_n_positive}
	We compute that
	\begin{align*}
		&a_0^\infty=1,\ a_1^\infty=2,\ a_2^\infty=\frac53,\ a_3^\infty=1,\ a_4^\infty=\frac23,\ a_5^\infty=\frac{13}{24}, \\
		a_6^\infty&=\frac{17}{36},\ a_7^\infty=\frac{11}{24},\ a_8^\infty=\frac{97}{216},\ a_9^\infty=\frac{6683}{13824},\ a_{10}^\infty=\frac{10547}{20736}.
	\end{align*}
	We remark that $a_n^\infty>0$ holds for $n\in\Z\cap[0,10]$.
\end{remark}

In \eqref{Eq.a_n_new_recurrence}, we note that the coefficient of $a_{n-1}$, which is $2p_n$, becomes zero if $A=+\infty$, due to \eqref{Eq.p_n_asymp}. Hence, we consider an auxiliary sequence $a_n+\la_na_{n-1}$, for which the degeneration disappears. We first construct $\la_n$.

\begin{lemma}\label{Lem.la_n_def}
	For each $n\in\Z\cap[0, A^2)$, we define
	\begin{equation}\label{Eq.mu_n_la_n}
		\mu_n^*:=\sqrt{\gamma_{n+\frac12}q_{n+\frac32}+p_{n+1}^2},\quad \mu_n:=\mu_n^*+p_{n+\frac12},\quad \la_n:=\frac{q_{n+1}}{\mu_n}.
	\end{equation}
	Then there exists $A_0>10$ such that the following properties hold for all $A>A_0$ and all $n\in\Z_+\cap [1, A^2)$:
	\begin{align}
		\mu_n^*\asymp \sqrt n,\quad \mu_n^*>&\ \mu_{n-1}^*,\quad0<\mu_n^*-\mu_{n-1}^*\asymp\frac1{\sqrt n},\quad\la_n\asymp\sqrt n,\quad \mu_n\asymp \sqrt n,\label{Eq.mu_n^*}\\
		&0\leq\g_nq_{n+1}+p_{n+\frac12}^2-\mu_n^*\mu_{n-1}^*\leq \frac Cn,\label{Eq.coeffi_error}\\
		1\leq&\ \frac{2p_n+\g_n\la_n}{\mu_{n-1}}\leq 1+C\left(\frac 1{n^2}+\frac1{\sqrt nA}\right),\label{Eq.quotient_est}
	\end{align}
	where $C>0$ is an absolute constant independent of $A>A_0$ and $n$.
\end{lemma}
\begin{proof}
	Recall $\g_n=\g_{(1)}n+\g_{(2)}$, $p_n=p_{(1)}n+p_{(2)}$, $q_n=q_{(1)}n+q_{(2)}$, Remark \ref{Rmk.gamma_p_q_t} and the definition \eqref{Eq.mu_n_la_n}, one computes that
	\begin{equation}\label{Eq.mu_n^*_expression}
		\left(\mu_n^*\right)^2=\g_{n+\frac12}q_{n+\frac32}+p_{n+1}^2=\mu_{(2)}n^2+\mu_{(1)}n+\mu_{(0)},
	\end{equation}
	where
	\begin{align*}
		\mu_{(2)}:&=\g_{(1)}q_{(1)}+p_{(1)}^2,\\
		\mu_{(1)}:&=\left(\frac12\g_{(1)}+\g_{(2)}\right)q_{(1)}+\g_{(1)}\left(\frac32q_{(1)}+q_{(2)}\right)+2p_{(1)}\left(p_{(1)}+p_{(2)}\right),\\
		\mu_{(0)}:&=\left(\frac12\g_{(1)}+\g_{(2)}\right)\left(\frac32q_{(1)}+q_{(2)}\right)+\left(p_{(1)}+p_{(2)}\right)^2.
	\end{align*}
	We have 
	\begin{align*}
		\mu_{(2)}=\frac1{A^2}%-\frac{132}{A^5}
+\cO(A^{-3}),\quad \mu_{(1)}=8%+\frac{8}{3A^2}
+\cO(A^{-1}),\quad \mu_{(0)}=\frac43%+\frac{80}{A}
+\cO(A^{-1})
	\end{align*}
	as $A\to+\infty$. As a result, for $A$ sufficiently large, we have
	\begin{equation}\label{Eq.mu_n^*_est}
		7n<\mu_{(1)}n<\left(\mu_n^*\right)^2<\frac{2n^2}{A^2}+9n+2\leq 11n+2\leq 13n,\quad\forall\ n\in\Z_+\cap[1, A^2),
	\end{equation}
	hence $\mu_n^*$ is well-defined and $\mu_n^*\asymp\sqrt n$ for $n\in\Z_+\cap[1, A^2)$. It follows from $\mu_{(2)}>0$ and $\mu_{(1)}>0$ that $n\mapsto\left(\mu_n^*\right)^2$ is strictly increasing, hence $\mu_n^*>\mu_{n-1}^*$ for each $n\in\Z_+\cap[1, A^2)$. Moreover, since
	\begin{align*}
		 \mu_n^*-\mu_{n-1}^*=\frac{\left(\mu_n^*\right)^2-\left(\mu_{n-1}^*\right)^2}{\mu_n^*+\mu_{n-1}^*}=\frac{(2n-1)\mu_{(2)}+\mu_{(1)}}{\mu_n^*+\mu_{n-1}^*},
	\end{align*}
	for $n\in\Z_+\cap[1, A^2)$ we have
	\begin{align*}
		\mu_n^*-\mu_{n-1}^*<\frac{2n\mu_{(2)}+\mu_{(1)}}{\mu_n^*}<\frac{{3n}/{A^2}+9}{\sqrt{7n}}\leq \frac{12}{\sqrt{7n}},\quad
		\mu_n^*-\mu_{n-1}^*>\frac{\mu_{(1)}}{2\mu_n^*}>\frac{7}{2\sqrt{13n}}.
	\end{align*}
	To complete the proof of \eqref{Eq.mu_n^*}, it remains to check that $\la_n\asymp \sqrt n$ and $\mu_n\asymp \sqrt n$. By definition,
	\[\mu_n=%\frac{q_{n+1}}
{\mu_n^*+p_{n+\frac12}}=%\frac{q_{(1)}(n+1)+q_{(2)}}
{\mu_n^*+p_{(1)}\left(n+1/2\right)+p_{(2)}}.\]
	Using \eqref{Eq.p_n_asymp}, \eqref{Eq.q_n_asymp}, if $A$ is sufficiently large, for $n\in\Z\cap[1,A^2)$, on one hand, we have
	\begin{align*}
		\mu_n\geq \mu_n^*\geq \sqrt{7n},
	\end{align*}
	on the other hand, we have
	\begin{align*}
		\mu_n\leq %\frac{\frac56(n+1)-\frac32}\frac{n/6}
{\sqrt{13n}+(4/A)\left(n+1/2\right)}\leq {(\sqrt{13}+6)\sqrt n}.%\gtrsim \sqrt n.
	\end{align*}
	This proves $\mu_n\asymp\sqrt n$ for $n\in\Z\cap[1,A^2)$. Then $\la_n\asymp \sqrt n$ follows from $q_{n+1}\asymp n$, $\la_n=\frac{q_{n+1}}{\mu_n}$ and $\mu_n\asymp\sqrt n$, and hence \eqref{Eq.mu_n^*} holds.
	
	To prove \eqref{Eq.coeffi_error}, recalling the definition \eqref{Eq.mu_n_la_n}, we first show that
	\begin{equation}\label{Eq.coeffi_error_claim}
		 \left(\g_nq_{n+1}+p_{n+\frac12}^2\right)^2-\left(\g_{n+\frac12}q_{n+\frac32}+p_{n+1}^2\right)\left(\g_{n-\frac12}q_{n+\frac12}+p_n^2\right)\geq0,
\quad\forall\ n\in\Z_+\cap[1, A^2).
	\end{equation}
	Indeed, by \eqref{Eq.mu_n^*_expression}, we can compute
	\begin{align*}&\g_nq_{n+1}+p_{n+\frac12}^2=\mu_{(2)}(n-1/2)^2+\mu_{(1)}(n-1/2)+\mu_{(0)},\\ &\g_{n-\frac12}q_{n+\frac12}+p_n^2=\mu_{(2)}(n-1)^2+\mu_{(1)}(n-1)+\mu_{(0)}.\end{align*}
	Hence,
	\begin{align*}
		 \left(\g_nq_{n+1}+p_{n+\frac12}^2\right)^2-\left(\g_{n+\frac12}q_{n+\frac32}+p_{n+1}^2\right)\left(\g_{n-\frac12}q_{n+\frac12}+p_n^2\right)
=\Delta_{(2)}n^2+\Delta_{(1)}n+\Delta_{(0)},
	\end{align*}
	where
	\begin{align*}
		&\Delta_{(2)}=\mu_{(2)}^2/2,\quad %2\mu_{(02)}\mu_{(00)}+\mu_{(01)}^2-\mu_{(2)}\mu_{(10)}-\mu_{(1)}\mu_{(11)}-\mu_{(0)}\mu_{(12)},\\
		\Delta_{(1)}=\left(\mu_{(1)}-\mu_{(2)}\right)\mu_{(2)}/2,
%2\mu_{(01)}\mu_{(00)}-\mu_{(1)}\mu_{(10)}-\mu_{(0)}\mu_{(11)},
\quad\Delta_{(0)}=\left(\mu_{(1)}/2-\mu_{(2)}/4\right)^2-\mu_{(0)}\mu_{(2)}/2.%\mu_{(00)}^2-\mu_{(0)}\mu_{(10)}.
	\end{align*}
	We find that
	\begin{align*}
		\Delta_{(2)}=\frac1{2A^4}%-\frac{132}{A^7}
+\cO(A^{-5}),\quad \Delta_{(1)}=\frac4{A^2}%+\frac{5}{6A^4}
+\cO(A^{-3}),\quad \Delta_{(0)}=16%+\frac8{A^2}
+\cO(A^{-1})
	\end{align*}
	as $A\to +\infty$. Thus, \eqref{Eq.coeffi_error_claim} holds if $A$ is sufficiently large. Then \eqref{Eq.mu_n_la_n} and \eqref{Eq.coeffi_error_claim} imply
	\begin{equation*}
		\g_nq_{n+1}+p_{n+\frac12}^2-\mu_n^*\mu_{n-1}^*\geq 0,\quad \forall\ n\in\Z_+\cap[1, A^2).
	\end{equation*}
	This is the left side of \eqref{Eq.coeffi_error}. The other side is a consequence of
	\begin{align*}
		 \g_nq_{n+1}+p_{n+\frac12}^2-\mu_n^*\mu_{n-1}^*&=\frac{\left(\g_nq_{n+1}+p_{n+\frac12}^2\right)^2-\left(\mu_n^*\mu_{n-1}^*\right)^2}{\g_nq_{n+1}+p_{n+\frac12}^2+\mu_n^*\mu_{n-1}^*}\overset{(*)}{\lesssim}\frac{\Delta_{(2)}n^2+\Delta_{(1)}n+\Delta_{(0)}}{\left(\mu_n^*\right)^2}\\
		&\lesssim \frac{\frac{n^2}{A^4}+\frac{5n}{A^2}+17}{n}\lesssim \frac1n
	\end{align*}
	for all $n\in\Z_+\cap[1, A^2)$, where in $(*)$ we have used $\mu_{n-1}^*\asymp \mu_n^*$, which follows from $\mu_n^*\asymp\sqrt n$ for all $n\in\Z_+\cap[1, A^2)$ and $\mu_0^*=\sqrt{\mu_{(0)}}\asymp1$.
	
	Finally, we prove \eqref{Eq.quotient_est}. By Remark \eqref{Rmk.gamma_p_q_t}, we have $p_{n-\frac12}+p_{n+\frac12}=2p_n$. Hence, \eqref{Eq.mu_n_la_n} implies that $\mu_{n-1}^*=\mu_{n-1}-p_{n-\frac12}=\mu_{n-1}-\left(2p_n-p_{n+\frac12}\right)$, thus,
	\begin{align*}
		 \g_nq_{n+1}+p_{n+\frac12}^2-\mu_n^*\mu_{n-1}^*&=\g_nq_{n+1}+p_{n+\frac12}^2-\left(\mu_n-p_{n+\frac12}\right)\left(\mu_{n-1}-2p_n+p_{n+\frac12}\right)\\
		&=\g_nq_{n+1}-\mu_n\left(\mu_{n-1}-2p_n\right)+p_{n+\frac12}\left(\mu_{n-1}^*-\mu_n^*\right).
	\end{align*}
	Then it follows from \eqref{Eq.mu_n^*}, \eqref{Eq.coeffi_error} and \eqref{Eq.p_n}, \eqref{Eq.p_n_asymp} that
	\begin{align*}
		0\leq p_{n+\frac12}\left(\mu_n^*-\mu_{n-1}^*\right)&\leq \g_nq_{n+1}-\mu_n\left(\mu_{n-1}-2p_n\right)\leq \frac{C}{n}+p_{n+\frac12}\left(\mu_n^*-\mu_{n-1}^*\right)\\
		&\leq \frac Cn+C\frac nA\frac1{\sqrt n}= C\left(\frac1n+\frac{\sqrt n}A\right)
	\end{align*}
	for all $n\in\Z_+\cap [1, A^2)$. Dividing both sides of the above inequality by $\mu_n\mu_{n-1}$, we obtain
	\begin{align*}
		0\leq \frac1{\mu_{n-1}}\left(2p_n+\frac{\g_nq_{n+1}}{\mu_n}\right)-1&\leq \frac{C}{\mu_n\mu_{n-1}}\left(\frac1n+\frac{\sqrt n}A\right)\leq \frac{C}{\mu_n^*\mu_{n-1}^*}\left(\frac1n+\frac{\sqrt n}A\right)\\
		&\leq\frac{C}{\left(\mu_n^*\right)^2}\left(\frac1n+\frac{\sqrt n}A\right)\leq \frac Cn\left(\frac1n+\frac{\sqrt n}A\right),
	\end{align*}
	where we have used $\mu_n>\mu_n^*$ (due to \eqref{Eq.mu_n_la_n}), $\mu_n^*\asymp\mu_{n-1}^*$ and \eqref{Eq.mu_n^*}. At this stage, \eqref{Eq.quotient_est} follows from the definition of $\la_n$ in \eqref{Eq.mu_n_la_n}.
\end{proof}

\begin{lemma}\label{Lem.la_n_increasing}
	There exists $A_0>10$ such that
	\begin{equation}\label{Eq.la_n_increasing}
		0\leq \la_n-\la_{n-1}\asymp \frac1{\sqrt n},\quad\forall\ n\in\Z\cap[2, A^2),
	\end{equation}
	for all $A>A_0$.
\end{lemma}
\begin{proof}
	By the definition, we have
	\begin{align}\label{Eq.la_n-la_n-1}
		\la_n-\la_{n-1}=\frac{q_{n+1}}{\mu_n}-\frac{q_n}{\mu_{n-1}}=\frac{q_{n+1}\mu_{n-1}-q_n\mu_n}{\mu_n\mu_{n-1}}.
	\end{align}
	We also have
	\begin{align*}
		q_{n+1}\mu_{n-1}-q_n\mu_n&=q_{n+1}\left(\mu_{n-1}^*+p_{n-\frac12}\right)-q_n\left(\mu_n^*+p_{n+\frac12}\right)\\
		&=\frac{q_{n+1}^2\left(\mu_{n-1}^*\right)^2-q_n^2\left(\mu_n^*\right)^2}{q_{n+1}\mu_{n-1}^*+q_n\mu_n^*}+q_{n+1}p_{n-\frac12}-q_np_{n+\frac12}.
	\end{align*}
	We compute that
	\begin{align*}
		 &q_{n+1}^2\left(\mu_{n-1}^*\right)^2-q_n^2\left(\mu_n^*\right)^2=q_{n+1}^2\left(\g_{n-\frac12}q_{n+\frac12}+p_n^2\right)-q_n^2\left(\g_{n+\frac12}q_{n+\frac32}+p_{n+1}^2\right)\\
		 =&\left(q_{(1)}n+q_{(1)}+q_{(2)}\right)^2\left(\mu_{(2)}(n-1)^2+\mu_{(1)}(n-1)+\mu_{(0)}\right)\\
&-\left(q_{(1)}n+q_{(2)}\right)^2\left(\mu_{(2)}n^2+\mu_{(1)}n+\mu_{(0)}\right)
		=\wt\Delta_{(2)}n^2+\wt\Delta_{(1)}n+\wt\Delta_{(0)},
	\end{align*}
	where
	\begin{align*}
		\wt\Delta_{(2)}:&=q_{(1)}^2(\mu_{(1)}-2\mu_{(2)})-2q_{(1)}q_{(2)}\mu_{(2)}
%+2q_{(1)}\left(q_{(1)}+q_{(2)}\right)\mu_{(11)}+\left(q_{(1)}+q_{(2)}\right)^2\mu_{(12)}\\
%		&\qquad-q_{(1)}^2\mu_{(0)}-2q_{(1)}q_{(2)}\mu_{(1)}-q_{(2)}^2\mu_{(2)}\\+\frac{32}A
		=8+\cO(A^{-1}),\\
		 \wt\Delta_{(1)}:&=q_{(1)}^2(2\mu_{(0)}-\mu_{(1)})-2q_{(2)}\left(q_{(1)}+q_{(2)}\right)\mu_{(2)}
%+\left(q_{(1)}+q_{(2)}\right)^2\mu_{(11)}-2q_{(1)}q_{(2)}\mu_{(0)}-q_{(2)}^2\mu_{(1)}\\
		=-\frac{16}3+\cO(A^{-1}),\\
		\wt\Delta_{(0)}:&=\left(q_{(1)}+q_{(2)}\right)^2(\mu_{(0)}-\mu_{(1)}+\mu_{(2)})-q_{(2)}^2\mu_{(0)}=-\frac{28}9+\cO(A^{-1}),
	\end{align*}
	as $A\to+\infty$; and
	\begin{align*}
		 q_{n+1}p_{n-\frac12}-q_np_{n+\frac12}&=\left(q_{(1)}+q_{(2)}\right)\left(-\frac12p_{(1)}+p_{(2)}\right)-q_{(2)}\left(\frac12p_{(1)}+p_{(2)}\right)\\
		&=\frac{5}{6A}+\cO(A^{-2}),\quad\text{as}\quad A\to+\infty.
	\end{align*}
	Then it follows from \eqref{Eq.mu_n^*} and $q_n\asymp q_{n+1}\asymp n$ that
	\begin{align}\label{Eq.la_n_diff_fenzi}
		0\leq q_{n+1}\mu_{n-1}-q_n\mu_n\asymp\sqrt{n},\quad\forall\ n\in\Z\cap[2, A^2),\ \forall\ A>A_0,
	\end{align}
	by taking $A_0>10$  large enough. Now, \eqref{Eq.la_n_increasing} is a consequence of \eqref{Eq.la_n-la_n-1}, \eqref{Eq.la_n_diff_fenzi} and $\mu_n\asymp\sqrt n$.
\end{proof}

\begin{lemma}\label{Lem.tildeM_nest}
	Define the sequence $\left\{\wt M_n>0\right\}_{n=1}^{\lceil A^2\rceil-1}$ by
	\begin{equation}\label{Eq.tilde_M_n}
		\wt M_n:=M_n+\la_nM_{n-1},\quad\forall\ n\in\Z_+,
	\end{equation}
	where the sequence $\{M_n\}_{n=0}^\infty$ is given by \eqref{Eq.M_n_sequence}. Then there exists $A_0>10$ such that for all $A>A_0$, we have
	\begin{equation}\label{Eq.tildeM_n_quotient}
		\frac{\g_n\wt M_n}{\wt M_{n-1}}\in\left[\mu_{n-1},2p_n+\g_n\la_n\right],\quad\forall\ n\in\Z\cap[2, A^2),
	\end{equation}
	and
	\begin{equation}\label{Eq.tildeM_n_est}
		\wt M_n\asymp\prod_{j=2}^{n}\frac{\mu_{j-1}}{\g_j}\asymp\prod_{j=2}^n\frac{2p_j+\g_j\la_j}{\g_j},\quad\forall\ n\in\Z\cap[1, A^2).
	\end{equation}
\end{lemma}
\begin{proof}
	It follows from \eqref{Eq.M_n_sequence}, \eqref{Eq.tilde_M_n} and \eqref{Eq.mu_n_la_n} that
	 \begin{align}
	 	\g_n\wt M_n&=(2p_n+\g_n\la_n)\wt M_{n-1}+\left[q_n-\la_{n-1}(2p_n+\g_n\la_n)\right]M_{n-2}\nonumber\\
	 	&=(2p_n+\g_n\la_n)\wt M_{n-1}+\la_{n-1}\left[\mu_{n-1}-(2p_n+\g_n\la_n)\right]M_{n-2},\quad\forall\ n\in\Z_+\cap[1,A^2).\label{Eq.tilde_M_n_recurrence}
	 \end{align}
	 By \eqref{Eq.quotient_est}, we have $\mu_{n-1}-(2p_n+\g_n\la_n)\leq0$, hence,
	 \begin{equation}\label{Eq.tildeM_n_quotient1}
	 	\frac{\g_n\wt M_n}{\wt M_{n-1}}\leq2p_n+\g_n\la_n,\quad\forall\ n\in\Z_+\cap[1,A^2).
	 \end{equation}
	 On the other hand, \eqref{Eq.tilde_M_n_recurrence} and $\mu_{n-1}-(2p_n+\g_n\la_n)\leq0$ imply that
	 \begin{align*}
	 	\g_n\wt M_n&=(2p_n+\g_n\la_n)\wt M_{n-1}+\la_{n-1}\left[\mu_{n-1}-(2p_n+\g_n\la_n)\right]M_{n-2}\\
	 	&=\mu_{n-1}\wt M_{n-1}+(2p_n+\g_n\la_n-\mu_{n-1})\left(\wt M_{n-1}-\la_{n-1}M_{n-2}\right)\\
	 	&=\mu_{n-1}\wt M_{n-1}+(2p_n+\g_n\la_n-\mu_{n-1})M_{n-1}\\
	 	&\geq \mu_{n-1}\wt M_{n-1}.
	 \end{align*}
	 This, combining with \eqref{Eq.tildeM_n_quotient1}, gives \eqref{Eq.tildeM_n_quotient}.
	
	 As a consequence, we obtain
	 \[\wt M_1\prod_{j=2}^n\frac{\mu_{j-1}}{\gamma_j}\leq\wt M_n\leq \wt M_1\prod_{j=2}^n\frac{2p_j+\g_j\la_j}{\g_j},\quad\forall\ n\in\Z\cap[2,A^2).\]
	 Moreover, thanks to \eqref{Eq.quotient_est}, we have
	 \begin{align*}
	 	1\leq\frac{\prod_{j=2}^n\frac{2p_j+\g_j\la_j}{\g_j}}{\prod_{j=2}^n\frac{\mu_{j-1}}{\gamma_j}}=\prod_{j=2}^n\frac{2p_j+\g_j\la_j}{\mu_{j-1}}\leq \prod_{j=2}^{\lceil A^{2}\rceil-1}\left(1+\frac{C}{j^2}\right)\cdot\prod_{j=2}^{\lceil A^{2}\rceil-1}\left(1+\frac{C}{\sqrt{j}A}\right)\lesssim 1,
	 \end{align*}
     where we have used 
     $$1<\prod_{j=2}^\infty(1+Cj^{-2})<+\infty,\quad 1<\sup_{A>0}\prod_{j=2}^{\lceil A^{2}\rceil-1}\left(1+\frac C{\sqrt jA}\right)<+\infty.$$ Therefore, \eqref{Eq.tildeM_n_est} follows, noting that $\wt M_1\asymp 1$.
\end{proof}

\begin{lemma}\label{Lem.quotientM_n}
	There exist $A_0>10$ and $\delta_0>0$ such that
	\begin{equation}\label{Eq.quotient_M_n}
		\frac{M_n}{M_{n-1}}\geq \delta_0\sqrt n,\quad\forall\ n\in\Z\cap[1,A^2),\quad\forall\ A>A_0.
	\end{equation}
\end{lemma}
\begin{proof}
	Since $\lim_{A\to+\infty}M_1/M_0=\lim_{A\to+\infty}a_1=2>0$, \eqref{Eq.quotient_M_n} holds for $n=1$. By \eqref{Eq.M_n_sequence}, we have
	\begin{align*}
		\frac{M_2}{M_1}=\frac{2p_2+\frac{q_2}{a_1}}{\g_2}\longrightarrow \frac{0+(2/3)/2}{8}=\frac1{24}>0\quad\text{as}\quad A\to+\infty.
	\end{align*}
	Thus, there exists $A_0>10$ such that
	\[\frac{M_1}{M_0}>\delta_0:=\frac1{25\sqrt2},\quad \frac{M_2}{M_1}>\delta_0\sqrt2,\quad\forall\ A>A_0.\]
	
	By \eqref{Eq.p_n_asymp} and \eqref{Eq.q_n_asymp}, we have
	\begin{align*}
		\frac{p_{(2)}}{p_{(1)}}=-\frac59+\cO(A^{-1}),\quad \frac{q_{(2)}}{q_{(1)}}=-\frac43+\cO(A^{-1})
	\end{align*}
	Hence, by taking $A_0>10$ to be larger if necessary, there holds
	\begin{equation}\label{Eq.p_2/p_1}
		\frac{p_{(2)}}{p_{(1)}}<-\frac12,\quad \frac{q_{(2)}}{q_{(1)}}<-1,\quad\forall\ A>A_0.
	\end{equation}
	
	Next we use the induction to prove \eqref{Eq.quotient_M_n}. Let $A>A_0$. We assume that $n\in\Z_{\geq 3}\cap(1,A^2)$ and $M_j/M_{j-1}>\delta_0\sqrt j$ holds for all $j\in\Z\cap[1,n-1]$. Since
	\[\g_nM_{n}=2p_nM_{n-1}+q_nM_{n-2},\quad \g_{n-1}M_{n-1}=2p_{n-1}M_{n-2}+q_{n-1}M_{n-3},\]
	we have
	\begin{align*}
		\frac{\g_n}{\g_{n-1}}\frac{M_n}{M_{n-1}}\geq\min\left\{\frac{p_n}{p_{n-1}}\frac{M_{n-1}}{M_{n-2}}, \frac{q_n}{q_{n-1}}\frac{M_{n-2}}{M_{n-3}}\right\}\geq\delta_0\min\left\{\frac{p_n}{p_{n-1}}\sqrt{n-1}, \frac{q_n}{q_{n-1}}\sqrt{n-2}\right\}.
	\end{align*}
	It follows from \eqref{Eq.p_n}, \eqref{Eq.q_n} and \eqref{Eq.p_2/p_1} that
	\begin{align*}
		\frac{p_n}{p_{n-1}}&=\frac{p_{(1)}n+p_{(2)}}{p_{(1)}n-p_{(1)}+p_{(2)}}=1+\frac1{n+\frac{p_{(2)}}{p_{(1)}}-1}>1+\frac{1}{n-\frac32},\\
		\frac{q_n}{q_{n-1}}&=1+\frac1{n+\frac{q_{(2)}}{q_{(1)}}-1}>1+\frac1{n-2}.
	\end{align*}
	Noting that $j\mapsto \g_j=\frac8{(A+1)^2}(A^2-j)>0$ is decreasing for $j<A^2$, we have $\g_n<\g_{n-1}$, then
	\begin{align*}
		\frac{M_n}{M_{n-1}}>\frac{\g_n}{\g_{n-1}}\frac{M_n}{M_{n-1}}\geq \delta_0\min\left\{\left(1+\frac1{n-3/2}\right)\sqrt{n-1}, \left(1+\frac1{n-2}\right)\sqrt{n-2}\right\}.
	\end{align*}
	To complete the induction, it suffices to prove that
	\begin{align*}
		\left(1+\frac1{n-3/2}\right)\sqrt{n-1}>\sqrt n\quad\text{and}\quad \left(1+\frac1{n-2}\right)\sqrt{n-2}>\sqrt n,\quad\forall\ n\in\Z_{\geq 3}.
	\end{align*}
	This is a direct consequence of
	\begin{align*}
		\sqrt{\frac{n}{n-1}}&=\sqrt{1+\frac1{n-1}}<1+\frac1{n-1}<1+\frac1{n-3/2},\\
		&\sqrt{\frac{n}{n-2}}=\sqrt{1+\frac2{n-2}}<1+\frac1{n-2}
	\end{align*}
	for all $n\in\Z_{\geq 3}$.
\end{proof}

\begin{lemma}\label{Lem.M_n_tildeM_n}
	There exists $A_0>10$ such that for all $A>A_0$, we have
	\begin{equation}
		M_n\asymp\wt M_n,\quad\forall\ n\in\Z\cap[1, A^2).
	\end{equation}
\end{lemma}
\begin{proof}
	It is trivial from \eqref{Eq.tilde_M_n} that $M_n<\wt M_{n}$ for $n\in\Z\cap[1,A^2)$. On the other hand, by Lemma \ref{Lem.quotientM_n} and \eqref{Eq.mu_n^*}, we have
	\begin{align*}
		\frac{\wt M_n}{M_n}=1+\la_n\frac{M_{n-1}}{M_n}\leq 1+\frac{C\sqrt n}{\delta_0\sqrt n}=1+\frac{C}{\delta_0},\quad\forall\ n\in\Z\cap[1, A^2).
	\end{align*}
	This completes the proof.
\end{proof}

Now we are in a position to prove Proposition \ref{Prop.a_n_upper_bound}.

\begin{proof}[Proof of Proposition \ref{Prop.a_n_upper_bound}]
	\underline{\textbf{Step 1.}} By \eqref{Eq.a_n_new_recurrence}, \eqref{Eq.k_1k_2asymp}, \eqref{Eq.s_1asymp}, \eqref{Eq.s_2asymp} and \eqref{Eq.tilde_epsilon_n}, there exist $A_0>10$ and $C_0>0$ such that for $A>A_0$, we have
	\begin{align*}
		\g_n|a_n|&\leq 2p_n|a_{n-1}|+q_n|a_{n-2}|+C_0\left(|a_{n-3}|+|a_{n-4}|\right)\\
		 &\quad+C_0n\left(\sum_{j=6}^{n-5}|a_j||a_{n+1-j}|+\sum_{j=5}^{n-5}|a_j||a_{n-j}|+\sum_{j=4}^{n-5}|a_j||a_{n-1-j}|+\sum_{j=3}^{n-5}|a_j||a_{n-2-j}|\right)
	\end{align*}
	for all $n\in\Z\cap[10, A^2)$. Let
	\begin{equation}\label{Eq.tilde_a_n}
		\wt a_n:=|a_n|+\la_n|a_{n-1}|\geq |a_n|,\quad\forall\ n\in\Z\cap[1,A^2).
	\end{equation}
	Then
	\begin{align*}
		\g_n\wt a_n&\leq (2p_n+\g_n\la_n)\wt a_{n-1}+\left[q_n-\la_{n-1}(2p_n+\g_n\la_n)\right]|a_{n-2}|+C_0\left(|a_{n-3}|+|a_{n-4}|\right)\\
		 &\quad+C_0n\left(\sum_{j=6}^{n-5}|a_j||a_{n+1-j}|+\sum_{j=5}^{n-5}|a_j||a_{n-j}|+\sum_{j=4}^{n-5}|a_j||a_{n-1-j}|+\sum_{j=3}^{n-5}|a_j||a_{n-2-j}|\right)
	\end{align*}
	for all $n\in\Z\cap[10, A^2)$. Using \eqref{Eq.tilde_a_n} and
	\[q_n-\la_{n-1}(2p_n+\g_n\la_n)=\la_{n-1}\left[\mu_{n-1}-(2p_n+\g_n\la_n)\right]\leq 0,\quad\forall\ n\in\Z\cap[1, A^2),\]
	which follows from \eqref{Eq.quotient_est}, we obtain
	\begin{align}
		\g_n\wt a_n&\leq (2p_n+\g_n\la_n)\wt a_{n-1}+C_0\left(\wt a_{n-3}+\wt a_{n-4}\right)\nonumber\\
		&\quad+C_0n\left(\sum_{j=6}^{n-5}\wt a_j\wt a_{n+1-j}+\sum_{j=5}^{n-5}\wt a_j\wt a_{n-j}+\sum_{j=4}^{n-5}\wt a_j\wt a_{n-1-j}+\sum_{j=3}^{n-5}\wt a_j\wt a_{n-2-j}\right)\label{Eq.tilde_a_n_recurrence}
	\end{align}
	for all $n\in\Z\cap[10, A^2)$.
	
	\underline{\textbf{Step 2.}} We consider a sequence $\{M_n^*>0\}_{n=1}^{\lceil A^2\rceil-1}$, which is defined by $M_1^*:=\wt a_1>0$, $M_2^*:=\wt a_2+\wt a_1>M_1^*>0$, and
	\begin{equation}\label{Eq.M_n^*small}
		M_n^*:=\max\left\{\wt a_n,\frac{\left(M_{n-1}^*\right)^2}{M_{n-2}^*}\right\},\quad\forall\ n\in[3,9],
	\end{equation}
	and for $n\in\Z\cap[10, A^2)$,
	\begin{equation}\label{Eq.M_n^*large}
		M_n^*:=\max\left\{a_n^*, \frac{\left(M_{n-1}^*\right)^2}{M_{n-2}^*}\right\}
	\end{equation}
	where
	\begin{align}
		\g_na_n^*:&=(2p_n+\g_n\la_n)M_{n-1}^*+C_0\left(M_{n-3}^*+M_{n-4}^*\right)\nonumber\\
		 &\quad+C_0n\left(\sum_{j=6}^{n-5}M_j^*M_{n+1-j}^*+\sum_{j=5}^{n-5}M_j^*M_{n-j}^*
+\sum_{j=4}^{n-5}M_j^*M_{n-1-j}^*+\sum_{j=3}^{n-5}M_j^*M_{n-2-j}^*\right),
\label{Eq.M_n^*_recurrence}
	\end{align}
	with the constant $C_0>0$ the same as in \eqref{Eq.tilde_a_n_recurrence}.
	Then we have (thanks to $\g_n>0$ for $n<A^2$)
	\begin{equation}\label{Eq.a_n_bound_byM_n^*}
		|a_n|\leq \wt a_n\leq M_n^*,\quad\forall\ n\in\Z\cap[1, A^2);
	\end{equation}
	and
	\begin{equation}\label{Eq.quotient_M_n^*_est}
		(2p_n+\g_n\la_n)M_{n-1}^*\leq \g_na_n^*\leq \g_nM_n^*\Longrightarrow \frac{M_{n-1}^*}{M_n^*}\leq \frac1{\wt b_n},\quad\forall\ n\in\Z\cap[10, A^2),
	\end{equation}
	where (using \eqref{Eq.mu_n^*})
	\begin{equation}\label{Eq.tilde_b_n}
		\wt b_n:=\frac{2p_n+\g_n\la_n}{\g_n}\geq \la_n\gtrsim \sqrt n>0, \quad\forall\ n\in\Z\cap[1,A^2).
	\end{equation}
	We claim that
	\begin{equation}\label{Eq.M_n^*claim}
		M_{i}^*M_{k-i}^*\leq M_{j}^*M_{k-j}^*\quad\text{for}\quad i\in\Z\cap[j,k-j], j\in\Z\cap[1,k/2], k\in\Z\cap[2,A^2).
	\end{equation}
	Indeed, it follows from the definition of $\{M_n^*\}$ that
	\begin{equation*}
		M_n^*\geq\frac{\left(M_{n-1}^*\right)^2}{M_{n-2}^*}\Longleftrightarrow\frac{M_n^*}{M_{n-1}^*}\geq \frac{M_{n-1}^*}{M_{n-2}^*}\quad\forall\ n\in\Z\cap[3,A^2).
	\end{equation*}
	Hence, $n\mapsto \frac{M_n^*}{M_{n-1}^*}$ is increasing for $n\in\Z\cap[2,A^2)$. So,
	\begin{equation}\label{Eq.M_n^*_increasing}
		\frac{M_n^*}{M_{n-1}^*}\geq \frac{M_2^*}{M_1^*}>1\Longrightarrow M_n^*\geq M_{n-1}^*\quad\forall\ n\in\Z\cap[2, A^2).
	\end{equation}
	Let $k\in\Z\cap[2,A^2)$ and $j\in\Z\cap[1,k/2-1]$, (if $j=k/2$, \eqref{Eq.M_n^*claim} is trivial,) then $j\leq k-j-2$ and thus,
	\begin{align*}
		M_{k-j}^*\geq \frac{\left(M_{k-j-1}^*\right)^2}{M_{k-j-2}^*}=\frac{M_{k-j-1}^*}{M_{k-j-2}^*}M_{k-j-1}^*\geq \frac{M_{j+1}^*}{M_{j}^*}M_{k-j-1}^*.
	\end{align*}
	Hence, $M_{j+1}^*M_{k-(j+1)}^*\leq M_{j}^*M_{k-j}^*$. This proves our claim \eqref{Eq.M_n^*claim}.
	
	\underline{\textbf{Step 3.}} Thanks to \eqref{Eq.a_n_bound_byM_n^*}, it suffices to estimate $M_n^*$, the recurrence relation of which was given by \eqref{Eq.M_n^*_recurrence}. By \eqref{Eq.M_n^*claim}, for $n\in\Z\cap[14, A^2)$, we have
	\begin{align*}
		\sum_{j=6}^{n-5}M_j^*M_{n+1-j}^*&=2M_6^*M_{n-5}^*+2M_7^*M_{n-6}^*+\sum_{j=8}^{n-7}M_j^*M_{n+1-j}^*\\
		&\leq 2M_6^*M_{n-5}^*+2M_7^*M_{n-6}^*+(n-14)M_8^*M_{n-7}^*,
	\end{align*}
	and similarly,
	\begin{align*}
		\sum_{j=5}^{n-5}M_j^*M_{n-j}^*&\leq 2M_5^*M_{n-5}^*+2M_6^*M_{n-6}^*+(n-13)M_7^*M_{n-7}^*,\\
		\sum_{j=4}^{n-5}M_j^*M_{n-1-j}^*&\leq 2M_4^*M_{n-5}^*+2M_5^*M_{n-6}^*+(n-12)M_6^*M_{n-7}^*,\\
		\sum_{j=3}^{n-5}M_j^*M_{n-2-j}^*&\leq 2M_3^*M_{n-5}^*+2M_4^*M_{n-6}^*+(n-11)M_5^*M_{n-7}^*.
	\end{align*}
	It follows from \eqref{Eq.a_n_limit}, Remark \ref{Rmk.a_n_positive}, \eqref{Eq.mu_n^*}, \eqref{Eq.tilde_a_n} and \eqref{Eq.M_n^*small} that, by taking $A_0>10$ to be larger if necessary, we have
	\begin{equation}\label{Eq.M_n^*small_est}
		M_j^*\asymp1,\quad\forall\ j\in\Z\cap[1,10],\ \forall\ A>A_0.
	\end{equation}
	Hence, \eqref{Eq.M_n^*_recurrence} implies that
	\begin{equation}\label{Eq.a_n^*_est}
		\g_na_n^*\leq (2p_n+\g_n\la_n)M_{n-1}^*+C_1\left(M_{n-3}^*+M_{n-4}^*\right)+C_1n\left(M_{n-5}^*+M_{n-6}^*\right)+C_1n^2M_{n-7}^*
	\end{equation}
	for all $n\in\Z\cap[14, A^2)$, all $A>A_0$ and some absolute constant $C_1>0$ which is independent of $n$ and $A$. Using \eqref{Eq.quotient_M_n^*_est}, \eqref{Eq.tilde_b_n} and \eqref{Eq.a_n^*_est}, we obtain
	\begin{align*}
		\frac{a_n^*}{M_{n-1}^*}\leq \frac{2p_n+\g_n\la_n}{\g_n}+\frac{C_1}{\g_n}\left(\frac{M_{n-3}^*}{M_{n-1}^*}+\frac{M_{n-4}^*}{M_{n-1}^*}\right)+\frac{C_1n}{\g_n}\left(\frac{M_{n-5}^*}{M_{n-1}^*}+\frac{M_{n-6}^*}{M_{n-1}^*}\right)+\frac{C_1n^2}{\g_n}\frac{M_{n-7}^*}{M_{n-1}^*}\leq b_n,
	\end{align*}
	where
	\begin{align}
		b_n:&=\wt b_n+C_1\left(\frac1{\g_n\wt b_{n-1}\wt b_{n-2}}+\frac1{\g_n\wt b_{n-1}\wt b_{n-2}\wt b_{n-3}}\right)+\frac{C_1n}{\g_n\wt b_{n-1}\wt b_{n-2}\wt b_{n-3}\wt b_{n-4}}\label{Eq.b_n_def}\\
		&\qquad+\frac{C_1n}{\g_n\wt b_{n-1}\wt b_{n-2}\wt b_{n-3}\wt b_{n-4}\wt b_{n-5}}+\frac{C_1n^2}{\g_n\wt b_{n-1}\wt b_{n-2}\wt b_{n-3}\wt b_{n-4}\wt b_{n-5}\wt b_{n-6}}\nonumber
	\end{align}
	for $n\in\Z\cap [14, A^2)$. Then \eqref{Eq.M_n^*large} implies
	\begin{equation}\label{Eq.quotient_M_n^*est}
		\frac{M_n^*}{M_{n-1}^*}=\max\left\{\frac{a_n^*}{M_{n-1}^*}, \frac{M_{n-1}^*}{M_{n-2}^*}\right\}\leq \max\left\{b_n, \frac{M_{n-1}^*}{M_{n-2}^*}\right\},\quad\forall\ n\in\Z\cap[14, A^2).
	\end{equation}
	We define
	\begin{equation}\label{Eq.b_nsmall}
		b_n:=\frac{M_n^*}{M_{n-1}^*},\quad\forall\ n\in\Z\cap[2, 13].
	\end{equation}
	Using \eqref{Eq.M_n^*small_est}, \eqref{Eq.M_n^*_recurrence} and a similar argument to the proof of \eqref{Eq.a_n^*_est}, we then obtain
	\begin{equation}\label{Eq.b_nsmall_est}
		0<b_n\lesssim 1,\quad\forall\ n\in\Z\cap[2,13],\ \forall\ A>A_0.
	\end{equation}
	It follows from \eqref{Eq.quotient_M_n^*est}, \eqref{Eq.b_nsmall} and the induction that
	\begin{equation} \label{Eq.quotient_M_n^*_est2}
		\frac{M_n^*}{M_{n-1}^*}\leq \max_{2\leq j\leq n}b_j,\quad\forall\ n\in\Z\cap[2, A^2).
	\end{equation}
	
	\underline{\textbf{Step 4.}} By \eqref{Eq.tilde_b_n}, $\g_n\wt b_n=2p_n+\g_n\la_n\asymp \mu_{n-1}\asymp \sqrt n$ (due to \eqref{Eq.mu_n^*} and \eqref{Eq.quotient_est}), $0<\g_n<8$ and \eqref{Eq.b_n_def}, we have
	\begin{equation}\label{Eq.b_n_tilde_bn}
		1\leq \frac{b_n}{\wt b_n}=1+\cO\left(n^{-\frac32}\right)\Longrightarrow b_n\gtrsim \sqrt n,\quad\forall\ n\in\Z\cap[14, A^2).
	\end{equation}
    Now we claim that there exist $N_1\in\Z_{\geq 14}$ and $A_0>2N_1+1$ such that
    \begin{equation}\label{Eq.bn_increasing}
        b_n>b_{n-1},\quad\forall\ n\in\Z\cap[N_1, A^2),\ \forall\ A>A_0.
    \end{equation}
    Assuming \eqref{Eq.bn_increasing} temporarily, there holds
    \begin{equation}\label{Eq.b_j_max}
		\max_{2\leq j\leq n}b_j=\max\left\{b_2, b_3, \cdots, b_{N_0-1}, b_n\right\},\quad\forall\ n\in\Z\cap[N_1, A^2).
	\end{equation}
	By \eqref{Eq.tilde_b_n}, \eqref{Eq.p_n}, \eqref{Eq.mu_n^*} and the definition of $\g_n$, for $n\in\Z\cap[1, N_1]$ and $A>A_0$, we have $n<N_1<A_0^2/2<A^2/2$ and
	\begin{align*}
		\g_n=\frac8{(A+1)^2}&(A^2-n)\geq  \frac8{(A+1)^2}\frac{A^2}{2}=\frac{4A^2}{(A+1)^2}\gtrsim 1,\quad p_{n}\lesssim \frac nA\lesssim 1\\
		&\wt b_n=\frac{2p_{n}}{\g_n}+\la_n\lesssim 1+\la_n\lesssim \sqrt n\lesssim \sqrt{N_1}\lesssim 1.
	\end{align*}
	Then \eqref{Eq.b_nsmall_est} and \eqref{Eq.b_n_tilde_bn} imply that
	\begin{equation}\label{Eq.b_nsmall_est2}
		0<b_n\lesssim 1,\quad\forall\ n\in\Z\cap[2, N_1],\ \forall\ A>A_0.
	\end{equation}
	Now, by \eqref{Eq.b_n_tilde_bn}, \eqref{Eq.b_j_max} and \eqref{Eq.b_nsmall_est2}, there exists $N_0\in\Z\cap(N_1, +\infty)$ such that	$\max_{2\leq j\leq n}b_j=b_n$ for $n\in\Z\cap[N_0, A^2)$ and $A>A_0$.%, where we take $A_0$ to be larger if necessary. %In particular, we have shown that
	Then \eqref{Eq.quotient_M_n^*_est2} implies that
	\begin{equation*}
		M_n^*\leq M_{N_0-1}^*\prod_{j=N_0}^n b_j,\quad\forall\ n\in\Z\cap[N_0, A^2).
	\end{equation*}
	Due to \eqref{Eq.b_n_tilde_bn}, the fact $1<\prod_{j=1}^\infty\left(1+\frac C{j^{3/2}}\right)<+\infty$ and \eqref{Eq.tilde_b_n}, we have
	\begin{align*}
		M_n^*\lesssim M_{N_0-1}^*\prod_{j=N_0}^n\wt b_j\lesssim M_{N_0-1}^*\prod_{j=2}^n\wt b_j,\quad\forall\ n\in\Z\cap[N_0, A^2).
	\end{align*}
	It follows from Lemma \ref{Lem.a_n_trivial_bound} and the definition of $\{M_n^*\}$ that $M_{j}^*\lesssim1$ for all $j\in\Z\cap[1,N_0-1]$ and $A>A_0$, by taking $A_0$ to be larger if necessary. Hence,
	\begin{equation}\label{Eq.M_n^*est}
		M_n^*\lesssim \prod_{j=2}^n\wt b_j\asymp \wt M_n\asymp M_n,\quad\forall\ n\in\Z\cap[1, A^2),
	\end{equation}
	where we have used Lemma \ref{Lem.tildeM_nest} and Lemma \ref{Lem.M_n_tildeM_n}. Combining this with \eqref{Eq.a_n_bound_byM_n^*} completes the proof of \eqref{Eq.a_n_upper_bound}.

    \underline{\bf Step 5.} Finally, we prove \eqref{Eq.bn_increasing}. We denote $N:=\lceil A^2\rceil-1$. For $n\in\Z\cap[14, N-1]$, by $\g_n=\frac 8{(A+1)^2}(A^2-n)$, we have $\g_n\asymp\g_{n-1}$. Then $\g_{n-1}\wt b_{n-1}\asymp \sqrt n$ implies that $\g_n\wt b_{n-1}\asymp\sqrt n$, and hence it follows from \eqref{Eq.b_n_def} and \eqref{Eq.tilde_b_n} that $b_n=\wt b_n+\cO(n^{-1})$. As a consequence, by \eqref{Eq.tilde_b_n}, \eqref{Eq.b_n_def}, Lemma \ref{Lem.la_n_increasing}, the increase of $n\mapsto p_n$ and the decrease of $n\mapsto \g_n$, we have
	\begin{align}
		b_n-b_{n-1}&=\wt b_n-\wt b_{n-1}+\cO\left(\frac1n\right)
		=\frac{2p_n}{\g_n}-\frac{2p_{n-1}}{\g_{n-1}}+\la_n-\la_{n-1}+\cO\left(\frac1n\right)\nonumber\\
		&>\la_n-\la_{n-1}+\cO\left(\frac1n\right)\gtrsim \frac1{\sqrt n}+\cO\left(\frac1n\right),\nonumber
	\end{align}
	for $n\in\Z\cap[14, N-1]$. Hence, there exist $N_1\in\Z_{\geq 14}$ and $A_0>2N_1+1$ such that
	\begin{equation*}
		b_n>b_{n-1},\quad\forall\ n\in\Z\cap[N_1, N-1].
	\end{equation*}
    On the other hand, for $n=N$, we have
    \[\wt b_N=\frac{2p_N}{\g_N}+\la_N=\frac{p_N(A+1)^2}{4(A^2-N)}+\la_N\asymp \frac{A^3}{A^2-N}+A\asymp \frac{A^3}{A^2-N}.\]
    Here we used \eqref{Eq.p_n}, \eqref{Eq.p_n_asymp} and \eqref{Eq.mu_n^*}. Therefore, it follows from \eqref{Eq.b_n_tilde_bn}, \eqref{Eq.p_n}, \eqref{Eq.p_n_asymp}, the decrease of $n\mapsto \g_n$ and Lemma \ref{Lem.la_n_increasing} that
    \begin{align*}
        b_N-b_{N-1}&=\wt b_N-\wt b_{N-1}+\cO\left(A^{-3}\frac{A^3}{A^2-N}\right)\\
        &=\frac{2p_N}{\g_N}-\frac{2p_{N-1}}{\g_{N-1}}+\la_N-\la_{N-1}+\cO\left((A^2-N)^{-1}\right)\\
        &>\frac{2p_N-2p_{N-1}}{\g_N}+\cO\left((A^2-N)^{-1}\right)\\
        &\gtrsim \frac{1}{A\g_N}+\cO\left((A^2-N)^{-1}\right)\gtrsim \frac{(A+1)^2}{A(A^2-N)}+\cO\left((A^2-N)^{-1}\right)\\
        &\gtrsim A(A^2-N)^{-1}+\cO\left((A^2-N)^{-1}\right)>0
    \end{align*}
    as long as we take $A_0>10$ to be larger if necessary. This proves \eqref{Eq.bn_increasing}.
\end{proof}

\subsection{Lower bound of $\{a_n\}$}\label{Subsec.a_n_lower_bound}
In this subsection, we prove the ''$\gtrsim$" part of Proposition \ref{Prop.a_n_sharp_est}. Our main goal  is the following proposition.

\begin{proposition}\label{Prop.an_lower_bound}
	There exist $A_0>10$ and $c_0\in(0,1)$ such that
	\begin{equation*}
		\frac{a_n}{M_n}>c_0,\quad\forall\ n\in\Z\cap[A, A^2),\ \forall\ A>A_0.
	\end{equation*}
\end{proposition}

The proof reduces to analyzing the limiting sequence $\{a_n^\infty\}$, for which a key non-degenerate property (see \eqref{Eq.S_infty}) is assumed and verified using numerical methods.

\begin{lemma}
	Let
	\begin{equation}\label{Eq.hat_a_n}
		\wh a_n:=a_n+\la_na_{n-1},\quad\forall\ n\in\Z_+.
	\end{equation}
	We define the sequence $\{\wh M_n\}_{n=1}^{\lceil A^2\rceil-1}$ by $\wh M_1=a_1$ and
	\begin{equation}\label{Eq.hat_M_n}
		\wh M_n=\frac{\mu_{n-1}}{\g_n}\wh M_{n-1},\quad\forall\ n\in\Z\cap[2, A^2).
	\end{equation}
	Then there exist $A_0>10$ and a constant $C>0$ such that
	\begin{equation}\label{Eq.a_n/hat_M_n}
		\left|\frac{\wh a_n}{\wh M_n}-\frac{\wh a_{n-1}}{\wh M_{n-1}}\right|\leq C\left(\frac1{n^{3/2}}+\frac1{\sqrt nA}\right),\quad\forall\ n\in\Z\cap[2, A^2),\ \forall\ A>A_0.
	\end{equation}
\end{lemma}
\begin{proof}
	By \eqref{Eq.hat_M_n}, \eqref{Eq.tildeM_n_est} and Lemma \ref{Lem.M_n_tildeM_n}, we have
	\begin{equation}\label{Eq.hat_M_n_est}
		\wh M_n=\wh M_1\prod_{j=2}^n\frac{\mu_{j-1}}{\g_j}\asymp \wt M_n\asymp M_n,\quad\forall\ Z\cap[1, A^2),
	\end{equation}
	if $A_0>10$ is large enough, where we have used $\lim_{A\to+\infty}\wh M_1=a_1^\infty=2>0$.
	Let
	\begin{align}\label{Eq.epsilon_n_def}
		\varepsilon_n:=s_1a_{n-3}+s_2a_{n-4}+\wt\varepsilon_n,\quad\forall\ n\in\Z\cap[10, +\infty).
	\end{align}
	The proof of \eqref{Eq.a_n^*_est} gives that
	\begin{equation*}
		\left|\varepsilon_n\right|\leq C_1\left(M_{n-3}^*+M_{n-4}^*\right)+C_1n\left(M_{n-5}^*+M_{n-6}^*\right)+C_1n^2M_{n-7}^*
	\end{equation*}
	for $n\in\Z\cap[14, A^2+3)$ and $A>A_0$, by taking $A_0$ to be larger if necessary, where $C_1>0$ is the same as the constant in \eqref{Eq.a_n^*_est}. Combining this with the fact $\wt b_n\asymp\mu_{n-1}/\g_n$ by \eqref{Eq.quotient_est}, \eqref{Eq.tilde_b_n}, \eqref{Eq.M_n^*est} and \eqref{Eq.hat_M_n_est} gives that
	\begin{align}
		\left|\varepsilon_n\right|&\lesssim \wh M_{n-3}+\wh M_{n-4}+n\wh M_{n-5}+n\wh M_{n-6}+n^2\wh M_{n-7}\nonumber\\
		&\lesssim \wh M_{n-3}\left(1+\frac1{\wt b_{n-3}}+\frac n{\wt b_{n-3}\wt b_{n-4}}+\frac n{\wt b_{n-3}\wt b_{n-4}\wt b_{n-5}}+\frac{n^2}{\wt b_{n-3}\wt b_{n-4}\wt b_{n-5}\wt b_{n-6}}\right)\lesssim \wh M_{n-3},\label{Eq.epsilon_n_est}
	\end{align}
	for all $n\in\Z\cap[14, A^2+3)$ and $A>A_0$. It follows from \eqref{Eq.a_n_new_recurrence}, \eqref{Eq.mu_n_la_n} and \eqref{Eq.epsilon_n_def} that
	\begin{align*}
		\g_n\wh a_n&=(2p_n+\g_n\la_n)a_{n-1}+q_na_{n-2}+\varepsilon_n\\
		&=\mu_{n-1}\wh a_{n-1}+\left(2p_n+\g_n\la_n-\mu_{n-1}\right)a_{n-1}+\varepsilon_n,
	\end{align*}
	for all $n\in\Z\cap[10, +\infty)$. Using \eqref{Eq.quotient_est}, Proposition \ref{Prop.a_n_upper_bound} and \eqref{Eq.hat_M_n_est},  we obtain
	\begin{align*}
		\left|\g_n\wh a_n-\mu_{n-1}\wh a_{n-1}\right|&\leq \mu_{n-1}\left|\frac{2p_n+\g_n\la_n}{\mu_{n-1}}-1\right|\wh M_{n-1}+\wh M_{n-3}\\
		&\lesssim \left(\frac1{n^2}+\frac1{\sqrt nA}\right)\mu_{n-1}\wh M_{n-1}+\wh M_{n-3},
	\end{align*}
	for all $n\in\Z\cap[14, A^2)$ and $A>A_0$. Dividing  both sides by $\g_n\wh M_n=\mu_{n-1}\wh M_{n-1}$ yields  that
	\begin{align*}
		\left|\frac{\wh a_n}{\wh M_n}-\frac{\wh a_{n-1}}{\wh M_{n-1}}\right|&\lesssim \frac1{n^2}+\frac1{\sqrt nA}+\frac{\wh M_{n-3}}{\mu_{n-1}\wh M_{n-1}}\\
		&\lesssim \frac1{n^2}+\frac1{\sqrt nA}+\frac{\g_{n-1}\g_{n-2}}{\mu_{n-1}\mu_{n-2}\mu_{n-3}}\\
		&\lesssim \frac1{n^2}+\frac1{\sqrt nA}+\frac1{n^{3/2}}
	\end{align*}
	for all $n\in\Z\cap[14, A^2)$ and $A>A_0$, where we have used $\mu\asymp\sqrt n$ and $\g_n\leq 8$. Finally, \eqref{Eq.a_n/hat_M_n} for $n\in\Z\cap[2,13]$ follows from Lemma \ref{Lem.a_n_trivial_bound}, \eqref{Eq.mu_n^*} and \eqref{Eq.hat_M_n_est}.
\end{proof}

We recall the definition of $\{a_n^\infty\}_{n=0}^\infty$ from \eqref{Eq.a_n^infty}, and the property \eqref{Eq.a_n_limit}. Letting $A\to\infty$, we obtain (see also \eqref{Eq.p_n}--\eqref{Eq.q_n_asymp}, \eqref{Eq.mu_n_la_n} and \eqref{Eq.mu_n^*_expression})
\begin{align}
	\g_n^\infty:&=\lim_{A\to\infty}\g_n=8,\quad \lim_{A\to\infty}p_n=0,\quad q_n^\infty:=\lim_{A\to\infty}q_n=n-\frac43,\nonumber\\
	\mu_n^\infty:&=\lim_{A\to\infty}\mu_n^*=\sqrt{8n+\frac43},\quad
	\la_n^\infty:=\lim_{A\to\infty}\la_n=\frac{n-\frac13}{\sqrt{8n+\frac43}},\quad\forall\ n\in\Z_{\geq1}.\label{Eq.la_n_limit}
\end{align}
We define the sequence $\{\wh M_n^\infty\}_{n=1}^\infty$ by $\wh M_1^\infty=a_1^\infty=2$, and
\begin{equation}\label{Eq.hat_M_n^infty}
	\wh M_n^\infty=\frac{\mu_{n-1}^\infty}{8}\wh M_{n-1}^\infty,\quad\forall\ n\in\Z_{\geq2}.
\end{equation}

\begin{corollary}\label{Cor.limit}
	Let
	\begin{equation}\label{Eq.hat_a_n^infty}
		\wh a_n^\infty:=a_n^\infty+\la_n^\infty a_{n-1}^\infty,\quad\forall\ n\in\Z_{\geq 1}.
	\end{equation}
	Then the limit $\lim_{n\to\infty}\frac{\wh a_n^\infty}{\wh M_n^\infty}\in\R$ exists.
\end{corollary}
\begin{proof}
	By \eqref{Eq.hat_a_n}, \eqref{Eq.hat_M_n}, \eqref{Eq.la_n_limit}, \eqref{Eq.hat_M_n^infty} and \eqref{Eq.hat_a_n^infty}, we have
	\begin{equation}\label{Eq.hat_a_n_limit}
		\lim_{A\to+\infty}\wh a_n=\wh a_n^\infty,\quad\lim_{A\to+\infty}\wh M_n=\wh M_n^\infty,\quad\forall\ n\in\Z_+.
	\end{equation}
	Letting $A\to+\infty$ in \eqref{Eq.a_n/hat_M_n} gives that
	\begin{align*}
		\left|\frac{\wh a_n^\infty}{\wh M_n^\infty}-\frac{\wh a_{n-1}^\infty}{\wh M_{n-1}^\infty}\right|\leq \frac C{n^{3/2}},\quad\forall\ n\in\Z_{\geq 2},
	\end{align*}
	for some absolute constant $C>0$. Therefore, $\left\{\wh a_n^\infty/\wh M_n^\infty\right\}_{n=1}^\infty$ is a Cauchy sequence in $\R$.
	\if0
	Letting $A\to+\infty$ in \eqref{Eq.a_n_new_recurrence}, using \eqref{Eq.k_1k_2asymp} and \eqref{Eq.p_n}---\eqref{Eq.tilde_epsilon_n}, we obtain the following re-formulated recurrence relation of $\{a_n^\infty\}_{n=0}^\infty$:
	\begin{equation}
		8a_n^\infty=\left(n-\frac43\right)a_{n-2}^\infty-\frac{17}3a_{n-3}^\infty-\frac54a_{n_4}^\infty+\wt\varepsilon_n^\infty,
	\end{equation}
	where
	\begin{align*}
		\wt\varepsilon_n^\infty:&=-\frac32(n+1)\sum_{j=6}^{n-5}a_j^\infty a_{n+1-j}^\infty+(3n-2)\sum_{j=5}^{n-5}a_j^\infty a_{n-j}^\infty\\
		&\quad-\left(\frac{39}{16}n-\frac{71}{16}\right)\sum_{j=4}^{n-5}a_j^\infty a_{n-1-j}^\infty+\frac{5}{8}\left(\frac32n-5\right)\sum_{j=3}^{n-5}a_j^\infty a_{n-2-j}^\infty,
	\end{align*}
	for all $n\in\Z\cap[10, +\infty)$. The first ten terms of $\{a_n^\infty\}$ are given by Remark \ref{Rmk.a_n_positive}.\fi
\end{proof}

The rest of our proof is based on the positivity of the limit in Corollary \ref{Cor.limit}, which can be checked numerically.

\textit{Numerical claim:} We have
\begin{equation}\label{Eq.S_infty}
	\cS_\infty:=\lim_{n\to\infty}\frac{\wh a_n^\infty}{\wh M_n^\infty}>\frac12>0.
\end{equation}

\begin{lemma}\label{Lem.a_n_lower}
	There exist $N_0\in\Z_+$, $A_0>10$ and $c_{00}\in(0,1)$
such that
	\begin{equation*}
		\frac{a_n+\la_n a_{n-1}}{M_n}>c_{00},\quad\forall\ n\in\Z\cap[N_0, A^{3/2}].
	\end{equation*}
\end{lemma}
\if0
\begin{proof}
	By \eqref{Eq.a_n/hat_M_n}, there exist $A_0>10$ such that for all $A>A_0$, $n\in\Z\cap[2, A^2)$ and $k\in\Z\cap[1,n]$, we have
	\begin{equation}\label{Eq.a_n/M_ndiff}
		\left|\frac{\wh a_n}{\wh M_n}-\frac{\wh a_k}{\wh M_k}\right|\lesssim \sum_{j=k+1}^n\left(\frac1{j^{3/2}}+\frac{1}{\sqrt jA}\right)\leq C_0\left(\frac1{\sqrt k}+\frac{\sqrt n}{A}\right)
	\end{equation}
	for some constant $C_0>0$. By \eqref{Eq.S_infty}, there exists $N_1\in\Z_+$ such that
	\begin{equation}\label{Eq.S_n>0}
		\frac{\wh a_n^\infty}{\wh M_n^\infty}>\frac{1}{2}>0,\quad\forall\ n\in\Z\cap[N_1,+\infty).
	\end{equation}
	Fix an $N_0\in\Z_+$ satisfying $N_0>N_1+(8C_0)^2$. Then it follows from \eqref{Eq.hat_a_n_limit} and \eqref{Eq.S_n>0} that
	\begin{equation}\label{Eq.a_N_0}
		\frac{\wh a_{N_0}}{\wh M_{N_0}}>\frac{1}{2}>0,\quad\forall\ A>A_0,
	\end{equation}
	by taking $A_0>10$ to be larger if necessary. Now, for $N_0\leq n\leq(8C_0)^{-2}A^2$ and $A>A_0$, by \eqref{Eq.a_n/M_ndiff} and the choice of $N_0$, we have
	\begin{align*}
		\left|\frac{\wh a_n}{\wh M_n}-\frac{\wh a_{N_0}}{\wh M_{N_0}}\right|\leq C_0\left(\frac1{\sqrt{N_0}}+\frac{\sqrt n}{A}\right)<\frac{1}{8}+\frac{1}{8}=\frac{1}{4},
	\end{align*}
	which along with \eqref{Eq.a_N_0} implies that $\frac{\wh a_n}{\wh M_n}>\frac{1}{4}$.
	Then our desired result follows from \eqref{Eq.hat_a_n} and \eqref{Eq.hat_M_n_est}, by taking $A_0>10$ to be larger if necessary.
\end{proof}\fi 

\begin{proof}
	By \eqref{Eq.a_n/hat_M_n}, there exists $A_0>10$ such that for $A>A_0$, $n\in\Z\cap[2, A^2)$, $k\in\Z\cap[1,n]$, we have
	\begin{equation}\label{Eq.a_n/M_ndiff}
		\left|\frac{\wh a_n}{\wh M_n}-\frac{\wh a_k}{\wh M_k}\right|\lesssim \sum_{j=k+1}^n\left(\frac1{j^{3/2}}+\frac{1}{\sqrt jA}\right)\leq C_0\left(\frac1{\sqrt k}+\frac{\sqrt n}{A}\right)
	\end{equation}
	for some constant $C_0>0$. By \eqref{Eq.S_infty}, there exists $N_1\in\Z_+$ such that
	\begin{equation}\label{Eq.S_n>0}
		\frac{\wh a_n^\infty}{\wh M_n^\infty}>\frac{\cS_\infty}{2}>0,\quad\forall\ n\in\Z\cap[N_1,+\infty).
	\end{equation}
	Fix an $N_0\in\Z_+$ satisfying $N_0>N_1+(16C_0/\cS_\infty)^2$, then it follows from \eqref{Eq.hat_a_n_limit} and \eqref{Eq.S_n>0} that
	\begin{equation}\label{Eq.a_N_0}
		\frac{\wh a_{N_0}}{\wh M_{N_0}}>\frac{\cS_\infty}{4}>0,\quad\forall\ A>A_0,
	\end{equation}
	by taking $A_0>10$ to be larger if necessary. Now, for $N_0\leq n\leq\left(\frac{\cS_\infty}{16C_0}\right)^2A^2$ and $A>A_0$, by \eqref{Eq.a_n/M_ndiff} and the choice of $N_0$ we have
	\begin{align*}
		\left|\frac{\wh a_n}{\wh M_n}-\frac{\wh a_{N_0}}{\wh M_{N_0}}\right|\leq C_0\left(\frac1{\sqrt{N_0}}+\frac{\sqrt n}{A}\right)<\frac{\cS_\infty}{16}+\frac{\cS_\infty}{16}=\frac{\cS_\infty}{8},
	\end{align*}
	which combining with \eqref{Eq.a_N_0} implies that $\frac{\wh a_n}{\wh M_n}>\frac{\cS_\infty}{8}$.	Then our desired result follows from \eqref{Eq.hat_a_n} and \eqref{Eq.hat_M_n_est}, by taking $A_0>10$ to be larger if necessary.
\end{proof}

\begin{lemma}\label{Lem.a_n/M_n_lower}
	There exist $A_0>10$ and $c_{01}\in(0,1)$ such that
	\begin{equation}\label{Eq.a_n/M_n_lower}
		\frac{a_n}{M_n}>c_{01},\quad\forall\ n\in\Z\cap[A, A^{3/2}],\ \forall\ A>A_0.
	\end{equation}
\end{lemma}
\begin{proof}
	Using \eqref{Eq.M_n_sequence}, \eqref{Eq.a_n_new_recurrence} and \eqref{Eq.epsilon_n_def}, we obtain
	\begin{equation}\label{Eq.Wronski_recurrence}
		\g_n(a_nM_{n-1}-a_{n-1}M_n)=-q_n(a_{n-1}M_{n-2}-a_{n-2}M_{n-1})+\varepsilon_nM_{n-1},\ \ \forall\ n\in\Z\cap[10, A^2).
	\end{equation}
	We consider an auxiliary sequence $\{L_n\}_{n=0}^{\lceil A^2\rceil-1}$ defined by
	\begin{equation}\label{Eq.L_n}
		L_0=1,\quad L_n=-\frac{q_n}{\g_n}L_{n-1}\Longrightarrow L_n=(-1)^n\prod_{j=1}^n\frac{q_j}{\g_j},\quad\forall\ n\in\Z\cap[1, A^2).
	\end{equation}
	Dividing both sides of \eqref{Eq.Wronski_recurrence} by $\g_n L_n=-q_nL_{n-1}$ gives that
	\begin{equation}\label{Eq.Wronski_recurrence2}
		\frac{a_nM_{n-1}-a_{n-1}M_n}{L_n}-\frac{a_{n-1}M_{n-2}-a_{n-2}M_{n-1}}{L_{n-1}}=-\frac{\varepsilon_nM_{n-1}}{q_nL_{n-1}},\quad\forall\ n\in\Z\cap[10, A^2).
	\end{equation}
	Using \eqref{Eq.epsilon_n_est}, \eqref{Eq.q_n}, \eqref{Eq.q_n_asymp}, \eqref{Eq.hat_M_n}, $\wt b_n\asymp \mu_{n-1}/\g_n$, \eqref{Eq.tilde_b_n} and \eqref{Eq.hat_M_n_est}, we have
	\begin{align*}
		\left|\frac{\varepsilon_n}{q_n}\right|\lesssim \frac1n \wh M_{n-3}\lesssim \frac1{n}\frac{1}{\wt b_{n-2}\wt b_{n-1}}\wh M_{n-1}
		\lesssim \frac1{n^2}\wh M_{n-1}\lesssim \frac1{n^2}M_{n-1},\quad\forall\ n\in\Z\cap[14, A^2),\ \forall\ A>A_0,
	\end{align*}
	if $A_0>10$ is chosen to be large enough. By \eqref{Eq.q_n}, \eqref{Eq.q_n_asymp}, Lemma \ref{Lem.a_n_trivial_bound}, \eqref{Eq.mu_n^*}, \eqref{Eq.hat_M_n_est} and \eqref{Eq.L_n} we have
$|a_{13}M_{12}-a_{12}M_{13}|\leq C $, $|L_{13}|+1/|L_{13}|+1/|\wh M_{13}|\leq C$. So, \eqref{Eq.Wronski_recurrence2} and \eqref{Eq.hat_M_n_est} imply that
	\begin{equation}\label{Eq.Wronski_est1}
		\left|\frac{a_nM_{n-1}-a_{n-1}M_n}{L_n}\right|\leq C+\sum_{k=14}^n\left|\frac{\varepsilon_kM_{k-1}}{q_kL_{k-1}}\right|\lesssim1+
\sum_{k=14}^n\frac{M_{k-1}^2}{k^2|L_{k-1}|}\lesssim
\sum_{k=14}^n\frac{\wh M_{k-1}^2}{k^2|L_{k-1}|}.
	\end{equation}
	for all $n\in\Z\cap[14,A^2)$. By \eqref{Eq.hat_M_n}, \eqref{Eq.L_n} and \eqref{Eq.mu_n_la_n}, we have
	\begin{align}
		\frac{\wh M_n^2/|L_n|}{\wh M_{n-1}^2/|L_{n-1}|}&=\left(\frac{\mu_{n-1}}{\g_n}\right)^2\frac{\g_n}{q_n}=\frac{\mu_{n-1}^2}{q_n\g_n}=\frac{\left(\mu_{n-1}^*+p_{n-\frac12}\right)^2}{q_n\g_n}\nonumber\\
		&\geq \frac{\left(\mu_{n-1}^*+p_{n-\frac12}\right)^2}{\left(\mu_{n-1}^*\right)^2}=\left(1+\frac{p_{n-\frac12}}{\mu_{n-1}^*}\right)^2\label{Eq.quotient_lower1}\\
		&\geq\left(1+c_1\frac{n/A}{\sqrt n}\right)^2=\left(1+c_1\frac{\sqrt n}{A}\right)^2\label{Eq.quotient_lower2}
	\end{align}
	for all $n\in\Z\cap[1, A^2)$, where in \eqref{Eq.quotient_lower1} we have used
	\[\left(\mu_{n-1}^*\right)^2=\g_{n-\frac12}q_{n+\frac12}+p_n^2\geq \g_{n-\frac12}q_{n+\frac12}\geq \g_nq_n,\]
	thanks to the increasing of $t\mapsto q_t$ for $t>2$ and the decreasing of $t\mapsto \g_t$ for $t<A^2$, and in \eqref{Eq.quotient_lower2} we have used \eqref{Eq.p_n_asymp} and \eqref{Eq.mu_n^*}. As a consequence, for $n\in\Z\cap[2, A^2)$ and $k\in\Z\cap[2,n]$, we have
	\begin{equation*}
		\frac{\wh M_{k-1}^2}{|L_{k-1}|}\leq \prod_{j=k}^{n-1}\left(1+c_1\frac{\sqrt{j}}A\right)^{-2}\frac{\wh M_{n-1}^2}{|L_{n-1}|}.
	\end{equation*}
	Hence, \eqref{Eq.Wronski_est1} implies that
	\begin{equation}\label{Eq.Wronski_est2}
		\left|\frac{a_nM_{n-1}-a_{n-1}M_n}{L_n}\right|\lesssim d_n\frac{\wh M_{n-1}^2}{|L_{n-1}|},
	\end{equation}
	where
	\begin{equation*}
		d_n:=\sum_{k=14}^n\frac1{k^2}\prod_{j=k}^{n-1}\left(1+c_1\frac{\sqrt{j}}A\right)^{-2}
	\end{equation*}
	for all $n\in\Z\cap[14, A^2)$.
	
	We claim that there exist $A_0>10$ and a constant $C_0>0$ such that
	\begin{equation}\label{Eq.d_n_est}
		d_n\leq C_0\frac{\sqrt A}{n^2}\leq \frac1n,\quad\forall\ n\in\Z\cap[A, A^2),\ \forall\ A>A_0.
	\end{equation}
	
	Assume that \eqref{Eq.d_n_est} holds. Then by \eqref{Eq.Wronski_est2}, \eqref{Eq.hat_M_n_est}, \eqref{Eq.hat_M_n} and \eqref{Eq.L_n}, we have
	\begin{align}
		\left|\frac{a_n}{M_n}-\frac{a_{n-1}}{M_{n-1}}\right|&=\frac{|L_n|}{M_nM_{n-1}}\left|\frac{a_nM_{n-1}-a_{n-1}M_n}{L_n}\right|\lesssim \frac{|L_n|}{\wh M_n\wh M_{n-1}}d_n\frac{\wh M_{n-1}^2}{|L_{n-1}|}\nonumber\\
		&\lesssim \frac{q_n}{\mu_{n-1}}d_n\lesssim \sqrt nd_n,\quad\forall\ n\in\Z\cap[A, A^2),\ \forall\ A>A_0,\label{Eq.a_n/M_ndiff2}
	\end{align}
	where in the last inequality we have used \eqref{Eq.mu_n^*} and $q_n\asymp n$ due to \eqref{Eq.q_n_asymp}. Noting the identity
	\begin{equation*}
		 \frac{M_n}{M_{n-1}}\frac{a_n+\la_na_{n-1}}{M_n}+\la_n\left(\frac{a_n}{M_n}-\frac{a_{n-1}}{M_{n-1}}\right)=\left(\frac1{M_{n-1}}+\frac{\la_n}{M_n}\right)a_n=\left(\frac{M_n}{M_{n-1}}+\la_n\right)\frac{a_n}{M_n},
	\end{equation*}
	it follows from Lemma \ref{Lem.a_n_lower}, \eqref{Eq.hat_M_n_est}, \eqref{Eq.hat_M_n}, \eqref{Eq.mu_n^*}, \eqref{Eq.a_n/M_ndiff2} and \eqref{Eq.d_n_est} that
	\begin{align}
		\left(\frac{M_n}{M_{n-1}}+\la_n\right)\frac{a_n}{M_n}&\geq c_{00}\frac{M_n}{M_{n-1}}-C_1\la_n\sqrt nd_n\geq \wt c_0\frac{\wh M_n}{\wh M_{n-1}}-\wt C_0 nd_n\nonumber\\
		&=\wt c_0\frac{\mu_{n-1}}{\g_n}-\wt C_0 nd_n\geq \wt c_1\sqrt n-\wt C_0\gtrsim \sqrt n,\label{Eq.a_n/M_n_lower1}
	\end{align}
	for all $n\in\Z\cap[A, A^{3/2}]$ and $A>A_0$, by taking $A_0$ to be larger if necessary, where in the last inequality we have used $\mu_n\asymp\sqrt n$ and $\g_n\asymp 1$ for $n\in[A, A^{3/2}]\subset [2, A^2/2]$. On the other hand, by \eqref{Eq.mu_n^*}, \eqref{Eq.tildeM_n_est}, \eqref{Eq.tildeM_n_quotient} and $\g_n\asymp 1$ for $n\in[A, A^{3/2}]$, we have
	\begin{align*}
		0<\frac{M_n}{M_{n-1}}+\la_n\lesssim \sqrt n+\frac{\wt M_n}{\wt M_{n-1}}\lesssim \sqrt n+\frac{\mu_{n-1}}{\g_n}\lesssim \sqrt n,\quad\forall\ n\in\Z\cap[A, A^{3/2}].
	\end{align*}
	This along with \eqref{Eq.a_n/M_n_lower1} gives our desired \eqref{Eq.a_n/M_n_lower}.
	
	It remains to prove the claimed estimate \eqref{Eq.d_n_est} for $\{d_n\}$. Let $A>A_0$ and $n\in\Z\cap[A, A^2)$. If $k\in\Z\cap[n/2, n-1]$, then
	\[\ln\left(1+c_1\frac{\sqrt j}{A}\right)\geq \ln\left(1+c_1\frac{\sqrt k}{A}\right)\geq \ln\left(1+\frac{c_1}{\sqrt 2}\frac{\sqrt n}{A}\right)\gtrsim\frac{\sqrt n}{A},\]
	hence,
	\begin{equation*}
		\sum_{j=k}^{n-1}\ln\left(1+c_1\frac{\sqrt j}{A}\right)\gtrsim (n-k)\frac{\sqrt n}{A};
	\end{equation*}
	if $k\in\Z\cap[2, n/2]$, then
	\begin{equation*}
		\sum_{j=k}^{n-1}\ln\left(1+c_1\frac{\sqrt j}{A}\right)\geq \sum_{j=\lceil n/2\rceil}^{n-1}\ln\left(1+c_1\frac{\sqrt j}{A}\right)\gtrsim \frac n2\frac{\sqrt n}{A} \gtrsim (n-k)\frac{\sqrt n}{A}.
	\end{equation*}
	Thus, we have
	\begin{equation*}
		\sum_{j=k}^{n-1}\ln\left(1+c_1\frac{\sqrt j}{A}\right)\gtrsim (n-k)\frac{\sqrt n}{A},\quad \forall\ k\in[2, n-1].
	\end{equation*}
	Then there exists a constant $c_2\in(0,1)$ independent of $A$ and $n$ such that
	\begin{align*}
		\ln\prod_{j=k}^{n-1}\left(1+c_1\frac{\sqrt{j}}A\right)^{-2}=-2\sum_{j=k}^{n-1}\ln\left(1+c_1\frac{\sqrt j}{A}\right)\leq -c_2(n-k)\frac{\sqrt n}{A},\quad \forall\ k\in[2, n-1].
	\end{align*}
	Hence,
	\begin{align*}
		d_n&=\sum_{k=14}^n\frac1{k^2}\prod_{j=k}^{n-1}\left(1+c_1\frac{\sqrt{j}}A\right)^{-2}\leq \sum_{k=14}^n\frac1{k^2}{\e^{-c_2(n-k)\frac{\sqrt n}{A}}}\\
		&\leq \sum_{k=14}^{\lceil n/2\rceil-1}\frac1{k^2}\exp\left(-c_2\frac n2\frac{\sqrt n}{A}\right)+\sum_{k=\lceil n/2\rceil}^n\frac1{k^2}\e^{-c_2(n-k)\frac{\sqrt n}{A}}\\
        &\lesssim \exp\left(-c_2\sqrt n/2\right)+\frac1{n^2}\sum_{k=0}^{\lceil n/2\rceil}\e^{-c_2k\sqrt n/A}\lesssim \frac1{n^2}+\frac1{n^2}\frac1{1-\e^{-c_2\sqrt n/A}}\\
        &\lesssim n^{-2}+\frac{n^{-2}}{1-\e^{-c_2A^{-1/2}}}\lesssim \frac{\sqrt A}{n^2},
	\end{align*}
	where we have used $n\geq A$. This completes the proof of \eqref{Eq.d_n_est}.
\end{proof}

Now we are ready to prove Proposition \ref{Prop.an_lower_bound}.

\begin{proof}[Proof of Proposition \ref{Prop.an_lower_bound}]
	By Lemma \ref{Lem.a_n/M_n_lower}, there exists $A_0>10$ and $c_{01}\in(0,1)$ such that
	\begin{equation}\label{Eq.a_n/M_n_lower4.81}
		\frac{a_n}{M_n}>c_{01},\quad\forall\ n\in\Z\cap[A, A^{3/2}],\ \forall\ A>A_0.
	\end{equation}
    Let $n\in\Z\cap[A^{3/2}, A^2]$. Using \eqref{Eq.a_n/M_ndiff2} and \eqref{Eq.d_n_est} gives that
    \begin{align*}
        \left|\frac{a_n}{M_n}-\frac{a_{\lceil A^{3/2}\rceil-1}}{M_{\lceil A^{3/2}\rceil-1}}\right|\leq \sum_{j=\lceil A^{3/2}\rceil}^n\left|\frac{a_j}{M_j}-\frac{a_{j-1}}{M_{j-1}}\right|\lesssim \sum_{j=\lceil A^{3/2}\rceil}^n\sqrt jd_j\lesssim \sum_{j=\lceil A^{3/2}\rceil}^n\sqrt Aj^{-3/2}\leq C_0A^{-1/4},
    \end{align*}
    where $C_0>0$ is a constant independent of $A$ and $n$. By adjusting $A_0>10$ to be larger such that $C_0A_0^{-1/4}<c_{01}/2$, \eqref{Eq.a_n/M_n_lower4.81} then implies  that $a_n/M_n>c_{01}/2$ for all $n\in\Z\cap[A^{3/2}, A^2]$.
\end{proof}

\subsection{Proof of Lemma \ref{Lem.a_N+1<0}}\label{Subsec.a_N+1<0}
\begin{proof}[Proof of Lemma \ref{Lem.a_N+1<0}]
	Let $A_0>10$ be large enough such that all properties in Subsections \ref{Subsec.re-formulation}---\ref{Subsec.a_n_lower_bound} hold. Let $N_0\in\N$ be such that $N_0>A_0^2$ and we fix an $N\in\Z\cap (N_0, +\infty)$. Let $A^2=R\in(N, N+1)$. By \eqref{Eq.a_n_new_recurrence} and \eqref{Eq.epsilon_n_def}, we have
	\begin{equation}
		\g_{N+1}a_{N+1}=2p_{N+1}a_N+q_{N+1}a_{N-1}+\varepsilon_{N+1}.
	\end{equation}
	It follows from \eqref{Eq.epsilon_n_est} that $|\varepsilon_{N+1}|\lesssim \wh M_{N-2}$. Using \eqref{Eq.hat_M_n}, \eqref{Eq.mu_n^*} and $\g_n=\frac{8}{(A+1)^2}(A^2-n)=\frac{8}{(A+1)^2}(R-n)$, we obtain
	\begin{align}\label{Eq.M_N-2}
		\frac{\wh M_{N-2}}{\wh M_N}=\frac{\g_{N-1}\g_N}{\mu_{N-2}\mu_{N-1}}\asymp\frac{A^{-2}\cdot A^{-2}(R-N)}{\sqrt N\cdot\sqrt N}\asymp A^{-6}(R-N),
	\end{align}
	where we have used $(R-(N-1))\asymp1$ for $R\in(N, N+1)$. By Proposition \ref{Prop.a_n_sharp_est}, \eqref{Eq.hat_M_n_est}, \eqref{Eq.q_n} and \eqref{Eq.q_n_asymp}, similarly to \eqref{Eq.M_N-2}, we have
	\begin{align*}
		q_{N+1}a_{N-1}\asymp N\wh M_{N-1}\asymp NA^{-3}(R-N)\wh M_N\asymp A^{-1}(R-N)\wh M_N>0.
	\end{align*}
	Similarly, Proposition \ref{Prop.a_n_sharp_est}, \eqref{Eq.hat_M_n_est}, \eqref{Eq.p_n} and \eqref{Eq.p_n_asymp} imply that $2p_{N+1}a_N\asymp \frac{N}{A}\wh M_N\asymp A\wh M_N>0$. Take $A_0>0$ to be larger if necessary, then $\g_{N+1}a_{N+1}\asymp A\wh M_N>0$, hence,
	$$A^{-2}(R-(N+1))a_{N+1}\asymp A\wh M_N,$$
	namely, 
	\begin{equation}\label{Eq.a_N+1_est}
		-a_{N+1}\asymp\frac{A^3}{N+1-R}\wh M_N\asymp\frac{A^3}{N+1-R}a_N\asymp\frac{A^3}{N+1-R} M_N>0,
	\end{equation}
	here we recall $R\in(N, N+1)$.
\end{proof}

\section{Solution curve passing through the sonic point}\label{Solution curve passing through the sonic point}
Recall that $A=\sqrt R$ by \eqref{Eq.A_def}. By Proposition \ref{Prop.local_sol}, for $R\in(1, +\infty)\setminus\Z$ (equivalently $A\in(1, +\infty)\setminus\sqrt\N$), the $u-\tau$ ODE \eqref{ODE_u_tau} has a unique local analytic solution $u_L(\tau)=u_L(\tau; R)=\sum_{n=0}^\infty a_n\tau^n$ near $Q_2$ with $u_L(0)=1$ and $u_L'(0)=a_1$. We have also constructed the $Q_6-Q_2$ solution $u=u_F$.

Our goal in this section is to prove that the solution $u_F$ can be extended to pass through $Q_2$ smoothly. This is achieved by considering $R\to+\infty$, and by applying Proposition \ref{Prop.R_near_odd}, Proposition \ref{Prop.R_near_even}, and the intermediate value theorem. Readers may find several similarities between this section and \cite[Section 6]{SWZ2024_1}.

We begin with the proof of Proposition \ref{Prop.R_near_odd}.

\begin{proof}[Proof of Proposition \ref{Prop.R_near_odd}]
	Let $A_0>10$ be large enough such that all properties in Section \ref{Sec.an} hold. Let $N_0\in\N$ be such that $N_0>A_0^2$ and we fix an $N\in\Z\cap (N_0, +\infty)$. Let $R\in(N, N+1)$ and
	\begin{equation*}
		U_1(\tau):=U_1(\tau; R)=\sum_{j=0}^Na_j\tau^j+\kappa_1\tau^{N+1},\quad\ \kappa_1:=a_{N+1}-1<a_{N+1}.
	\end{equation*}
	Then for each $R\in(N, N+1)$, there exists $\tau_{1,1}(R)>0$ such that
	\begin{equation}\label{Eq.u_L>U_1}
		u_{L}(\tau; R)>U_1(\tau; R)\quad\forall\ \tau\in(0, \tau_{1,1}(R)].
	\end{equation}
	We define $\tau_{1,2}(R)=(R-N)^{1/N}$ for $R\in(N, N+1)$, and we claim that, taking $A_0>10$ to be larger if necessary, there exists $\th_1\in(0,1/2)$ such that for all $R\in(N, N+\th_1)$, we have
	\begin{equation}\label{Eq.U_1>u_F}
		U_1(\tau; R)>u_F(\tau;R)\quad\forall\ \tau\in(0, \tau_{1,2}(R)].
	\end{equation}
	Assuming \eqref{Eq.U_1>u_F}, for $R\in(N, N+\th_1)$ we take $\tau_1(R):=\min\{\tau_{1,1}(R), \tau_{1,2}(R)\}>0$. Then \eqref{Eq.u_L>u_F} follows from \eqref{Eq.u_L>U_1}.
	
	It suffices to prove \eqref{Eq.U_1>u_F}, which is a consequence of the following two inequalities:
	\begin{align}
		U_1&(\tau_{1,2}(R))>u_\text b(\tau_{1,2}(R))>u_F(\tau_{1,2}(R)),\label{Eq.U_1>u_b}\\
		&\cL[U_1](\tau)>0\quad\forall\ \tau\in(0, \tau_{1,2}(R)],\label{Eq.U_1_compare}
	\end{align}
	where $\cL$ is defined in \eqref{Eq.Lu}. Now we prove \eqref{Eq.U_1>u_F} by using \eqref{Eq.U_1>u_b} and \eqref{Eq.U_1_compare}. Assume on the contrary that \eqref{Eq.U_1>u_F} does not hold for some $R\in(N, N+\th_1)$. By \eqref{Eq.U_1>u_b} and the continuity, there exists $\tau_*\in(0, \tau_{1,2}(R))$ such that $U_1(\tau_*)=u_F(\tau_*)$ and $U_1(\tau)>u_F(\tau)$ for $\tau\in(\tau_*, \tau_{1,2}(R))$, hence $\frac{\mathrm dU_1}{\mathrm d\tau}(\tau_*)\geq \frac{\mathrm du_F}{\mathrm d\tau}(\tau_*)$. Since $\Delta_\tau(\tau_*, U_1(\tau_*))=\Delta_\tau(\tau_*, u_F(\tau_*))<0$, we have
	\begin{align*}
		\cL[U_1](\tau_*)&=\Delta_\tau(\tau_*, U_1(\tau_*))\frac{\mathrm dU_1}{\mathrm d\tau}(\tau_*)-\Delta_u(\tau_*, U_1(\tau_*))\\
		&\leq \Delta_\tau(\tau_*, u_F(\tau_*))\frac{\mathrm du_F}{\mathrm d\tau}(\tau_*)-\Delta_u(\tau_*, u_F(\tau_*))=\cL[u_F](\tau_*)=0,
	\end{align*}
	which contradicts with \eqref{Eq.U_1_compare}. This proves \eqref{Eq.U_1>u_F}.
	
	It remains to prove \eqref{Eq.U_1>u_b} and \eqref{Eq.U_1_compare} for $R$ sufficiently close to $N$. We fix $N\in\Z\cap (N_0, +\infty)$ and we assume $R\in(N, N+1/2)$ is sufficiently close to $N$. Using $\g_n=\frac8{(A+1)^2}(R-n)$, one can prove by  induction that
	\begin{align}
		|a_n|&\lesssim \wt M_n\lesssim_N1,\quad\forall\ n\in\Z\cap[1, N-1],\label{Eq.a_n<N_asymp1}\\
		a_n&\asymp M_n\asymp \wh M_n\asymp_N 1,\quad\forall\ n\in\Z\cap[\sqrt N, N-1]\label{Eq.a_n<N_asymp2}
	\end{align}
	and
	\begin{equation}\label{Eq.a_N_asymp}
		a_{N}\asymp M_N\asymp \wh M_N\asymp_N (R-N)^{-1}>0
	\end{equation}
	as $R\downarrow N$. Here, we note that the implicit constants in \eqref{Eq.a_n<N_asymp1}, \eqref{Eq.a_n<N_asymp2} and \eqref{Eq.a_N_asymp} may depend on $N$, which is fixed. We are interested in the scenario $R\downarrow N$, and the implicit constants are independent of $R$ for $R$ sufficiently close to $N$. By \eqref{Eq.a_N+1_est}, we have $-a_{N+1}\asymp A^3a_N\asymp_N(R-N)^{-1}$, hence,
    \begin{equation}\label{Eq.kappa1_asymp}
		-\kappa_1\asymp_N (R-N)^{-1},\quad\text{as}\quad R\downarrow N;
	\end{equation}
    Therefore, $-\kappa_1\geq -a_{N+1}/2\gtrsim A^3a_N>0$ for $R\in(N, N+1/2)$, and then \eqref{Eq.A_1B_1_asymp} implies that
    \begin{equation*}
        \big(A_1(N+2)+B_1\big)\kappa_1\geq -{\frac N{CA}}a_{N+1}\geq C_*A^4a_N.
    \end{equation*}
    By \eqref{Eq.A_2B_2_asymp}, taking $A_0>10$ to be larger if necessary, we have
	\begin{align}
		&\big(A_1(N+2)+B_1\big)\kappa_1+\big(A_2(N+2)+B_2\big)a_N\nonumber\\
        \geq &\ \big(C_*A^4+A_2(N+2)+B_2\big)a_N>a_N,\quad\forall\ A>A_0.\label{Eq.main_coefficient_asymp}
	\end{align}

	\underline{\it Proof of \eqref{Eq.U_1>u_b}.} Using \eqref{Eq.u_b}, we compute that
	\begin{align*}
		3(\al-\tau)\left(U_1(\tau)-u_\text b(\tau)\right)&=\tau f_1(\tau),
	\end{align*}
	where
	\begin{align*}
		f_1(\tau):&=(3\al a_1-4\al-2)+(3\al a_2-3a_1-\al+2)\tau+(3\al a_3-3a_2+1)\tau^2\\
		&\quad +\sum_{j=4}^{N-1}(3\al a_j-3a_{j-1})\tau^{j-1}+(3\al a_N-3a_{N-1})\tau^{N-1}+(3\al\kappa_1-3a_N)\tau^{N}-3\kappa_1\tau^{N+1}.
	\end{align*}
	It follows from \eqref{Eq.a_n<N_asymp1}, \eqref{Eq.a_n<N_asymp2}, \eqref{Eq.a_N_asymp} and \eqref{Eq.kappa1_asymp} that
	\begin{align*}
		&\left|(3\al a_1-4\al-2)+(3\al a_2-3a_1-\al+2)\tau_{1,2}(R)+(3\al a_3-3a_2+1)\tau_{1,2}(R)^2\right|\lesssim_N1,\\
		&\qquad\qquad\left|\sum_{j=4}^{N-1}(3\al a_j-3a_{j-1})\tau_{1,2}(R)^{j-1}\right|\lesssim_N \tau_{1,2}(R)^3\lesssim_N (R-N)^{3/N},\\
		&(3\al a_N-3a_{N-1})\tau_{1,2}(R)^{N-1}\asymp_N(R-N)^{-1}(R-N)^{(N-1)/N}\asymp_N(R-N)^{-1/N}>0,\\
		&\qquad\left|(3\al\kappa_1-3a_N)\tau_{1,2}(R)^{N}-3\kappa_1\tau_{1,2}(R)^{N+1}\right|\lesssim_N(R-N)^{-1}\tau_{1,2}(R)^{N}\lesssim_N1,
	\end{align*}
	as $R\downarrow N$. Therefore, $f_1(\tau_{1,2}(R))\asymp_N(R-N)^{-1/N}>0$ as $R\downarrow N$. This proves $U_1(\tau_{1,2}(R))>u_\text b(\tau_{1,2}(R))$ for some $\th_1\in(0,1/2)$ and all $R\in(N, N+\th_1)$. The second inequality in \eqref{Eq.U_1>u_b} follows obviously from \eqref{Eq.u_F_property}.\smallskip
	
	\underline{\it Proof of \eqref{Eq.U_1_compare}.} By Lemma \ref{Lem.recurrence_relation}, \eqref{Eq.a_n_recurrence_relation} and \eqref{Eq.a_n_recurrence1}, we compute that
	\begin{equation*}
		\cL[U_1](\tau)=\sum_{n=N+1}^{2N+2}\cU_n^{(1)}\tau^n=\tau^{N+1}g_1(\tau),
	\end{equation*}
	where $g_1(\tau):=\cU_{N+1}^{(1)}+\cU_{N+2}^{(1)}\tau+\cdots+\cU_{2N+2}^{(1)}\tau^{N+1}$, and (for $k\in\Z\cap[4, N-2]$)
	\begin{align*}
		\cU_{N+1}^{(1)}&=\delta(R-(N+1))\kappa_1-\wt E_{N+1}=\delta(R-(N+1))(\kappa_1-a_{N+1})=\delta(N+1-R),\\
		\cU_{N+2}^{(1)}&=\big(A_1(N+2)+B_1\big)\kappa_1+\big(A_2(N+2)+B_2\big)a_N+\frac32\al(N+3)\sum_{j=4}^{N-1}a_ja_{N+3-j}\\
		&\qquad-\frac{3N+2}{2}\sum_{j=3}^{N-1}a_ja_{N+2-j},\\
		\cU_{N+3}^{(1)}&=\big(A_2(N+3)+B_2\big)\kappa_1+\frac32\al(N+4)\sum_{j=4}^{N}a_ja_{N+4-j}-\frac{3N+5}{2}\sum_{j=3}^Na_ja_{N+3-j},\\
		\cU_{N+k}^{(1)}&=\left[3\al(N+k+1)a_k-(3N+3k-4)a_{k-1}\right]\kappa_1+\frac32\al(N+k+1)\sum_{j=k+2}^{N-1}a_ja_{N+k+1-j}\\
		&\quad+\left[3\al(N+k+1)a_{k+1}-(3N+3k-4)a_{k}\right]a_N -\frac{3N+3k-4}{2}\sum_{j=k+1}^{N-1}a_ja_{N+k-j},\\
		\cU_{2N-1}^{(1)}&=3\al Na_N^2+\left[6\al Na_{N-1}-(6N-7)a_{N-2}\right]\kappa_1-(6N-7)a_{N-1}a_N,\\
		\cU_{2N}^{(1)}&=3\al(2N+1)a_N\kappa_1-(3N-2)a_N^2-2(3N-2)a_{N-1}\kappa_1,\\
		\cU_{2N+1}^{(1)}&=3\al(N+1)\kappa_1^2-(6N-1)a_N\kappa_1,\quad \cU_{2N+2}^{(1)}=-(3N+1)\kappa_1^2.
	\end{align*}
	According to \eqref{Eq.a_n<N_asymp1}, \eqref{Eq.a_n<N_asymp2}, \eqref{Eq.a_N_asymp}, \eqref{Eq.kappa1_asymp}, \eqref{Eq.main_coefficient_asymp} and $\delta>0$, we have
	\begin{align*}
		\cU_{N+1}^{(1)}&>0,\\ \cU_{N+2}^{(1)}&>a_N+\frac32\al(N+3)\sum_{j=4}^{N-1}a_ja_{N+3-j}-\frac{3N+2}{2}\sum_{j=3}^{N-1}a_ja_{N+2-j}\gtrsim_N(R-N)^{-1},\\
		\left|\cU_{N+k}^{(1)}\right|&\lesssim_N(R-N)^{-1},\quad\forall\ k\in\Z\cap[3,N-2],\\
		\cU_{2N-1}^{(1)}&\asymp_N(R-N)^{-2},\quad\left|\cU_{2N+k}^{(1)}\right|\lesssim_N(R-N)^{-2},\quad\forall\ k\in\Z\cap[0, 2],
	\end{align*}
	as $R\downarrow N$. As a consequence, for $\tau\in(0, \tau_{1,2}(R))$, we have $0<\tau\ll1$ and
	\[g_1(\tau)\gtrsim_N \cU_{N+2}^{(1)}\tau+\cU_{2N-1}^{(1)}\tau^{N-2}>0.\]
	This completes the proof of \eqref{Eq.U_1_compare}.
\end{proof}

Next, we turn to prove Proposition \ref{Prop.R_near_even}. We need the following lemma.

\begin{lemma}\label{Lem.u_2_compare}
	Let $u_{(2)}(\tau):=a_0+a_1\tau+a_2\tau^2$. There exists $N_0\in\Z_+$ such that for all $R>N_0$ we have
	\begin{equation}\label{Eq.u_2<u_F}
		u_{(2)}(\tau)<u_F(\tau),\quad\ \forall\ \tau\in(0,\al).
	\end{equation}
\end{lemma}
\begin{proof}
	By Lemma \ref{Lem.recurrence_relation} and \eqref{Eq.a_n_recurrence_relation}, we compute that
	\begin{align*}
		\cL\left[u_{(2)}\right](\tau)=\cV\tau^3+\cW\tau^4,\quad\cV:=-\delta(R-3)a_3,\ \cW:=-4a_2^2-4a_2/3.
	\end{align*}
	Using $\delta=2(1-\sqrt\al)^2$, the definition of $R$ in \eqref{Eq.R} and $\lim_{R\to+\infty}\al(R)=1$, we obtain
	\[\lim_{R\to+\infty}\delta(R)(R-3)=\lim_{R\to+\infty}\delta(R)R=2\lim_{\al\to1}(1-\sqrt\al)^2\frac{(\sqrt\al+1)^2}{(\sqrt\al-1)^2}=8,\]
	hence (by Remark \ref{Rmk.a_n_positive}),
	\[\lim_{R\to+\infty}\cV=-8a_3^\infty<0;\]
	moreover, we also have $\lim_{R\to+\infty}\cW=-4(a_2^\infty)^2-4a_2^\infty/3<0$. Therefore, there exists $N_0\in\Z_+$ such that for all $R>N_0$, we have
	\begin{equation}\label{Eq.u_2_compare}
		\cL\left[u_{(2)}\right](\tau)<0,\quad\forall\ \tau>0.
	\end{equation}
	
	Since $\lim_{\tau\uparrow\al}u_F(\tau)=+\infty$ and $u_{(2)}$ is bounded on any finite interval, there exists $\wt\tau_*\in(0,\al)$ such that $u_{(2)}(\tau)<u_F(\tau)$ for all $\tau\in[\wt\tau_*, \al)$. Assume, on the contrary, that \eqref{Eq.u_2<u_F} does not hold. Then by the continuity, we can find $\tau_*\in(0, \wt\tau_*)$ such that $u_{(2)}(\tau_*)=u_F(\tau_*)$ and $u_{(2)}(\tau)< u_F(\tau)$ for all $\tau\in(\tau_*, \al)$, hence $\frac{\mathrm du_{(2)}}{\mathrm d\tau}(\tau_*)\leq\frac{\mathrm du_F}{\mathrm d\tau}(\tau_*)$. It follows from $\Delta_\tau\left(\tau_*, u_{(2)}(\tau_*)\right)=\Delta_\tau(\tau_*, u_F(\tau_*))<0$ that
	\begin{align*}
		\cL\left[u_{(2)}\right](\tau_*)&=\Delta_\tau\left(\tau_*, u_{(2)}(\tau_*)\right)\frac{\mathrm du_{(2)}}{\mathrm d\tau}(\tau_*)-\Delta_u\left(\tau_*, u_{(2)}(\tau_*)\right)\\
		&\geq \Delta_\tau\left(\tau_*, u_F(\tau_*)\right)\frac{\mathrm du_{F}}{\mathrm d\tau}(\tau_*)-\Delta_u\left(\tau_*, u_{F}(\tau_*)\right)=\cL[u_F](\tau_*)=0,
	\end{align*}
	which contradicts  \eqref{Eq.u_2_compare}. This proves \eqref{Eq.u_2<u_F}.
\end{proof}

Now, we are ready to prove Proposition \ref{Prop.R_near_even}.

\begin{proof}[Proof of Proposition \ref{Prop.R_near_even}]
	Let $A_0>10$ be large enough such that all properties in Section \ref{Sec.an} hold. Let $N_0\in\N$ be such that $N_0>A_0^2$ and we fix an $N\in\Z\cap (N_0, +\infty)$. Let $R\in(N, N+1)$ and
	\begin{equation*}
		U_2(\tau):=U_2(\tau; R)=\sum_{j=0}^Na_j\tau^j+\frac12a_{N+1}\tau^{N+1}+\kappa_2\tau^{N+2},\quad\ \kappa_2:=(N+1-R)^{-1-1/N}.
	\end{equation*}
	Since $a_{N+1}<0$ by Lemma \ref{Lem.a_N+1<0}, for each $R\in(N, N+1)$, there exists $\tau_{2,1}(R)>0$ such that
	\begin{equation}\label{Eq.u_L<U_2}
		u_{L}(\tau; R)<U_2(\tau; R)\quad\forall\ \tau\in(0, \tau_{2,1}(R)].
	\end{equation}
	We define $\tau_{2,2}(R)=(N+1-R)^{1/(N-1)}$ for $R\in(N, N+1)$, and we claim that, by taking $A_0>10$ to be larger if necessary, there exists $\th_2\in(0,1/2)$ such that for all $R\in(N+1-\th_2, N+1)$, we have
	\begin{equation}\label{Eq.U_2<u_F}
		U_2(\tau; R)<u_F(\tau;R)\quad\forall\ \tau\in(0, \tau_{2,2}(R)].
	\end{equation}
	Assuming \eqref{Eq.U_2<u_F}, for $R\in(N+1-\th_2, N+1)$, we take $\tau_2(R):=\min\{\tau_{2,1}(R), \tau_{2,2}(R)\}>0$. Then \eqref{Eq.u_L<u_F} follows from \eqref{Eq.u_L<U_2}.
	
	It suffices to prove \eqref{Eq.U_2<u_F}, which is a consequence of the following two inequalities:
	\begin{align}
		U_2&(\tau_{2,2}(R))<u_{(2)}(\tau_{2,2}(R))<u_F(\tau_{2,2}(R)),\label{Eq.U_2<u_g}\\
		&\cL[U_2](\tau)<0\quad\forall\ \tau\in(0, \tau_{2,2}(R)],\label{Eq.U_2_compare}
	\end{align}
	where $\cL$ is defined in \eqref{Eq.Lu}. 
	Now we prove \eqref{Eq.U_2<u_F} by using \eqref{Eq.U_2<u_g} and \eqref{Eq.U_2_compare}. Assume, on the contrary, that \eqref{Eq.U_2<u_F} does not hold for some $R\in(N+1-\th_2, N+1)$. By \eqref{Eq.U_2<u_g} and continuity, there exists $\tau_*\in(0, \tau_{2,2}(R))$ such that $U_2(\tau_*)=u_F(\tau_*)$ and $U_2(\tau)<u_F(\tau)$ for $\tau\in(\tau_*, \tau_{2,2}(R))$, hence $\frac{\mathrm dU_2}{\mathrm d\tau}(\tau_*)\leq \frac{\mathrm du_F}{\mathrm d\tau}(\tau_*)$. Since $\Delta_\tau(\tau_*, U_2(\tau_*))=\Delta_\tau(\tau_*, u_F(\tau_*))<0$, we have
	\begin{align*}
		\cL[U_2](\tau_*)&=\Delta_\tau(\tau_*, U_2(\tau_*))\frac{\mathrm dU_2}{\mathrm d\tau}(\tau_*)-\Delta_u(\tau_*, U_2(\tau_*))\\
		&\geq \Delta_\tau(\tau_*, u_F(\tau_*))\frac{\mathrm du_F}{\mathrm d\tau}(\tau_*)-\Delta_u(\tau_*, u_F(\tau_*))=\cL[u_F](\tau_*)=0,
	\end{align*}
	which contradicts \eqref{Eq.U_2_compare}. This proves \eqref{Eq.U_2<u_F}.
	
	It remains to prove \eqref{Eq.U_2<u_g} and \eqref{Eq.U_2_compare} for $R\in(N, N+1)$ sufficiently close to $N+1$.
We fix $N\in\Z\cap (N_0, +\infty)$ and we assume $R\in(N+1/2, N+1)$ is sufficiently close to $N+1$. Using $\g_n=\frac8{(A+1)^2}(R-n)$, one can prove by  induction that
	\begin{align}
		|a_n|&\lesssim \wt M_n\lesssim_N1,\quad\forall\ n\in\Z\cap[1, N],\label{Eq.a_n<N_asymp3}\\
		a_n&\asymp M_n\asymp \wh M_n\asymp_N 1,\quad\forall\ n\in\Z\cap[\sqrt N, N]\label{Eq.a_n<N_asymp4}
	\end{align}
	and (thanks to \eqref{Eq.a_N+1_est})
	\begin{equation}\label{Eq.a_N+1_asymp}
		-a_{N+1}\asymp \frac{A^3}{N+1-R}a_N\asymp_N (N+1-R)^{-1}
	\end{equation}
	as $R\uparrow N+1$.

	\underline{\it Proof of \eqref{Eq.U_2<u_g}.} We compute that
	\[U_2(\tau)-u_{(2)}(\tau)=\tau^3f_2(\tau),\]
	where
	\begin{equation*}
		f_2(\tau):=a_3+a_4\tau+\cdots+a_N\tau^{N-3}+a_{N+1}\tau^{N-2}/2+\kappa_2\tau^{N-1}.
	\end{equation*}
	It follows from \eqref{Eq.a_n<N_asymp3}, \eqref{Eq.a_N+1_asymp} and $\tau_{2,2}(R)=(N+1-R)^{1/(N-1)}$ that
	\begin{align*}
		&\qquad\qquad\qquad\left|a_3+a_4\tau_{2,2}(R)+\cdots+a_N\tau_{2,2}(R)^{N-3}\right|\lesssim_N1,\\
		&\qquad-a_{N+1}\tau_{2,2}(R)^{N-2}/2\asymp_N(N+1-R)^{-1+\frac{N-2}{N-1}}\asymp_N(N+1-R)^{-1/(N-1)},\\
		&\kappa_2\tau_{2,2}(R)^{N-1}=(N+1-R)^{-1/N}\ll (N+1-R)^{-1/(N-1)}\asymp_N-a_{N+1}\tau_{2,2}(R)^{N-2}/2,
	\end{align*}
	as $R\uparrow N+1$. Therefore, $f_2(\tau_{2,2}(R))<0$ for $R$ sufficiently close to $N+1$. This proves $U_2(\tau_{2,2}(R))<u_{(2)}(\tau_{2,2}(R))$ for some $\th_2\in(0,1/2)$ and for all $R\in(N+1-\th_2, N+1)$. The second inequality in \eqref{Eq.U_2<u_g} follows from Lemma \ref{Lem.u_2_compare}.
	
	\underline{\it Proof of \eqref{Eq.U_2_compare}.} By Lemma \ref{Lem.recurrence_relation}, \eqref{Eq.a_n_recurrence_relation} and \eqref{Eq.a_n_recurrence1}, we compute that
	\begin{equation}\label{Eq.L[U_2]}
		\cL[U_2](\tau)=\sum_{n=N+1}^{2N+4}\cU_n^{(2)}\tau^n=\tau^{N+1}g_2(\tau),
	\end{equation}
	where $g_2(\tau):=\cU_{N+1}^{(2)}+\cU_{N+2}^{(2)}\tau+\cdots+\cU_{2N+4}^{(2)}\tau^{N+3}$, and (for $k\in\Z\cap[5, N-2]$)
	\begin{align*}
		\cU_{N+1}^{(2)}&=\delta(R-(N+1))a_{N+1}/2-\wt E_{N+1}=\delta(R-(N+1))(a_{N+1}/2-a_{N+1})\\
		&=\delta(N+1-R)a_{N+1}/2<0,\\
		\cU_{N+2}^{(2)}&=\delta(R-(N+2))\kappa_2+\big(A_1(N+2)+B_1\big)a_{N+1}/2+\big(A_2(N+2)+B_2\big)a_N\\
		&\qquad +\frac32\al(N+3)\sum_{j=4}^{N-1}a_ja_{N+3-j}-\frac{3N+2}{2}\sum_{j=3}^{N-1}a_ja_{N+2-j},\\
		\cU_{N+3}^{(2)}&=\big(A_1(N+3)+B_1\big)\kappa_2+\big(A_2(N+3)+B_2\big)a_{N+1}/2\\
		&\qquad+\frac32\al(N+4)\sum_{j=4}^{N}a_ja_{N+4-j}-\frac{3N+5}{2}\sum_{j=3}^{N}a_ja_{N+3-j},\\
		\cU_{N+4}^{(2)}&=\big(A_2(N+4)+B_2\big)\kappa_2+\left[3\al(N+5)a_4-(3N+8)a_3\right]a_{N+1}/2\\
		&\qquad+\frac32\al(N+5)\sum_{j=5}^{N}a_ja_{N+5-j}-\frac{3N+8}{2}\sum_{j=4}^{N}a_ja_{N+4-j},\\
		\cU_{N+k}^{(2)}&=\left[3\al(N+k+1)a_{k-1}-(3N+3k-4)a_{k-2}\right]\kappa_2+\frac32\al(N+k+1)\sum_{j=k+1}^Na_ja_{N+k+1-j}\\
		&\quad+\left[3\al(N+k+1)a_k-(3N+3k-4)a_{k-1}\right]\frac{a_{N+1}}2-\frac{3N+3k-4}{2}\sum_{j=k}^Na_ja_{N+k-j},\\
		\cU_{2N-1}^{(2)}&=\left[6\al Na_{N-2}-(6N-7)a_{N-3}\right]\kappa_2+\left[6\al Na_{N-1}-(6N-7)a_{N-2}\right]a_{N+1}/2\\
		&\qquad+3\al Na_N^2-(6N-7)a_{N-1}a_N,\\
		\cU_{2N}^{(2)}&=\left[3\al(2N+1)a_{N-1}-2(3N-2)a_{N-2}\right]\kappa_2\\
		&\qquad+\left[3\al(2N+1)a_N-2(3N-2)a_{N-1}\right]a_{N+1}/2-(3N-2)a_N^2,\\
		\cU_{2N+1}^{(2)}&=\frac{3\al(N+1)}{4}a_{N+1}^2+\left[6\al(N+1)a_N-(6N-1)a_{N-1}\right]\kappa_2-\frac{6N-1}{2}a_Na_{N+1},\\
		\cU_{2N+2}^{(2)}&=-(3N+1)a_{N+1}^2/4+3\al(2N+3)a_{N+1}\kappa_2/2-2(3N+1)a_N\kappa_2,\\
		\cU_{2N+3}^{(2)}&=3\al(N+2)\kappa_2^2-(6N+5)a_{N+1}\kappa_2/2,\quad \cU_{2N+4}^{(2)}=-(3N+4)\kappa_2^2.
	\end{align*}
	According to \eqref{Eq.a_n<N_asymp3}, \eqref{Eq.a_n<N_asymp4}, \eqref{Eq.a_N+1_asymp}, $\kappa_2=(N+1-R)^{-1-1/N}$ and $\delta\asymp_N1>0$, we have
	\begin{align*}
		\cU_{N+1}^{(2)}&<0,\\
		-\cU_{N+2}^{(2)}&\geq\delta(N+2-R)\kappa_2{-C(N+1-R)^{-1}}
%\big(A_2(N+2)+B_2\big)a_N-\frac32\al(N+3)\sum_{j=4}^{N-1}a_ja_{N+3-j}\\
%		&\qquad\qquad+\frac{3N+2}{2}\sum_{j=3}^{N-1}a_ja_{N+2-j}
\gtrsim_N(N+1-R)^{-1-1/N},\\
		\left|\cU_{N+k}^{(2)}\right|&\lesssim_N(N+1-R)^{-1-1/N},\quad\forall\ k\in\Z\cap[3, N],\\
		\left|\cU_{2N+1}^{(2)}\right|&\lesssim_N(N+1-R)^{-2},\quad \left|\cU_{2N+2}^{(2)}\right|\lesssim_N(N+1-R)^{-2-1/N},\\
		\left|\cU_{2N+3}^{(2)}\right|&+\left|\cU_{2N+4}^{(2)}\right|\lesssim_N(N+1-R)^{-2-2/N},
	\end{align*}
	as $R\uparrow N+1$. Therefore, $g_2(\tau)<\tau\wt g_2(\tau)$ for $\tau\in(0, \tau_{2,2}(R)]$, where
	\[\wt g_2(\tau):=\cU_{N+2}^{(2)}+\cU_{N+3}^{(2)}\tau+\cdots+\cU_{2N+4}^{(2)}\tau^{N+2};\]
	and
	\begin{align*}
		&\left|\cU_{N+3}^{(2)}\tau+\cdots+\cU_{2N+4}^{(2)}\tau^{N+2}\right|\lesssim_N(N+1-R)^{-1-\frac1N+\frac1{N-1}}+(N+1-R)^{-2+1}\\
		&\qquad\qquad\qquad+(N+1-R)^{-2-\frac1N+\frac N{N-1}}+(N+1-R)^{-2-\frac2N+\frac{N+1}{N-1}}\\
		&\qquad\lesssim_N(N+1-R)^{-1}\ll -\cU_{N+2}^{(2)},
	\end{align*}
	for all $\tau\in(0, \tau_{2,2}(R)]$, as $R\uparrow N+1$. Hence, $\wt g_2(\tau)<0$, and thus $g_2(\tau)<0$ for $\tau\in(0, \tau_{2,2}(R)]$. Then \eqref{Eq.U_2_compare} follows from \eqref{Eq.L[U_2]}.
\end{proof}

\begin{proposition}\label{Prop.sol_sonic_point}
	There exists $N_0\in\Z_+$ such that for all $N\in\Z\cap(N_0, +\infty)$, we can find $R_N\in(N, N+1)$ such that $u_L(\cdot; R_N)\equiv u_F(\cdot;R_N)$. As a result, for these special $R=R_N$, the $Q_6-Q_2$ solution curve $u_F$ of the $\tau-u$ ODE \eqref{ODE_u_tau} can be extended to pass through $Q_2$ smoothly.
\end{proposition}
\begin{proof}
	Let $N_0\in\Z_+$ be large enough such that both Proposition \ref{Prop.R_near_odd} and Proposition \ref{Prop.R_near_even} hold. Let $N\in\Z\cap(N_0, +\infty)$ and let $\th_1,\th_2\in(0,1/2)$ be given by Proposition \ref{Prop.R_near_odd} and Proposition \ref{Prop.R_near_even}, respectively. Let $R_{N, 1}:=N+\th_1/2\in(N, N+1)$, $R_{N,2}:=N+1-\th_2/2\in(N, N+1)$ and $\tau^*:=\min\{\tau_1(R_{N,1}), \tau_2(R_{N,2}), \tau_0([R_{N,1} R_{N,2}])\}>0$, where $\tau_1(R_{N,1})>0$ is given by Proposition \ref{Prop.R_near_odd}, $\tau_2(R_{N,2})>0$ is given by Proposition \ref{Prop.R_near_even}, and $\tau_0([R_{N,1} R_{N,2}])>0$ is given by Proposition \ref{Prop.local_sol}. Proposition \ref{Prop.R_near_odd} implies that $u_L(\tau^*; R_{N,1})>u_F(\tau^*; R_{N,1})$, and Proposition \ref{Prop.R_near_even} implies that $u_L(\tau^*; R_{N,2})<u_F(\tau^*; R_{N,2})$. By Proposition \ref{Prop.local_sol} and Remark \ref{Rmk.u_F_continuity}, we know that the function
	\[[R_{N,1}, R_{N,2}]\ni R\mapsto u_L(\tau^*; R)-u_F(\tau^*; R)\]
	is continuous. By the intermediate value theorem, there exists $R_{N}\in(R_{N,1}, R_{N,2})\subset(N, N+1)$ such that $u_L(\tau^*; R_{N})=u_F(\tau^*; R_{N})$. Therefore, it follows from the uniqueness of solutions to the $\tau-u$ ODE \eqref{ODE_u_tau} with $u(\tau^*)=u_F(\tau^*)$ that $u_L(\cdot; R_N)=u_F(\cdot; R_N)$.
\end{proof}

\section{Global extension of the solution curve}\label{Global extension of the solution curve}
In the preceding sections, we have established the existence of a sequence $\{R_N\}_{N\in\Z_{>N_0}}$ for some sufficiently large $N_0\in\Z_+$ such that: for each $N\in\Z\cap(N_0, +\infty)$, we have $R_N\in(N, N+1)$ and the $P_6-P_2$ solution $w_F(\cdot; R_N)=w_F(\sigma; R_N)$ crosses the sonic point $P_2$ smoothly. Consequently, it is a smooth solution defined on $\sigma\in[\sigma_2-\eta_{0}, +\infty)$ for some small enough $\eta_0=\eta_0(N)>0$.

In this section, we prove that if $N$ is an odd integer, then the extended solution will leave the region between $Q_2$ and $Q_5$ by crossing the red curve $\sigma\mapsto w_2(\sigma)$ solving $\Delta_1=0$, and then it can be extended to $P_4$, where $\sigma=0$. We denote the region between $P_2$ and $P_5$ by
\begin{equation}
	\cR:=\{(\sigma, w): w_2(\sigma)<w<w_-(\sigma),\ \sigma(P_5)=\sqrt 3r/6<\sigma<\sigma_2\}.
\end{equation}

Our approach is to employ the barrier function method. In Subsection \ref{Subsec.barrier_functions}, we establish the key barrier property of the polynomial $u_{(N)}(\tau)=\sum_{n=0}^Na_n\tau^n$; see Proposition \ref{Prop.u_N_compare}. We also determine the location of the graph of $u_{(N)}$ in Lemma \ref{Lem.position}. In Subsection \ref{Subsec.global_extension}, we prove that if $N$ is odd, then the local solution $w_L$ to the $\sigma-w$ ODE \eqref{ODE_w_sigma} near $P_2$ can be extended to the left,  exiting the region between $P_5$ and $P_2$ by crossing the red curve $\sigma\mapsto w_2(\sigma)$, and then be further extended until reaching the origin $P_4$. This is based on standard ODE theory and the key barrier property established in Subsection \ref{Subsec.barrier_functions}. Finally, in Subsection \ref{Subsec.Proof_Thm_ODE}, we prove Theorem \ref{Thm.ODE}.

\subsection{Barrier functions}\label{Subsec.barrier_functions}
We need the following upper-bound estimate on $\wh M_{\lceil R\rceil-2}$ and lower-bound estimate on $\wh M_{\lceil R\rceil-1}$.
\begin{lemma}\label{Lem.M_N-1upper_bound}
	There exist $N_0\in\Z_+$ and a constant $C_0>0$ such that
	\begin{equation*}
		\wh M_{N-1}\leq \left(C_0\sqrt N\right)^N,\quad\forall\ R\in(N, N+1),\ \forall\ N\in\Z\cap(N_0, +\infty).
	\end{equation*}
\end{lemma}
\begin{proof}
	By \eqref{Eq.hat_M_n}, we have
	\begin{equation*}
		\wh M_{N-1}=a_1\prod_{k=2}^{N-1}\frac{\mu_{k-1}}{\g_k}.
	\end{equation*}
	Using \eqref{Eq.mu_n^*}, $\g_k=\frac{8}{(A+1)^2}(R-k)$ and $R\in(N, N+1)$ yields that
	\begin{align*}
		\frac{\mu_{k-1}}{\g_k}\asymp \frac{A^2\sqrt k}{R-k}\asymp \frac{N\sqrt N}{N-k},\quad\forall\ k\in\Z\cap[2,N-1].
	\end{align*}
	Hence, there exists an absolute constant $C_1>1$ such that
	\begin{align*}
		\wh M_{N-1}\lesssim \prod_{k=2}^{N-1}\frac{C_1N^{3/2}}{N-k}\lesssim \frac{\left(C_1N^{3/2}\right)^{N-2}}{(N-2)!}\lesssim \frac{N\left(C_1N^{3/2}\right)^{N-2}}{N^{N-1/2}\e^{-N}}\lesssim (C_1\e)^N N^{N/2-3/2}.
	\end{align*}
    This completes the proof.
	\end{proof}

\begin{lemma}\label{Lem.M_N_lower_bound}
	There exists $N_0\in\Z$ such that
	\begin{equation*}
		\wh M_N\gtrsim \left(5\sqrt N/8\right)^N,\quad\forall\ R\in(N, N+1),\ \forall\ N\in\Z\cap(N_0, +\infty).
	\end{equation*}
\end{lemma}
\begin{proof}
	By \eqref{Eq.M_n_sequence} and $q_n>0$ due to \eqref{Eq.q_n_asymp}, we have $\g_kM_k=2p_{k}M_{k-1}+q_kM_{k-2}\geq 2p_kM_{k-1}$ for $k\in\Z\cap[2, N]$. Combining this with \eqref{Eq.p_n}, \eqref{Eq.p_n_asymp} and $\g_k=\frac{8}{(A+1)^2}(R-k)$ yields that
	\begin{align*}
		\frac{M_k}{M_{k-1}}\geq \frac{2p_k}{\g_k}\geq \frac{5k/A}{\frac8{A^2}(R-k)}=\frac58A\frac{k}{R-k},\ \forall\ k\in\Z\cap[4, N],\  \frac{M_k}{M_{k-1}}\geq\frac12A\frac{k}{R-k}\ \text{for}\ k\in\{2,3\}.
	\end{align*}
	Hence,
	\begin{align*}
		M_N&=M_1\prod_{k=2}^N\frac{M_k}{M_{k-1}}\gtrsim M_1\left(\frac58A\right)^{N-1}\prod_{k=2}^N\frac{k}{R-k}=a_1(R-1)\left(\frac58A\right)^{N-1}\prod_{k=1}^{N}\frac{k}{R-k}\\
		&\gtrsim \left(\frac58\sqrt N\right)^{N}\prod_{k=1}^{N}\frac{k}{R-(N+1-k)}\gtrsim \left(\frac58\sqrt N\right)^{N}\prod_{k=1}^{N}\frac{k}{k+R-(N+1)}\\
		&\gtrsim \left(5\sqrt N/8\right)^N,
	\end{align*}
	since $k>k+R-(N+1)$, thanks to $R\in(N, N+1)$.
\end{proof}

\begin{proposition}\label{Prop.u_N_compare}
	Let $u_{(N)}(\tau)=a_0+a_1\tau+\cdots+a_N\tau^N$. There exists $N_0\in\Z_+$ such that if $N\in(N_0, +\infty)\cap(2\Z+1)$ and $R\in(N, N+1)$, then
	\begin{equation}\label{Eq.u_N_compare}
		\cL\left[u_{(N)}\right](\tau)<0,\quad\forall\ \tau\in\left[-4/\sqrt R, 0\right).
	\end{equation}
\end{proposition}
\begin{proof}
	Let $A_0>10$ be large enough such that all properties in Section \ref{Sec.an} hold. Let $N_0\in\N$ be such that $N_0>A_0^2$ and we fix an $N\in(N_0, +\infty)\cap(2\Z+1)$.
	By Lemma \ref{Lem.recurrence_relation}, \eqref{Eq.a_n_recurrence_relation} and \eqref{Eq.a_n_recurrence1}, we compute that
	\begin{equation}\label{Eq.L[u_N]}
		\cL\left[u_{(N)}\right](\tau)=\tau^{N+1}f_N(\tau),\quad f_N(\tau):=\cU_{N, N+1}+\cU_{N, N+2}\tau+\cdots+\cU_{N, 2N}\tau^{N-1},
	\end{equation}
	where
	\begin{align*}
		\cU_{N, N+1}&=-\wt E_{N+1}=\delta(N+1-R)a_{N+1}=-\g_{N+1}a_{N+1},\\
		\cU_{N, N+2}&=\big(A_2(N+2)+B_2\big)a_N+\frac32\al(N+3)\sum_{j=4}^{N-1}a_ja_{N+3-j}-\frac{3N+2}{2}\sum_{j=3}^{N-1}a_ja_{N+2-j},\\
		\cU_{N, N+k}&=\frac32\al(N+k+1)\sum_{j=k+1}^Na_ja_{N+k+1-j}-\frac{3N+3k-4}{2}\sum_{j=k}^Na_ja_{N+k-j},\ \ \forall\ k\in\Z\cap[3, N].
	\end{align*}
	
	\underline{\bf Step 1.} In this step, we deal with the first term in $f_N$. We claim that, by taking $N_0$ to be larger if necessary, we have
	\begin{equation}\label{Eq.U_N,N+1_est}
		-\cU_{N, N+1}>21\sqrt N a_N/4>0,\quad\forall\ N\in\Z\cap(N_0, +\infty),\ \forall\ R\in(N, N+1).
	\end{equation}
	Indeed, by \eqref{Eq.a_n_new_recurrence} and \eqref{Eq.epsilon_n_def}, we have
	\begin{equation*}
		-\cU_{N, N+1}=\g_{N+1}a_{N+1}=2p_{N+1}a_N+q_{N+1}a_{N-1}+\varepsilon_{N+1}.
	\end{equation*}
	It follows from \eqref{Eq.q_n}, \eqref{Eq.q_n_asymp}, Proposition \ref{Prop.a_n_sharp_est}, \eqref{Eq.hat_M_n_est}, \eqref{Eq.mu_n^*} and $\g_n=\frac8{(A+1)^2}(R-n)$ that $q_{N+1}\asymp N$ and
	\begin{align}\label{Eq.M_{N-1}est}
		|a_{N-1}|\lesssim M_{N-1}\asymp \wh M_{N-1}= \frac{\g_N}{\mu_{N-1}}\wh M_N\asymp \frac{A^{-2}(R-N)}{\sqrt N}\wh M_N\lesssim N^{-3/2}M_N\asymp N^{-3/2}a_N,
	\end{align}
	hence, $|q_{N+1}a_{N-1}|\lesssim N^{-1/2}a_N$. By \eqref{Eq.epsilon_n_est}, we similarly have
	\begin{align}\label{Eq.M_{N-2}est}
		\left|\varepsilon_{N+1}\right|\lesssim \wh M_{N-2}=\frac{\g_{N-1}\g_N}{\mu_{N-2}\mu_{N-1}}\wh M_N\lesssim N^{-3}\wh M_N\asymp N^{-3}M_N\asymp N^{-3}a_N.
	\end{align}
	It follows from \eqref{Eq.p_n} and \eqref{Eq.p_n_asymp} that
	\[2p_n>\frac{23}4\frac nA-\frac4A\Longrightarrow 2p_{N+1}>\frac{23}4\frac{N+1}A-\frac4A>\frac{23}4\frac NA>\frac{11}2\sqrt N.\]
	Thus,
	\begin{align*}
		-\cU_{N, N+1}=\g_{N+1}a_{N+1}>11\sqrt Na_N/2-C(N^{-1/2}+N^{-3})a_N>21\sqrt Na_N/4>0.
	\end{align*}
	This proves \eqref{Eq.U_N,N+1_est}.
	
	\underline{\bf Step 2.} In this step, we deal with the second term in $f_N$. We claim that, by taking $N_0$ to be larger if necessary, we have
	\begin{equation}\label{Eq.U_N,N+2_est}
		-\cU_{N, N+2}<5N a_N/4,\quad\forall\ N\in\Z\cap(N_0, +\infty),\ \forall\ R\in(N, N+1).
	\end{equation}
	Indeed, by \eqref{Eq.A_2B_2_asymp}, we have
	\[-(A_2(N+2)+B_2)<10(N+2)/9-4<9N/8.\]
	On the other hand, similarly to the proof of \eqref{Eq.epsilon_n_est}, using the convexity of $\{M_n^*\}$ and \eqref{Eq.M_{N-1}est}, \eqref{Eq.M_{N-2}est}, we obtain
	\begin{align*}
		&\left|\frac32\al(N+3)\sum_{j=4}^{N-1}a_ja_{N+3-j}-\frac{3N+2}{2}\sum_{j=3}^{N-1}a_ja_{N+2-j}\right|\\
		\lesssim &\ N \left(M_4^*M_{N-1}^*+NM_5^*M_{N-2}^*\right)+N\left(M_3^*M_{N-1}^*+NM_4^*M_{N-2}^*\right)\\
		\lesssim &\ NM_{N-1}^*+N^2M_{N-2}^*\lesssim N\wh M_{N-1}+N^2\wh M_{N-2}\\
		\lesssim &\ N\cdot N^{-3/2}a_N+N^2\cdot N^{-3}a_N\lesssim N^{-1/2}a_N.
	\end{align*}
	Therefore, $-\cU_{N, N+2}<9Na_N/8+CN^{-1/2}a_N\leq 5Na_N/4$. This proves \eqref{Eq.U_N,N+2_est}.
	
	\underline{\bf Step 3.} In this step, we deal with the last two terms in $f_N$. We compute that
	\begin{align*}
		\cU_{N, 2N-1}=3\al Na_N^2-(6N-7)a_Na_{N-1},\quad \cU_{N, 2N}=-(3N-2)a_N^2<0.
	\end{align*}
	By Proposition \ref{Prop.a_n_sharp_est} and \eqref{Eq.M_{N-1}est}, we have
	\begin{equation*}
		|(6N-7)a_Na_{N-1}|\lesssim Na_N\wh M_{N-1}\lesssim Na_N\cdot N^{-3/2}a_N\lesssim N^{-1/2}a_N^2.
	\end{equation*}
	Hence, $\cU_{N, 2N-1}>3\al Na_N^2-CN^{-1/2}a_N^2>2Na_N^2>0$. Therefore,
	\begin{equation}\label{Eq.f_N_last}
		\cU_{N, 2N-1}\tau^{N-2}+\cU_{N, 2N}\tau^{N-1}<2Na_N^2\tau^{N-2}<0\quad\forall\ \tau<0,\ \forall\ N\in(N_0, +\infty)\cap(2\Z+1),
	\end{equation}
	for all $R\in(N, N+1)$.
	
	\underline{\bf Step 4.} In this step, we deal with the intermediate terms in $f_N$. We claim that
	\begin{equation}\label{Eq.f_N_middle}
		\left|\sum_{k=3}^{N-2}\cU_{N, N+k}\tau^{k-1}\right|\lesssim \frac1{\sqrt N}\left(\sqrt Na_N+Na_N^2|\tau|^{N-2}\right),\quad\forall\ \tau\in\left[-\frac4{\sqrt R}, 0\right).
	\end{equation}
	To prove \eqref{Eq.f_N_middle}, we first estimate $\cU_{N, N+k}$. Let $k\in\Z\cap[3, N-2]$. Similarly to the proof of \eqref{Eq.epsilon_n_est}, using the convexity of $\{M_n^*\}$, \eqref{Eq.M_n^*_increasing} and \eqref{Eq.hat_M_n_est}, we have
	\begin{align*}
		\left|\cU_{N, N+k}\right|&\lesssim N\left(M_{k+1}^*M_{N}^*+(N-k-2)M_{k+2}^*M_{N-1}^*\right)+N\left(M_k^*M_N^*+NM_{k+1}^*M_{N-1}^*\right)\\
		&\lesssim N\left(M_{k+1}^*M_{N}^*+(N-k-2)M_{k+2}^*M_{N-1}^*\right)+N\left(M_{k+1}^*M_N^*+NM_{k+1}^*M_{N-1}^*\right)\\
		&\lesssim N\left(\wh M_{k+1}\wh M_N+(N-k-2)\wh M_{k+2}\wh M_{N-1}\right)+N\left(\wh M_{k+1}\wh M_N+N\wh M_{k+1}\wh M_{N-1}\right).
	\end{align*}
	It follows from \eqref{Eq.hat_M_n}, \eqref{Eq.mu_n^*} and $\g_n=\frac8{(A+1)^2}(R-n)$ that
	\begin{align*}
		\wh M_{k+2}=\frac{\mu_{k+1}}{\g_{k+2}}\wh M_{k+1}\asymp \frac{\sqrt k}{A^{-2}(R-k-2)}\wh M_{k+1}\lesssim \frac{A^3}{N-k-2}\wh M_{k+1}.
	\end{align*}
	Hence, \eqref{Eq.M_{N-1}est} implies that
	\begin{align*}
		\left|\cU_{N, N+k}\right|&\lesssim N\wh M_{k+1}\wh M_N+NA^3\wh M_{k+1}\wh M_{N-1}+N^2\wh M_{k+1}\wh M_{N-1}\\
		&\lesssim N\wh M_{k+1}\wh M_N+NA^3\wh M_{k+1}\cdot A^{-3}\wh M_N+N^2\wh M_{k+1}\cdot N^{-3/2}\wh M_N\\
		&\lesssim N\wh M_{k+1}\wh M_N\lesssim N M_{k+1}^*\wh M_N,
	\end{align*}
	where we have used the fact $M_n^*\asymp \wh M_n$ for $n\in\Z\cap[1, N]$.\footnote{We recall that by \eqref{Eq.M_n^*est} we have $M_n^*\lesssim \wh M_n$. The fact $M_n^*\asymp \wh M_n$ follows from Proposition \ref{Prop.a_n_sharp_est} and $|a_n|\leq M_n^*$ by the definition of $\{M_n^*\}$ in \eqref{Eq.a_n_bound_byM_n^*}.} As a consequence, we have
	\begin{align*}
		\left|\sum_{k=3}^{N-2}\cU_{N, N+k}\tau^{k-1}\right|\lesssim N\wh M_N\sum_{k=3}^{N-2}M_{k+1}^*|\tau|^{k-1}\lesssim Na_N\sum_{k=3}^{N-2}M_{k+1}^*|\tau|^{k-1}.
	\end{align*}
	\if0the convexity property of $\{M_n^*\}$, we know that
	\[\left(M_k^*|\tau|^{k-2}\right)^2\leq \left(M_{k-1}^*|\tau|^{k-3}\right)\cdot\left(M_{k+1}^*|\tau|^{k-1}\right)\]
	hence the sequence $\{M_n^*|\tau|^{n-2}\}$ also satisfies the convexity property. \fi \if0It follows from \eqref{Eq.M_n^*claim} that
	\begin{equation}
		M_i^*|\tau|^{i-2}M_{k-i}^*|\tau|^{k-i-2}\leq M_j^*|\tau|^{j-2}M_{k-j}^*|\tau|^{k-j-2}
	\end{equation}
	for $i\in\Z\cap[j,k-j]$, $j\in\Z\cap[1,k/2]$, $k\in\Z\cap[2,N]$. \fi
	By the definition of $\{M_n^*\}$, we know that $\{M_n^*\}$ satisfies the convexity property $(M_k^*)^2\leq M_{k-1}^*M_{k+1}^*$, then so does the sequence $\{M_n^*|\tau|^{n-2}\}$. In particular, we have
	\begin{equation*}
		\left(M_k^*|\tau|^{k-2}\right)^2\leq \left(M_{k-1}^*|\tau|^{k-3}\right)\cdot \left(M_{k+1}^*|\tau|^{k-1}\right)\Longrightarrow 2M_k^*|\tau|^{k-2}\leq M_{k-1}^*|\tau|^{k-3}+M_{k+1}^*|\tau|^{k-1}.
	\end{equation*}
	Adapting the proof of \eqref{Eq.M_n^*claim} yields that
	\begin{equation}\label{Eq.M_n^*tau_convexity}
		M_i^*|\tau|^{i-2}+M_{k-i}^*|\tau|^{k-i-2}\leq M_j^*|\tau|^{j-2}+M_{k-j}^*|\tau|^{k-j-2}
	\end{equation}
	for all $i\in\Z\cap[j,k-j]$, $j\in\Z\cap[1,k/2]$, $k\in\Z\cap[2,N]$. Thus,
	\begin{align*}
		\sum_{k=3}^{N-2}M_{k+1}^*|\tau|^{k-1}&\leq M_4^*|\tau|^2+M_{N-1}^*|\tau|^{N-3}+M_5^*|\tau|^3+M_{N-2}^*|\tau|^{N-4}\\
		&\qquad\qquad+(N/2-4)\left(M_6^*|\tau|^4+M_{N-3}^*|\tau|^{N-5}\right)\\
		&\lesssim |\tau|^2+|\tau|^3+N|\tau|^4+M_{N-1}^*|\tau|^{N-3}+M_{N-2}^*|\tau|^{N-4}+NM_{N-3}^*|\tau|^{N-5}\\
		&\lesssim N^{-1}+|\tau|^{N-5}\left(M_{N-1}^*|\tau|^2+M_{N-2}^*|\tau|+NM_{N-3}^*\right)\\
		&\lesssim N^{-1}+|\tau|^{N-5}\left(\wh M_{N-1}|\tau|^2+\wh M_{N-2}|\tau|+N\wh M_{N-3}\right),
	\end{align*}
	for $\tau\in[-4/\sqrt R, 0)$. Similarly to \eqref{Eq.M_{N-1}est} and \eqref{Eq.M_{N-2}est}, we have
	\begin{equation}\label{Eq.M_N-2_est}
		\wh M_{N-2}\asymp A^{-3}\wh M_{N-1}\asymp N^{-3/2}\wh M_{N-1},\quad \wh M_{N-3}\asymp N^{-3}\wh M_{N-1},
	\end{equation}
	hence,
	\begin{align*}
		\sum_{k=3}^{N-2}M_{k+1}^*|\tau|^{k-1}\lesssim N^{-1}+|\tau|^{N-5}\wh M_{N-1}\left(|\tau|^2+N^{-3/2}|\tau|+N^{-2}\right).
	\end{align*}
	To obtain \eqref{Eq.f_N_middle}, it remains to prove that
	\begin{equation}\label{Eq.f_N_middle_claim}
		|\tau|^{N-5}\wh M_{N-1}\left(|\tau|^2+N^{-3/2}|\tau|+N^{-2}\right)\lesssim N^{-1}+N^{-1/2}a_N|\tau|^{N-2},\quad\forall\ \tau\in\R.
	\end{equation}
	Indeed, if $|\tau|\geq N^{-1}$, then \eqref{Eq.M_{N-1}est}, Proposition \ref{Prop.a_n_sharp_est} and \eqref{Eq.hat_M_n_est} imply that
	\begin{align*}
		|\tau|^{N-5}\wh M_{N-1}&\left(|\tau|^2+N^{-3/2}|\tau|+N^{-2}\right)\lesssim |\tau|^{N-5}N^{-3/2}\wh M_N\left(|\tau|^2+N^{-3/2}|\tau|+N^{-2}\right)\\
		&\lesssim a_N|\tau|^{N-2}N^{-3/2}\left(|\tau|^{-1}+N^{-3/2}|\tau|^{-2}+N^{-2}|\tau|^{-3}\right)\\
		&\lesssim a_N|\tau|^{N-2}N^{-3/2}\left(N+N^{1/2}\right)\lesssim N^{-1/2}a_N|\tau|^{N-2}.
	\end{align*}
	If $|\tau|\leq N^{-1}$, then Lemma \ref{Lem.M_N-1upper_bound} implies that
	\begin{align*}
		|\tau|^{N-5}\wh M_{N-1}&\left(|\tau|^2+N^{-3/2}|\tau|+N^{-2}\right)=\wh M_{N-1}\left(|\tau|^{N-3}+N^{-3/2}|\tau|^{N-4}+N^{-2}|\tau|^{N-5}\right)\\
		&\leq \wh M_{N-1}\left(2N^{3-N}+N^{3-N-1/2}\right)\leq 3\left(C_0\sqrt N\right)^NN^{3-N}=\frac{3C_0^8}{N}\left(\frac{C_0^2}{N}\right)^{N/2-4}\\
		&\lesssim N^{-1}.
	\end{align*}
	This proves \eqref{Eq.f_N_middle_claim}, thereby completing the proof of \eqref{Eq.f_N_middle}.
	
	\underline{\bf Step 5.} In this step, we conclude the proof of \eqref{Eq.u_N_compare}. By \eqref{Eq.U_N,N+1_est}, \eqref{Eq.U_N,N+2_est}, \eqref{Eq.f_N_last} and \eqref{Eq.f_N_middle}, for $\tau\in[-4/\sqrt R, 0)$ and $R\in(N, N+1)$, we have
	\begin{align*}
		f_N(\tau)&=\cU_{N, N+1}+\left(-\cU_{N, N+2}\right)(-\tau)+\sum_{k=3}^{N-2}\cU_{N, N+k}\tau^{k-1}+\left(\cU_{N, 2N-1}\tau^{N-2}+\cU_{N, 2N}\tau^{N-1}\right)\\
		&<-\frac{21}4\sqrt Na_N+\frac54Na_N\cdot\frac{4}{\sqrt N}+\frac{C}{\sqrt N}\left(\sqrt Na_N+Na_N^2|\tau|^{N-2}\right)-2Na_N^2|\tau|^{N-2}\\
		&=\left(-\frac14+\frac{C}{\sqrt N}\right)\sqrt Na_N+\left(\frac C{\sqrt N}-2\right)Na_N^2|\tau|^{N-2}<0.
	\end{align*}
	This proves \eqref{Eq.u_N_compare}, recalling \eqref{Eq.L[u_N]}.
\end{proof}

\begin{lemma}\label{Lem.u_N_zero}
	There exists $N_0\in\Z_+$ such that if $N\in(N_0,+\infty)\cap(2\Z+1)$ and $R\in(N, N+1)$, then we have
	\begin{equation}\label{Eq.u_N_zero}
		u_{(N)}(\tau_N)=0\ \text{for some}\ \tau_N\in\left(-\frac{9}{5\sqrt R},0\right),\quad\text{and}\quad u_{(N)}(\tau)>0\ \  \forall\ \tau\in(\tau_N, 0).
	\end{equation}
\end{lemma}
\begin{proof}
	We claim that
	\begin{equation}\label{Eq.u_N_middle_claim}
		\left|a_1\tau+a_2\tau^2+\cdots+a_{N-1}\tau^{N-1}\right|\lesssim N^{-1/2}\left(1+a_N|\tau|^N\right),\quad\forall\ \tau\in\left(-4/\sqrt R, 0\right).
	\end{equation}
	Now we prove \eqref{Eq.u_N_zero} by assuming \eqref{Eq.u_N_middle_claim}. For $\tau\in\left(-4/\sqrt R, 0\right)$ and $N\in2\Z+1$, there holds
	\begin{align*}
		u_{(N)}(\tau)\leq 1-a_N|\tau|^N+\frac C{\sqrt N}\left(1+a_N|\tau|^N\right)\leq 2-a_N|\tau|^N/2.
	\end{align*}
	Hence, it follows from Lemma \ref{Lem.M_N_lower_bound} and $R\in(N, N+1)$ that
	\begin{align}
		u_{(N)}\left(-\frac{9}{5\sqrt R}\right)&\leq 2-\frac{a_N}{2}\left(\frac{9}{5\sqrt{N+1}}\right)^N\nonumber\\
		&\leq 2-C\left(\frac{9}{5\sqrt{N+1}}\frac{5\sqrt N}{8}\right)^N=2-C\left(\frac{10}9\right)^N<0.\label{Eq.u_N<0}
	\end{align}
	Then, \eqref{Eq.u_N_zero} follows from $u_{(N)}(0)=a_0=1>0$ and the intermediate value theorem.
	
	It remains to prove the claim \eqref{Eq.u_N_middle_claim}. By \eqref{Eq.a_n_bound_byM_n^*}, we have
	\begin{align}\label{Eq.u_N_middle_triangle}
		\left|\sum_{j=1}^{N-1}a_j\tau^j\right|\lesssim \sum_{j=1}^{N-1}M_j^*|\tau|^j.
	\end{align}
	By mimicking the proof of \eqref{Eq.M_n^*tau_convexity}, we obtain the convexity of $\{M_n^*|\tau|^n\}$:
	\begin{equation*}
		M_i^*|\tau|^{i}+M_{k-i}^*|\tau|^{k-i}\leq M_j^*|\tau|^{j}+M_{k-j}^*|\tau|^{k-j}
	\end{equation*}
	for all $i\in\Z\cap[j,k-j]$, $j\in\Z\cap[1,k/2]$, $k\in\Z\cap[2,N]$. Hence, for $\tau\in(-4/\sqrt R, 0)$, we have
	\begin{align}
		\sum_{j=1}^{N-1}M_j^*|\tau|^j&\leq M_1^*|\tau|+M_{N-1}^*|\tau|^{N-1}+M_2^*|\tau|^2+M_{N-2}^*|\tau|^{N-2}\nonumber\\
		&\qquad\qquad+\frac{N-5}{2}\left(M_3^*|\tau|^3+M_{N-3}^*|\tau|^{N-3}\right)\nonumber\\
		&\lesssim |\tau|+|\tau|^2+N|\tau|^3+M_{N-1}^*|\tau|^{N-1}+M_{N-2}^*|\tau|^{N-2}+NM_{N-3}^*|\tau|^{N-3}\nonumber\\
		&\lesssim N^{-1/2}+|\tau|^{N-3}\left(M_{N-1}^*|\tau|^{2}+M_{N-2}^*|\tau|+NM_{N-3}^*\right)\nonumber\\
		&\lesssim N^{-1/2}+|\tau|^{N-3}\wh M_{N-1}\left(|\tau|^2+N^{-3/2}|\tau|+N^{-2}\right),\label{Eq.u_N_middle_est}
	\end{align}
	where in the last inequality we have used \eqref{Eq.hat_M_n_est} and \eqref{Eq.M_N-2_est}. Similarly to \eqref{Eq.f_N_middle_claim}, one can prove
	\begin{equation}\label{Eq.u_N_middle_est2}
		|\tau|^{N-3}\wh M_{N-1}\left(|\tau|^2+N^{-3/2}|\tau|+N^{-2}\right)\lesssim N^{-1/2}\left(1+a_N|\tau|^{N}\right),\quad\forall\ \tau\in\R.
	\end{equation}
	Thus, \eqref{Eq.u_N_middle_claim} follows from \eqref{Eq.u_N_middle_triangle}, \eqref{Eq.u_N_middle_est} and \eqref{Eq.u_N_middle_est2}.
\end{proof}

\begin{lemma}[Relative positions of barrier functions]\label{Lem.position}
	There exists $N_0\in\Z_+$ such that if $N\in(N_0,+\infty)\cap(2\Z+1)$ and $R\in(N, N+1)$, then
	\begin{equation}\label{Eq.u_N<u_g}
		u_{(N)}(\tau)<u_{\mathrm g}(\tau),\quad\forall\ \tau\in\left(-\frac{9}{5\sqrt R},0\right).
	\end{equation}
\end{lemma}
\begin{proof}
	We first recall the definition of $u_\text g$ from \eqref{Eq.u_g}. It is a direct computation to find that
	\begin{equation*}
		\cL\left[u_\text g\right](\tau)=-\frac43\tau(\tau+1-\al)(2\tau+1-\al)(5\tau+4-\al),
	\end{equation*}
	hence $\cL\left[u_\text g\right](\tau)>0$ for $\tau\in(-(1-\al)/2,0)$, where we have used $\al\in(0,1)$; see \eqref{maps-parameter}. It follows from \eqref{Eq.R} that
	\[\frac{1-\al}{2}=\frac{2\sqrt R}{(\sqrt R+1)^2}>\frac9{5\sqrt R},\quad\forall\ R>342,\]
	thus,
	\begin{equation}\label{Eq.u_g_compare}
		\cL\left[u_\text g\right](\tau)>0,\quad\forall\ \tau\in\left(-\frac{9}{5\sqrt R},0\right).
	\end{equation}
	
	Now we prove \eqref{Eq.u_N<u_g}. %By \eqref{Eq.slopes}, there exists $\tau_*\in\left(-\frac{9}{5\sqrt R},0\right)$ such that $u_{(N)}(\tau)<u_\text g(\tau)$ for all $\tau\in[\tau_*, 0)$. We assume for contradiction that $u_{(N)}(\tau^*)=u_\text g(\tau^*)$ for some $\tau^*\in \left(-\frac{9}{5\sqrt R},\tau_*\right)$ and $u_{(N)}(\tau)<u_\text g(\tau)$ for all $\tau\in(\tau^*, 0)$. Then $\frac{\mathrm du_{(N)}}{\mathrm d\tau}(\tau^*)\leq \frac{\mathrm du_\text g}{\mathrm d\tau}(\tau^*)$
	By the fact that $u_\text g$ is a quadratic polynomial, $-(1-\al)/2>-(4-\al)/5$, $u_\text g(-(1-\al)/2)=(\al^2+10\al+1)/2>0$, $-(1-\al)/2<-9/(5\sqrt R)$ and \eqref{Eq.u_N<0}, we have
	\begin{equation*}
		u_\text g\left(-\frac9{5\sqrt R}\right)>0>u_{(N)}\left(-\frac9{5\sqrt R}\right).
	\end{equation*}
	We assume for contradiction that $u_{(N)}(\tau_*)=u_\text g(\tau_*)$ for some $\tau_*\in \left(-\frac{9}{5\sqrt R},0\right)$ and $u_{(N)}(\tau)<u_\text g(\tau)$ for all $\tau\in\left(-\frac{9}{5\sqrt R},\tau_*\right)$. Then $\frac{\mathrm du_{(N)}}{\mathrm d\tau}(\tau_*)\geq \frac{\mathrm du_\text g}{\mathrm d\tau}(\tau_*)$. Since $$\Delta_\tau\left(\tau_*, u_{(N)}(\tau_*)\right)=\Delta_\tau\left(\tau_*, u_\text g(\tau_*)\right)>0,$$ we get by \eqref{Eq.u_g_compare} that 
	\begin{align*}
		\cL\left[u_{(N)}\right](\tau_*)&=\Delta_\tau\left(\tau_*, u_{(N)}(\tau_*)\right)\frac{\mathrm du_{(N)}}{\mathrm d\tau}(\tau_*)-\Delta_u\left(\tau_*, u_{(N)}(\tau_*)\right)\\
		&\geq \Delta_\tau\left(\tau_*, u_\text g(\tau_*)\right)\frac{\mathrm du_\text g}{\mathrm d\tau}(\tau_*)-\Delta_u\left(\tau_*, u_\text g(\tau_*)\right)=\cL\left[u_\text g\right](\tau_*)>0,
	\end{align*}
	which contradicts with Proposition \ref{Prop.u_N_compare}.
\end{proof}

\begin{lemma}\label{Lem.U_O_compare}
	Recall the definition of $a$ in \eqref{Eq.renormalization}. Let $U_\mathrm O(\tau):=1-\tau/a$. If $R>5$, then
	\begin{equation}\label{Eq.U_O_compare}
		\cL\left[U_\mathrm O\right](\tau)<0\quad\forall\ \tau\in(0,a).
	\end{equation}
\end{lemma}
\begin{proof}
	By \eqref{Eq.Lu}, we compute that
	\begin{equation}
		\cL[U_\text O](\tau)=\frac{\tau(\tau-a)(7a(a+3)\tau-4a^3+27a-9)}{3a^2(a+3)},
	\end{equation}
	recalling \eqref{Eq.renormalization}.
	Let $\varphi_O(a):=-4a^3+27a-9$. By the fact that $\varphi_\text O$ is a cubic polynomial and
	\begin{align*}
		\varphi_O&(-3)=18>0,\ \varphi_O(0)=-9<0,\ \varphi_O\left(\frac{3}{\sqrt{30}+3}\right)=\frac{3^3(14-\sqrt{30})}{(\sqrt{30}+3)^3}>0
%\frac{3\sqrt{6}(39\sqrt{30}-131)}{343(\sqrt{5}+\sqrt{6})}>0,
\\ &\varphi_O\left(\sqrt{3}\right)=-9+15\sqrt3>0,\ \varphi_O(3)=-36<0,
	\end{align*}
	we have 
	\begin{equation*}
		-4a^3+27a-9>0,\quad\forall\ a\in\left(3/(\sqrt{30}+3),\sqrt 3\right).
	\end{equation*}
	By \eqref{maps-parameter} and \eqref{Eq.r_to_R_increasing}, we know that the map $a\in(0, \sqrt 3)\mapsto R\in(1,+\infty)$ is strictly increasing. Since (note that $a=\frac{3}{\sqrt{30}+3}\Leftrightarrow\alpha=\frac{1}{7}$)
	\begin{align*}
		R\left(a=\frac{3}{\sqrt{30}+3}\right)=\frac{(\sqrt{7}+1)^2}{(\sqrt{7}-1)^2}=\frac{23+\sqrt{448}}{9}<\frac{23+\sqrt{484}}{9}=5,
%\frac{3350881+36696\sqrt{7035}}{1324801}<\frac{3350881+36696\cdot 89}{1324801}=\frac{6616825}{1324801}<\frac{6624005}{1324801}=5,
	\end{align*}
	we know that if $R>5$, then $-4a^3+27a-9>0$, and thus \eqref{Eq.U_O_compare} follows.
\end{proof}

\subsection{Global extension of the $\sigma-w$ solution}\label{Subsec.global_extension}
Let $N_0\in\Z_+$ be large enough such that all properties in previous sections hold true. In this subsection, we prove that for each $N\in(N_0,+\infty)\cap (2\Z+1)$ and $R\in(N, N+1)$, the local solution $w_L=w_L(\sigma)$ near the sonic point $P_2$ constructed in Section \ref{Sec.local_sol} leaves the region $\cR$ between $P_2$ and $P_5$ by crossing the red curve, and can  then be extended until reaching the origin $P_4$.

We begin with some fundamental definitions and properties of the systems \eqref{ODE_w_sigma} and \eqref{ODE_u_tau}. Let
\begin{align*}
    \cD:&=\left\{(\tau, u):u_\text b(\tau)<u<u_\text g(\tau), -\frac{1-\al}{2}=\tau(Q_5)<\tau<0\right\}\\
    &=\left\{(\tau, u): \Delta_\tau(\tau, u)>0, \Delta_u(\tau, u)>0, \tau(Q_5)<\tau<0\right\}.
\end{align*}
We denote the re-normalization map \eqref{Eq.renormalization} by $\Psi:(\sigma,w)\mapsto (\tau, u)$, where
\begin{equation*}
	\Psi(\sigma,w):=(-(1+a)(w-w_-), (1+a)^2\sigma^2).
\end{equation*}
Recall that
\begin{align*}
	\cR=\{(\sigma, w): \Delta_1(\sigma, w)<0, \Delta_2(\sigma, w)>0, w(P_2)=w_-<w<r/2=w(P_5)\}.
\end{align*}
It is a straightforward computation to verify that if $\Psi(\sigma, w)=(\tau, u)$, then
\begin{equation}\label{Eq.Delta_1,2_tau,u}
	\Delta_\tau(\tau, u)=-(1+a)^{3}\Delta_1(\sigma, w),\quad \Delta_u(\tau, u)=2(1+a)^4\sigma\Delta_2(\sigma,w).
\end{equation}
Using these relations, one can check that  both $\Psi: \{(\sigma,w): \sigma>0\}\to\{(\tau, u): u>0\}$ and $\Psi:\cR\to \cD$ are bijective.

\begin{lemma}\label{Lem.alternative}
	Assume that $R>5$. Given $(\sigma_0, w_0)\in\cR$, let $w=w(\sigma)$ be the unique solution of
	\begin{equation}\label{Eq.sigma-w-ODE}
		\frac{\mathrm dw}{\mathrm d\sigma}=\frac{\Delta_1(\sigma, w)}{\Delta_2(\sigma, w)}\quad\text{with}\quad w(\sigma_0)=w_0.
	\end{equation}
	Then, one of the following holds: either
	\begin{enumerate}[(i)]
		\item the solution exists for $\sigma\in(0, \sigma_0]$, and satisfies $\lim_{\sigma\downarrow0}w(\sigma)=0$, $w(\sigma)> a(1+a)\sigma^2$ for all $\sigma\in(0,\sigma_0)$, and there exists $\sigma_1\in(0, \sigma_0)$ such that $\frac{\mathrm dw}{\mathrm d\sigma}(\sigma)>0$ for $\sigma\in(0, \sigma_1)$, $\frac{\mathrm dw}{\mathrm d\sigma}(\sigma)<0$ for $\sigma\in(\sigma_1, \sigma_0)$; or
		\item there exists $\sigma_*\in[\sigma(P_5),\sigma_0)\subset(0,\sigma_0)$ such that the solution exists for $\sigma\in(\sigma_*, \sigma_0]$ and $(\sigma, w(\sigma))\in\cR$ for all $\sigma\in (\sigma_*, \sigma_0]$, and $w_*=w_2^-(\sigma_*)$, where $w_*:=\lim_{\sigma\downarrow\sigma_*}w(\sigma)$.
	\end{enumerate}
\end{lemma}
\begin{proof}
	Let $\Omega:=\{(\sigma, w): \Delta_2(\sigma, w)>0\}$. Then $(\sigma_0, w_0)\in\cR\subset \Omega$, $\Omega\subset \R^2$ is an open set and $\Delta_1/\Delta_2$ is a smooth function on $\Omega$. Assume that the maximal interval of existence of the solution to \eqref{Eq.sigma-w-ODE} (such that $(\sigma, w)\in\Omega$) is $(\sigma_*, \sigma_+)$. Then $0\leq\sigma_*<\sigma_0<\sigma_+$.\smallskip
	
	\underline{\bf Step 1.} In this step, we prove that if $R>5$, then
	\begin{equation}\label{Eq.w>w_O}
		w(\sigma)>a(1+a)\sigma^2,\quad\forall\ \sigma\in(\sigma_*,\sigma_0).
	\end{equation}
	We recall from Lemma \ref{Lem.U_O_compare} that $U_\text O(\tau)=1-\tau/a$ is a barrier function. Under the transformation \eqref{Eq.renormalization}, $U_\text O$ corresponds to the curve $w_\text O(\sigma)=a(1+a)\sigma^2$ in the $\sigma-w$ plane. We denote
	\begin{equation*}
		\cN[w](\sigma):=\Delta_2(\sigma,w(\sigma))\frac{\mathrm dw}{\mathrm d\sigma}-\Delta_1(\sigma,w(\sigma))
	\end{equation*}
	for any $C^1$ function $w=w(\sigma)$. Then it is a straightforward computation to find that
	\begin{align}
		\cN\left[w_\text O\right](\sigma)=\frac1{(1+a)^3}\frac{\cL\left[u_\text O\right](\tau)}{\mathrm du_\text O/\mathrm d\tau}>0,\quad\forall\ \sigma\in\left(0, 1-w_-=\sigma(P_2)\right).\label{Eq.w_O_compare}
	\end{align}
	By \eqref{Eq.w_j_monotonicity}, $w_2$ is strictly decreasing. Using the strict increase of $w_\text O$ in $(0,\sigma(P_2))$ and the fact $w_\text O(\sigma(P_2))=w_-=w(P_2)$, we obtain
	\begin{equation*}
		w_\text O(\sigma)<w_2(\sigma),\quad\forall\ \sigma\in(0,\sigma(P_2)).
	\end{equation*}
	Hence, $w_\text O(\sigma_0)<w_2(\sigma_0)<w_0=w(\sigma_0)$ due to $(\sigma_0,w_0)\in\cR$. We assume for contradiction that there exists $\sigma^*\in(\sigma_*, \sigma_0)$ such that $w(\sigma^*)=w_\text O(\sigma^*)$ and $w(\sigma)>w_\text O(\sigma)$ for $\sigma\in(\sigma^*,\sigma_0]$. Then $\frac{\mathrm dw_\text O}{\mathrm d\sigma}(\sigma^*)\leq \frac{\mathrm dw}{\mathrm d\sigma}(\sigma^*)$. It follows from $\Delta_2(\sigma^*,w(\sigma^*))= \Delta_2(\sigma^*,w_\text O(\sigma^*))>0$ that
	\begin{align*}
		\cN\left[w_\text O\right](\sigma^*)&=\Delta_2\left(\sigma^*,w_\text O(\sigma^*)\right) \frac{\mathrm dw_\text O}{\mathrm d\sigma}(\sigma^*)-\Delta_1 \left(\sigma^*,w_\text O(\sigma^*)\right)\\
		&\leq \Delta_2\left(\sigma^*,w(\sigma^*)\right) \frac{\mathrm dw}{\mathrm d\sigma}(\sigma^*)-\Delta_1 \left(\sigma^*,w(\sigma^*)\right)=0,
	\end{align*}
which contradicts \eqref{Eq.w_O_compare}.\smallskip
	
	\underline{\bf Step 2.} %In this step we prove the desired alternative property.
	Let $W(\sigma):=w(\sigma)-w_2(\sigma)$ for $\sigma\in(\sigma_*, \sigma_0]$. We compute that
	\begin{align*}
		\frac{\mathrm dW}{\mathrm d\sigma}&=\frac{\mathrm dw}{\mathrm d\sigma}-\frac{\mathrm dw_2}{\mathrm d\sigma}=\frac{\Delta_1(\sigma, w(\sigma))}{\Delta_2(\sigma, w(\sigma))}-\frac{\mathrm dw_2}{\mathrm d\sigma}\\
		&=\frac{(w(\sigma)-w_1(\sigma))(w(\sigma)-w_2(\sigma))(w(\sigma)-w_3(\sigma))}{\Delta_2(\sigma, w(\sigma))}-\frac{\mathrm dw_2}{\mathrm d\sigma}=W(\sigma)F_1(\sigma)-\frac{\mathrm dw_2}{\mathrm d\sigma},
	\end{align*}
	where $F_1(\sigma):=(w(\sigma)-w_1(\sigma))(w(\sigma)-w_3(\sigma))/\Delta_2(\sigma, w(\sigma))\in C((\sigma_*, \sigma_0])$. Then
	\begin{equation*}
		\frac{\mathrm d}{\mathrm d\sigma}\left(\exp\left(\int_\sigma^{\sigma_0}F_1\left(\wt\sigma\right)\,\mathrm d\wt\sigma\right)W\right)=-\exp\left(\int_\sigma^{\sigma_0}F_1\left(\wt\sigma\right)\,\mathrm d\wt\sigma\right)\frac{\mathrm dw_2}{\mathrm d\sigma}(\sigma)>0,\quad\forall\ \sigma\in(\sigma_*,\sigma_0],
	\end{equation*}
	where we have used \eqref{Eq.w_j_monotonicity}. Let
	\[\wt W(\sigma):=\exp\left(\int_\sigma^{\sigma_0}F_1\left(\wt\sigma\right)\,\mathrm d\wt\sigma\right)W(\sigma),\quad\forall\ \sigma\in(\sigma_*, \sigma_0].\]
	Then $\wt W$ is strictly increasing on $\sigma\in(\sigma_*, \sigma_0]$. It follows from $(\sigma_0, w_0)\in\cR$ that $W(\sigma_0)>0$, thus $\wt W(\sigma_0)>0$. Therefore, one of the following holds:
	\begin{enumerate}[(a)]
		\item $\wt W(\sigma)>0$ for all $\sigma\in(\sigma_*,\sigma_0]$;
		\item there exists $\sigma_1\in(\sigma_*,\sigma_0)$ such that $\wt W(\sigma_1)=0$, and
		\[\text{$\wt W(\sigma)>0$ for all $\sigma\in(\sigma_1,\sigma_0]$, \quad$\wt W(\sigma)<0$ for all $\sigma\in(\sigma_*,\sigma_1)$.}\]
	\end{enumerate}
	
	\underline{\bf Step 3.} In this step, we prove that (b) implies (i). Assume that (b) holds, then $W(\sigma)<0$ for all $\sigma\in(\sigma_*,\sigma_1)$,  {$W(\sigma)>0$ for all $\sigma\in(\sigma_1,\sigma_0)$}.  It follows from \eqref{Eq.w>w_O}, $w_1(\sigma)<0<w_\text O(\sigma)$ and $w_3(\sigma)>r>w(\sigma)$ for $\sigma\in(\sigma_*,\sigma_0)$ that $F_1(\sigma)<0$ for $\sigma\in(\sigma_*,\sigma_0)$. Thus,
	\begin{align*}
		\frac{\mathrm dw}{\mathrm d\sigma}(\sigma)=W(\sigma)F_1(\sigma)>0,\quad\forall\ \sigma\in(\sigma_*,\sigma_1),\quad \frac{\mathrm dw}{\mathrm d\sigma}(\sigma)=W(\sigma)F_1(\sigma)<0,\quad\forall\ \sigma\in(\sigma_1,\sigma_0).
	\end{align*}
We assume on contrary that $\sigma_*>0$.	Then $w$ is strictly increasing on $\sigma\in(\sigma_*,\sigma_1)$, and hence the limit $w_*:=\lim_{\sigma\downarrow\sigma_*}w(\sigma)$ exists and $w_*\in[0,1]$. Standard ODE theory (see \cite[Remark 3.2]{SWZ2024_1}) gives that $\Delta_2(\sigma_*, w_*)=0$, hence $w_*=w_2^+(\sigma_*)$ or $w_*=w_2^-(\sigma_*)$. On the other hand, one easily checks that $\sigma_1\in[\sigma(P_5),\sigma(P_2))$, which can be read from the phase portrait (see Figure \ref{Fig.Phase_Portrait_w_sigma}). Otherwise, we would obtain a contradiction with $\Delta_2(\sigma,w(\sigma))>0$ for all $\sigma\in(\sigma_1,\sigma_0]$. So, $w(\sigma)<w(\sigma_1)<w_2(\sigma_1)<w_2(\sigma(P_5))=w_2^-(\sigma(P_5))$ for all $\sigma\in(\sigma_*,\sigma_1)$. Therefore,
\begin{align*}
	w_*<\min\{w_2(\sigma(P_5)),w_2(\sigma_*)\}<w_2^-(\sigma_*)<w_2^+(\sigma_*).
\end{align*}
This is a contradiction. Thus, $\sigma_*=0$ and $w_*\in[0,w(P_5)=r/2]\subset[0,1]$. Next we prove $w_*=0$. We assume on contrary that $w_*\in(0,r/2]$. Recalling the definition of $\Delta_2$, we let
\[F_2(\sigma):=\sigma\frac{\Delta_1(\sigma,w(\sigma))}{\Delta_2(\sigma,w(\sigma))}\in\R,\quad\forall\ \sigma\in(0,\sigma_0].\]
Then $F_2$ is continuous and
\begin{equation}\label{Eq.F_2(0)>0}
	\lim_{\sigma\downarrow0}F_2(\sigma)=\frac{3w_*(w_*-1)(w_*-r)}{5w_*^2-(6+2r)w_*+3r}>0,
\end{equation}
where we have used $5(r/2)^2-(6+2r)r/2+3r=r^2/4>0$ and $(r+3)/5<r/2$, which follows from $r<3-\sqrt3<2$. Using $\mathrm dw/\mathrm d\sigma=F_2(\sigma)/\sigma$, we obtain
\[w_0-w_*=\int_0^{\sigma_0}\frac{F_2(\sigma)}{\sigma}\,\mathrm d\sigma\in\R,\]
which contradicts \eqref{Eq.F_2(0)>0}. Therefore, $w_*=0$. This proves (i).\smallskip

\underline{\bf Step 4.} In this step, we prove that (a) implies (ii). Assume that (a) holds. Then $W(\sigma)>0$ for all $\sigma\in(\sigma_*,\sigma_0]$. Or equivalently, $w(\sigma)>w_2(\sigma)$ for all $\sigma\in(\sigma_*,\sigma_0]$. We first show that $\sigma_*\geq \sigma(P_5)=\sqrt3r/6$. Assume on contrary that $0\leq \sigma_*<\sigma(P_5)$. Then $w(\sigma)>w_2(\sigma)>w_2^-(\sigma)$ for $\sigma\in\left(\max\{\sigma_2^{(0)}, \sigma_*\}, \sigma(P_5)\right)$. Combining with the fact that $\Delta_2(\sigma, w(\sigma))>0$ implies that $w(\sigma)>w_2^+(\sigma)$ for $\sigma\in\left(\max\{\sigma_2^{(0)}, \sigma_*\}, \sigma(P_5)\right)$. We take $\sigma_3:=\left(\max\{\sigma_2^{(0)}, \sigma_*\}+ \sigma(P_5)\right)/2<\sigma(P_5)<\sigma_0$, then $w(\sigma_3)>w_2^+(\sigma_3)>(r+3)/5>w_2^-(\sigma_0)>w(\sigma_0)$. Thus, there exists $\sigma_4\in(\sigma_3,\sigma_0)$ such that $w_2^-(\sigma_4)<w(\sigma_4)<w_2^+(\sigma_4)$, which implies that $\Delta_2(\sigma_4, w(\sigma_4))<0$, contradicting the definition of $\sigma_*$. Therefore, $\sigma_*\geq \sigma(P_5)$. Next we claim that
	\begin{equation}\label{Eq.w_eye_region}
		w_2(\sigma)<w(\sigma)<w_2^-(\sigma),\quad\forall\ \sigma\in(\sigma_*, \sigma_0).
	\end{equation}
	Indeed, the left inequality follows from the assumption (a), and the right inequality can be established using a similar argument to the proof of  $\sigma_*\geq\sigma(P_5)$. As a consequence, we have $(\sigma, w(\sigma))\in\cR$, $\Delta_1(\sigma, w(\sigma))<0$ and $\Delta_2(\sigma, w(\sigma))>0$ for all $\sigma\in(\sigma_*, \sigma_0)$. So, \eqref{Eq.sigma-w-ODE} implies that $\sigma\mapsto w$ is strictly decreasing on $\sigma\in(\sigma_*, \sigma_0)$. Hence, the limit $w_*:=\lim_{\sigma\downarrow\sigma_*}w(\sigma)$ exists. Standard ODE theory (see \cite[Remark 3.2]{SWZ2024_1}) shows that $\Delta_2(\sigma_*, w_*)=0$, hence $\sigma_*=0$ or $w_*=w_2^+(\sigma_*)$ or $w_*=w_2^-(\sigma_*)$. It follows from $\sigma_*\geq\sigma(P_5)>0$ and \eqref{Eq.w_eye_region} that $w_*=w_2^-(\sigma_*)$. This proves (ii).
\end{proof}

To achieve our goal, we only need to rule out case (ii). This requires more information on the local solution $v_L$, especially the barrier function property given by Proposition \ref{Prop.u_N_compare}. Let $N_0\in\Z_+$ be large enough such that all properties in previous sections hold. For all $N\in(N_0, +\infty)\cap(2\Z+1)$ and $R\in(N, N+1)$, we define
\begin{equation}
	\cD_N:=\left\{(\tau, u): \tau\in(\tau_N, 0), u_\text b(\tau)<u<u_{(N)}(\tau)\right\},\quad\cR_N:=\Psi^{-1}(\cD_N).
\end{equation}
Thanks to Lemma \ref{Lem.position} and $\tau_N>-\frac9{5\sqrt R}>-\frac{1-\al}2=\tau(Q_5)$, we have $\cD_N\subset\cD$, hence $\cR_N\subset\cR$.

To make a preparation, we compute by \eqref{Eq.Delta_u}, \eqref{Eq.Delta_tau} and \eqref{Eq.Lu} that
\begin{align}\label{Eq.L[u_N]_modified}
    \Delta_\tau(\tau, u)\frac{\mathrm du_{(N)}}{\mathrm d\tau}(\tau)-\Delta_u(\tau, u)=\cL\left[u_{(N)}\right](\tau)+\left(u-u_{(N)}(\tau)\right)G(\tau, u),
\end{align}
where $G(\tau, u):=3(\al-\tau)\frac{\mathrm du_{(N)}}{\mathrm d\tau}(\tau)+2\left(u+u_{(N)}(\tau)-u_\text g(\tau)\right)$.

\begin{proposition}\label{Prop.global}
	Assume that $N_0\in\Z_+$ is large enough, $N\in(N_0, +\infty)\cap(2\Z+1)$ and $R\in(N, N+1)$. Given $(\sigma_0, w_0)\in\cR_N$, let $w=w(\sigma)$ be the unique solution to \eqref{Eq.sigma-w-ODE}. Then the solution $w=w(\sigma)$ exists for $\sigma\in(0,\sigma_0]$ and
	\[\lim_{\sigma\downarrow0}w(\sigma)=0,\quad w(\sigma)>a(1+a)\sigma^2\ \quad\forall\ \sigma\in(0,\sigma_0).\]
    Moreover, there exists $\sigma_1\in(0, \sigma_0)$ such that $\frac{\mathrm dw}{\mathrm d\sigma}>0$ for $\sigma\in(0, \sigma_1)$, $\frac{\mathrm dw}{\mathrm d\sigma}<0$ for $\sigma\in(\sigma_1, \sigma_0)$.
\end{proposition}
\begin{proof}
	As $\cR_N\subset \cR$, by Lemma \ref{Lem.alternative}, it suffices to exclude (ii). Assume that (ii) holds. Then there exists $\sigma_*\in[\sigma(P_5),\sigma_0)\subset(0,\sigma_0)$ such that the solution $w=w(\sigma)$ to \eqref{Eq.sigma-w-ODE} exists for $\sigma\in(\sigma_*,\sigma_0]$ and satisfies
	\[(\sigma, w(\sigma))\in\cR\text{ for all}\ \sigma\in(\sigma_*,\sigma_0]\quad\text{and}\quad w_*:=\lim_{\sigma\downarrow\sigma_*}w(\sigma)=w_2^-(\sigma_*).\]
	Let $$(\tau(\sigma),u(\sigma)):=\Psi(\sigma,w(\sigma))=\big(-(1+a)(w(\sigma)-w_-), (1+a)^2\sigma^2\big)\in\cD,\quad\forall\ \sigma\in(\sigma_*,\sigma_0],$$
	and we let $Y(\sigma):=u(\sigma)-u_{(N)}(\tau(\sigma))$ for $\sigma\in(\sigma_*,\sigma_0]$. By \eqref{Eq.Delta_1,2_tau,u} and \eqref{Eq.L[u_N]_modified}, we compute that
	\begin{align}
		\frac{\mathrm dY}{\mathrm d\sigma}(\sigma)&=2(1+a)^2\sigma+(1+a)\frac{\mathrm dw}{\mathrm d\sigma}(\sigma)\frac{\mathrm du_{(N)}}{\mathrm d\tau}(\tau(\sigma))\nonumber\\
		&=(1+a)\left[2(1+a)\sigma+\frac{\Delta_1(\sigma,w(\sigma))}{\Delta_2(\sigma,w(\sigma))}\frac{\mathrm du_{(N)}}{\mathrm d\tau}(\tau(\sigma))\right]\nonumber\\
		&=2(1+a)^2\sigma\left[1-\frac{\Delta_\tau(\tau(\sigma), u(\sigma))}{\Delta_u(\tau(\sigma), u(\sigma))}\frac{\mathrm du_{(N)}}{\mathrm d\tau}(\tau(\sigma))\right]\nonumber\\
		&=-\frac{2(1+a)^2\sigma}{\Delta_u(\tau(\sigma), u(\sigma))}\left[\Delta_\tau(\tau(\sigma), u(\sigma))\frac{\mathrm du_{(N)}}{\mathrm d\tau}(\tau(\sigma))-\Delta_u(\tau(\sigma), u(\sigma))\right]\nonumber\\
		&=\eta_1(\sigma)\left(\eta_2(\sigma)+Y(\sigma)\eta_3(\sigma)\right),\label{Eq.Y_equation}
	\end{align}
	where
	\[\eta_1(\sigma):=-\frac{2(1+a)^2\sigma}{\Delta_u(\tau(\sigma), u(\sigma))}<0,\ \ \eta_2(\sigma):=\cL\left[u_{(N)}\right](\tau(\sigma)),\ \ \eta_3(\sigma):=G(\tau(\sigma), u(\sigma))\]
	for $\sigma\in(\sigma_*,\sigma_0]$. It follows from $(\sigma_0, w_0)\in\cR_N$ that $(\tau(\sigma_0), u(\sigma_0))\in\cD_N$. Thus, $Y(\sigma_0)=u(\sigma_0)-u_{(N)}(\tau(\sigma_0))<0$.

    \underline{\bf Claim 1.} We claim that $Y(\sigma)<0$ for $\sigma\in(\sigma_*,\sigma_0]$. Indeed, by $(\tau(\sigma), u(\sigma))\in\cD$ for all $\sigma\in(\sigma_*,\sigma_0]$, we know that $$\tau(\sigma)>\tau(Q_5)=-\frac{1-\al}2=-\frac{2\sqrt R}{(\sqrt R+1)^2}>-\frac4{\sqrt R}.$$
Hence, Proposition \ref{Prop.u_N_compare} implies that $\eta_2(\sigma)<0$ for $\sigma\in(\sigma_*, \sigma_0]$. By \eqref{Eq.Y_equation}, we have
\begin{align*}
	\frac{\mathrm d\wt Y}{\mathrm d\sigma}=\eta_1(\sigma)\eta_2(\sigma)\exp\left(\int_\sigma^{\sigma_0}\eta_1(\wt\sigma)\eta_3(\wt\sigma)\,\mathrm d\wt\sigma\right)>0,\\
	 \text{where}\quad \wt Y(\sigma):=Y(\sigma) \exp\left(\int_\sigma^{\sigma_0}\eta_1(\wt\sigma)\eta_3(\wt\sigma)\,\mathrm d\wt\sigma\right)
\end{align*}
for $\sigma\in(\sigma_*, \sigma_0]$. Hence, $\wt Y$ is strictly {increasing} on $\sigma\in(\sigma_*, \sigma_0]$. Since $\wt Y(\sigma_0)=Y(\sigma_0)<0$, we have $\wt Y(\sigma)<0$ for all $\sigma\in(\sigma_*, \sigma_0]$, hence $Y(\sigma)<0$ for all $\sigma\in(\sigma_*, \sigma_0]$.\smallskip

\underline{\bf Claim 2.} We claim that $\tau(\sigma)>\tau_N$ for all $\sigma\in(\sigma_*, \sigma_0]$. By the definition of $\tau(\sigma)$ and the fact that $\mathrm dw/\mathrm d\sigma<0$ due to $(\sigma,w(\sigma))\in\cR$, we have
\[\frac{\mathrm d\tau}{\mathrm d\sigma}=-(1+a)\frac{\mathrm dw}{\mathrm d\sigma}>0,\quad\forall\ \sigma\in(\sigma_*, \sigma_0].\]
Hence, $\sigma\mapsto \tau(\sigma)$ is strictly increasing on $\sigma\in(\sigma_*, \sigma_0]$. Since $(\tau(\sigma_0), u(\sigma_0))\in\cD_N$, we have $\tau(\sigma_0)\in(\tau_N, 0)$. We assume, on contrary, that $\tau(\sigma^*)=\tau_N$ for some $\sigma^*\in(\sigma_*, \sigma_0)$. Then, by Lemma \ref{Lem.u_N_zero} and Claim 1, it follows   that $u(\sigma^*)=Y(\sigma^*)+u_{(N)}(\tau_N)<0$. However, it follows from $(\tau(\sigma), u(\sigma))\in\cD$ that $u(\sigma)>u_\text b(\tau(\sigma))>0$ for all $\sigma\in(\sigma_*, \sigma_0]$. This is a contradiction.

So far, we have proved that
\begin{equation}\label{Eq.tau,u_in_D_N}
	(\tau(\sigma), u(\sigma))\in\cD_N,\quad\forall\ \sigma\in(\sigma_*, \sigma_0].
\end{equation}
On the other hand, we denote $(\tau_*, u_*):=\Psi(\sigma_*, w_*)$. Recalling that $\sigma_*\geq \sigma(P_5)$ and $w_*=w_2^-(\sigma_*)$, we have
\begin{equation}\label{Eq.u_*=u_g(tau_*)}
	u_*=u_\text g(\tau_*).
\end{equation}
 By Claim 2, the strict increase of $\sigma\mapsto \tau(\sigma)$ and $\tau_*=\lim_{\sigma\downarrow\sigma_*}\tau(\sigma)$, we obtain $\tau_*\in[\tau_N, 0)$. Using Lemma \ref{Lem.position} and \eqref{Eq.tau,u_in_D_N} gives that
\begin{align*}
	u_*=\lim_{\sigma\downarrow\sigma_*}u(\sigma)\leq \lim_{\sigma\downarrow\sigma_*}u_{(N)}(\tau(\sigma))=u_{(N)}(\tau_*)<u_\text g(\tau_*),
\end{align*}
which contradicts \eqref{Eq.u_*=u_g(tau_*)}. Therefore, (ii) in Lemma \ref{Lem.alternative} cannot be true, and (i) holds.
\end{proof}

\subsection{Proof of Theorem \ref{Thm.ODE}}\label{Subsec.Proof_Thm_ODE}
Now, we are ready to prove Theorem \ref{Thm.ODE}.

\begin{proof}[Proof of Theorem \ref{Thm.ODE}]
	Let $N_0\in\Z_+$ be large enough such that all properties in previous sections hold. For each $N\in(N_0,+\infty)\cap(2\Z+1)$, let $R_N\in(N, N+1)$ be given by Proposition \ref{Prop.sol_sonic_point}, and we let $r_N$ be given by $R_N$ through inverting the map \eqref{Eq.r_to_R_increasing}. Then the sequence $\{r_n\}_{n\in (N_0,+\infty)\cap(2\Z+1) }$ verifies \eqref{Eq.r_n}.
	
	Fix $r=r_N$ for some $N\in(N_0,+\infty)\cap(2\Z+1)$, then $R=R_N$. Proposition \ref{Prop.sol_sonic_point} implies that there exists $\tau_*\in(\tau_N, 0)$ such that the $\tau-u$ ODE \eqref{ODE_u_tau} has a smooth solution $u=\wh u(\tau)$ defined on $\tau\in[\tau_*, \al(R_N))$ such that $\wh u(0)=a_0=1$ and satisfies the following properties:
	\begin{itemize}
		\item $u_\text b(\tau)<\wh u(\tau)<u_\text g(\tau)$ for $\tau\in(\tau_*,0)$ and $u_\text g(\tau)<\wh u(\tau)<u_\text b(\tau)$ for $\tau\in(0,\al(R_N))$;
		\item $\wh u$ is strictly increasing on $\tau\in[\tau_*, \al(R_N))$;
		\item $\lim_{\tau\uparrow \al(R_N)}\wh u(\tau)=+\infty$;
		\item $\wh u$ can be represented as a power series $\wh u(\tau)=\sum_{n=0}^\infty a_n\tau^n$ for $\tau\in[\tau_*, -\tau_*]$.
	\end{itemize}
	In particular, we have
	\begin{equation}
		\wh u(\tau)>0\quad\text{and}\quad \frac{\mathrm d\wh u}{\mathrm d\tau}(\tau)>0\quad\text{for all}\quad \tau\in[\tau_*,\al(R_N)).
	\end{equation}
	
	Let
	\[\left(\sigma(\tau), \wh w(\tau)\right):=\Psi^{-1}\left(\tau, \wh u(\tau)\right),\quad\forall\ \tau\in[\tau_*, \al(R_N)).\]
	Then
	\[\frac{\mathrm d\sigma}{\mathrm d\tau}=\frac1{2(1+a)\sqrt{\wh u(\tau)}}\frac{\mathrm d\wh u}{\mathrm d\tau}>0,\quad\forall\ \tau\in[\tau_*, \al(R_N)).\]
	Hence, $\tau\mapsto \sigma(\tau)$ is strictly increasing on $\tau\in[\tau_*, \al(R_N))$. We denote the inversion of $\tau\mapsto \sigma(\tau)$ by $\tau=\wt\tau(\sigma)$ for $\sigma\in[\sigma_*, +\infty)$, where $\sigma_*:=\sigma(\tau_*)$. Let
	\[\wt w(\sigma):=\wh w\left(\wt \tau(\sigma)\right),\quad\forall\ \sigma\in[\sigma_*, +\infty).\]
	Using \eqref{Eq.Delta_1,2_tau,u} and the definition of $\wt w$, we can directly verify that
	\[\frac{\mathrm d\wt w}{\mathrm d\sigma}(\sigma)=\frac{\Delta_1\left(\sigma,\wt w(\sigma)\right)}{\Delta_2\left(\sigma,\wt w(\sigma)\right)},\quad\forall\ \sigma\in[\sigma_*, +\infty).\]
	Hence, $\wt w$ is a smooth solution to \eqref{ODE_w_sigma} on $\sigma\in[\sigma_*, +\infty)$ and $\sigma_*\in(0, \sigma(P_2))$.

	By Lemma \ref{Lem.a_N+1<0}, we have $a_{N+1}<0$, hence $u_\text b(\tau^*)<\wh u(\tau^*)< u_{(N)}(\tau^*)$ for some $\tau^*\in(\tau_*, 0)$, recalling that $N$ is an odd integer. So,
	\[\left(\tau^*, \wh u\left(\tau^*\right)\right)\in\cD_N\Longrightarrow (\sigma_0, w_0):=\Psi^{-1}\left(\tau^*, \wh u\left(\tau^*\right)\right)\in\cR_N.\]
	Let $\bar w$ be the unique solution to
	\[\frac{\mathrm d\bar w}{\mathrm d\sigma}(\sigma)=\frac{\Delta_1\left(\sigma,\bar w(\sigma)\right)}{\Delta_2\left(\sigma,\bar w(\sigma)\right)}\quad\text{with}\quad \bar w(\sigma_0)=w_0.\]
	Then Proposition \ref{Prop.global} implies that $\bar w$ is defined on $\sigma\in(0, \sigma_0]$ and satisfies  $\lim_{\sigma\downarrow0}\bar w(\sigma)=0$ and $\bar w(\sigma)>a(R_N)(1+a(R_N))\sigma^2$ for $\sigma\in(0, \sigma_0]$, where $a(R_N)$ is given by $R_N$ through \eqref{Eq.renormalization}. The uniqueness implies that $\bar w(\sigma)=\wt w(\sigma)$ for all $\sigma\in(\sigma_*, \sigma_0]$.
	
	Let
	\begin{equation}\label{Eq.w_def}
		w(\sigma):=\begin{cases} \bar w(\sigma) & \text{if}\ \sigma\in(0,\sigma_0],\\
			\wt w(\sigma)& \text{if}\ \sigma>\sigma_0.\end{cases}
	\end{equation}
	Then $w$ is a smooth solution to \eqref{ODE_w_sigma} on $(0, +\infty)$ such that $\lim_{\sigma\downarrow0}w(\sigma)=0$, $w(\sigma(P_2))=w(P_2)$ and $\lim_{\sigma\uparrow+\infty}w(\sigma)=r_N-1$. Moreover, we also have $w(\sigma)>a(R_N)(1+a(R_N))\sigma^2$ for $\sigma\in(0,\sigma_0]$. Since $w$ is strictly decreasing on $\sigma\in(\sigma_0,\sigma(P_2))$ due to $(\sigma_0, w_0)\in\cR_N\subset\cR$, $\sigma\mapsto a(R_N)(1+a(R_N))\sigma^2$ is strictly increasing on $\sigma\in(\sigma_0,\sigma(P_2))$ and $w(\sigma(P_2))=w(P_2)=a(R_N)(1+a(R_N))\sigma(P_2)^2$, we obtain $w(\sigma)>a(R_N)(1+a(R_N))\sigma^2$ for $\sigma\in[\sigma_0,\sigma(P_2))$.
	
    Finally, we consider the ODE
	\begin{equation}\label{Eq.simga-X_ODE}
		\frac{\mathrm dX}{\mathrm d\sigma}=f(\sigma):=-\frac{\Delta(\sigma,w(\sigma))}{\Delta_2(\sigma,w(\sigma))}\quad\text{with}\quad X(\sigma=\sigma(P_2))=0,
	\end{equation}
	where $w(\sigma)$ is the smooth solution to \eqref{ODE_w_sigma} defined in \eqref{Eq.w_def}. It is elementary to check that $f\in C^\infty((0,+\infty))$, hence $X=X(\sigma)$ is defined on $(0, +\infty)$. Since $\Delta(\sigma,w(\sigma))(\sigma-\sigma(P_2))<0$ and $\Delta(\sigma,w(\sigma))(\sigma-\sigma(P_2))<0$ for all $\sigma>0$, we have $f(\sigma)<0$ for all $\sigma>0$, thus $\sigma\mapsto X(\sigma)$ is strictly decreasing on $\sigma\in(0, +\infty)$. By the fact that
	\[f(\sigma)\asymp -\sigma^{-1}\text{ as}\ \sigma\downarrow0;\quad f(\sigma)\asymp -\sigma^{-1}\text{ as}\ \sigma\uparrow+\infty,\]
	we have $\lim_{\sigma\downarrow 0}X(\sigma)=+\infty$ and $\lim_{\sigma\uparrow+\infty}X(\sigma)=-\infty$. We denote the inversion of $\sigma\mapsto X(\sigma)$ by $\sigma=\sigma(x)$ for $x\in\R$, and we also let $w(x):=w(\sigma(x))$ for $x\in\R$,\footnote{Here we abuse the notation.} Then $x\mapsto (\sigma,w)(x)$ is a global smooth solution to \eqref{autonomousODE}. Letting $x_A:=X(\sigma_1)$ completes the proof of Theorem \ref{Thm.ODE}, where $\sigma_1\in(0, \sigma(P_2))$ is given by Proposition \ref{Prop.global}.
\end{proof}

\section{Proof of repulsivity property}\label{Section repulsivity property}
In this section, we focus on proving the repulsivity of these profiles, namely, Proposition \ref{Prop.coercive}.

\subsection{Proof of (\ref{repuls 1}) and (\ref{repuls 2})}\label{Subsection repuls 12}
In this subsection, we prove the first two properties of Proposition \ref{Prop.coercive}. The proof does not rely on the assumption $d=\ell=3$.

\begin{lemma}
    Let $d>0$, $\ell>0$ and $r>1$. Assume that $(\sigma, w)=(\sigma(x), w(x))$ is a $C^\infty$ global solution to the autonomous ODE system \eqref{autonomousODE} such that $x\mapsto\sigma(x)$ is a strictly decreasing bijection from $\R$ onto $(0, +\infty)$, and
    \[\lim_{x\to+\infty}w(x)=0,\quad\lim_{x\to-\infty}w(x)\in\R.\]
    Then we have
    \begin{align}
        \sigma(x)&\gtrsim_{d,\ell,r,\sigma,w} \min\left\{\e^{-rx},\e^{-x}\right\}  \label{Eq.sigma_lower_bound},\\
		\left|w^{(j)}(x)\right|+\left|\sigma^{(j)}(x)\right|&\lesssim_{d,\ell,r,j,\sigma,w} \min\left\{\e^{-rx},\e^{-x}\right\},\quad\forall\
         j\in\Z_{\geq 0}.\label{Eq.w-sigma_upper_bound}
    \end{align}
    Here the derivatives are taken with respect to $x\in\R$.
\end{lemma}

\begin{proof}
    We denote
    \begin{equation}\label{Eq.F_2_def}
        F_2(x):=-\frac{\sigma'(x)}{\sigma(x)}=\frac{\Delta_2(\sigma(x), w(x))}{\sigma(x)\Delta(\sigma(x), w(x))},\quad\forall\ x\in\R.
    \end{equation}
    It follows from \eqref{Eq.Delta}, \eqref{Eq.Delta_2} and $\lim_{x\to+\infty}\sigma(x)=\lim_{x\to+\infty} w(x)=0$ that $\lim_{x\to +\infty}F_2(x)=r>0$. Let $r_\pm>0$ be such that $r_-<r<r_+<2r_-$. Then there exists $x_*>0$ such that
    \begin{equation}\label{Eq.F_2_est}
        r_-<F_2(x)<r_+\quad\text{and}\quad |F_2(x)-r|\leq C\left(|w(x)|+|\sigma(x)|^2\right),\quad\forall\ x\geq x_*,
    \end{equation}
    where $C>0$ is a constant independent of $x\geq x_*$. Using Gr\"onwall's inequality, we obtain
    \begin{equation}\label{bound 1 for sigma}
    c_{1}\e^{-r_{+}x}<\sigma(x)<c_{2}\e^{-r_{-}x},\quad\forall\ x\geq x_*
    \end{equation}
    for some constants $c_1, c_2\neq 0$. To estimate $w$, we denote
    $F_1(x):=(w(x)-1)(w(x)-r)/\Delta(\sigma(x), w(x))$. Due to \eqref{Eq.Delta} and $\lim_{x\to+\infty}\sigma(x)=\lim_{x\to+\infty} w(x)=0$, by taking $x_*>0$ to be larger if necessary, we know that $F_1$ is well-defined for $x\geq x_*$, and
    \begin{align}
        w'(x)&=-\frac{\Delta_1(\sigma(x), w(x))}{\Delta(\sigma(x), w(x))}=-w(x)F_1(x)+\frac{\left[dw(x)-\ell(r-1)\right]\sigma(x)^2}{\Delta(\sigma(x), w(x))},\nonumber\\
        &r_-<F_1(x)<r_+,\quad |F_1(x)-r|\leq C\left(|w(x)|+|\sigma(x)|^2\right),\label{Eq.F_1_est}\\
        &\qquad\qquad\qquad\left|w'(x)+w(x)F_1(x)\right|\leq C\sigma(x)^2,\label{Eq.w_differential_inequality}
    \end{align}
    for all $x\geq x_*$. Applying Gr\"onwall's inequality to \eqref{Eq.w_differential_inequality} gives that
    \begin{align*}
        w(x)&\leq w(x_*)\exp\left(-\int_{x_*}^x F_1(\wt x)\,\mathrm d\wt x\right)+C\int_{x_*}^x\left[\sigma(\wt x)^2\exp\left(-\int_{\wt x}^xF_1(\wh x)\,\mathrm d\wh x\right)\right]\,\mathrm d\wt x\leq c_3\e^{-r_-x}
    \end{align*}
    for all $x\geq x_*$ and some constant $c_3>0$, where we have used \eqref{Eq.F_1_est} and \eqref{bound 1 for sigma}. By \eqref{Eq.F_2_def}, \eqref{Eq.F_2_est}, \eqref{Eq.F_1_est}, \eqref{Eq.w_differential_inequality}, \eqref{bound 1 for sigma} and $w(x)\leq c_3\e^{-r_-x}$, we have
    \begin{align*}
        \left|\sigma'(x)+r\sigma(x)\right|&\leq\left|F_2(x)-r\right|\sigma(x)\leq C\left(w(x)^2+\sigma(x)^2\right)\leq C\e^{-2r_-x},\\
        \left|w'(x)+rw(x)\right|\leq &\ \left|F_1(x)-r\right|w(x)+C\sigma(x)^2\leq C\left(w(x)^2+\sigma(x)^2\right)\leq C\e^{-2r_-x},
    \end{align*}
    for all $x\geq x_*$. Thanks to $r<2r_-$, using Gr\"onwall's inequality again implies that
    \[\sigma(x)=\sigma_{+\infty}\e^{-rx}+\cO\left(\e^{-2r_-x}\right),\quad w(x)=w_{+\infty}\e^{-rx}+\cO\left(\e^{-2r_-x}\right),\quad\forall\ x\geq x_*,\]
    where $\sigma_{+\infty}\geq 0$ and $w_{+\infty}\in\R$ are constants. We note that $\sigma_{+\infty}>0$. Otherwise, we would have $\sigma(x)=\cO\left(\e^{-2r_-x}\right)$, which contradicts  \eqref{bound 1 for sigma} due to $r_+<2r_-$. This proves the $x>0$ part of \eqref{Eq.sigma_lower_bound} and \eqref{Eq.w-sigma_upper_bound} for $j=0$. A standard induction argument proves the $x>0$ part of \eqref{Eq.w-sigma_upper_bound} for $j\in\Z_+$. The proof of the $x<0$ part of \eqref{Eq.sigma_lower_bound} and \eqref{Eq.w-sigma_upper_bound} follows a similar approach, and the full details are left to the reader.
    \end{proof}

\subsection{Proof of (\ref{repuls 3}) and (\ref{repuls 4})}
In this subsection, we prove the last two properties of Proposition \ref{Prop.coercive}. We first note the following behavior of solutions to \eqref{autonomousODE} as $x\to\pm\infty$.

\begin{lemma}\label{Lem.far-field_behavior}
    Let $d>0$, $\ell>0$ and $r>1$ satisfy $\ell(r-1)<d$. Assume that $(\sigma, w)=(\sigma(x), w(x))$ is a $C^\infty$ global solution to the autonomous ODE system \eqref{autonomousODE} such that $x\mapsto\sigma(x)$ is a strictly decreasing bijection from $\R$ onto $(0, +\infty)$, and
    \[\lim_{x\to+\infty}w(x)=0,\quad\lim_{x\to-\infty}w(x)=\ell(r-1)/d.\]
    Then $\lim_{x\to\pm\infty}\left|\sigma(x)+\sigma'(x)\right|=0$ and $\lim_{x\to\pm\infty}w'(x)=0$. As a consequence, we have
    \begin{align*}
    &\lim_{x\to+\infty}\big(1-(w(x)+w'(x))-\left|\sigma(x)+\sigma'(x)\right|\big)=\lim_{x\to+\infty}\big(1-w(x)-\left|\sigma(x)+\sigma'(x)\right|\big)=1,\\
    &\lim_{x\to-\infty}\big(1-(w(x)+w'(x))-\left|\sigma(x)+\sigma'(x)\right|\big)\\
    &\qquad\qquad
    =\lim_{x\to-\infty}\big(1-w(x)-\left|\sigma(x)+\sigma'(x)\right|\big)=1-\frac{\ell(r-1)}{d}>0.
\end{align*}
\end{lemma}
\begin{proof}
    By \eqref{Eq.F_2_def}, we compute that $\sigma(x)+\sigma'(x)=\sigma(x)\left(1-F_2(x)\right)$ for $x\in\R$. It follows from  $\lim_{x\to+\infty}F_2(x)=r\in\R$ and $\lim_{x\to+\infty}\sigma(x)=0$ that $\lim_{x\to+\infty}\left|\sigma(x)+\sigma'(x)\right|=0$. Using \eqref{Eq.Delta} and \eqref{Eq.Delta_2} gives that
    \begin{align*}
        F_2(x)-1=\frac{(\ell+d-1)w(x)^2-(\ell+d+\ell r-r)w(x)+\ell r-\ell(w(x)-1)^2}{\ell\left[(w(x)-1)^2-\sigma(x)^2\right]},\quad \forall\ x\in\R.
    \end{align*}
    Hence, $\lim_{x\to-\infty}\sigma(x)=+\infty$ and $\lim_{x\to-\infty}w(x)\in\R$ imply that $\lim_{x\to-\infty}\sigma(x)(1-F_2(x))=0$, i.e., $\lim_{x\to-\infty}\left|\sigma(x)+\sigma'(x)\right|=0$. One can check $\lim_{x\to\pm\infty}w'(x)=0$ similarly.
\end{proof}

By Lemma \ref{Lem.far-field_behavior}, to show \eqref{repuls 3} and \eqref{repuls 4}, it suffices to prove the following proposition.

\begin{proposition}\label{Prop.replusivity_reduce}
    Let $d=\ell=3$ and $\{r_n\}$ be  given by Theorem \ref{Thm.ODE}. For each $r\in\{r_n\}$, let $(\sigma,w)$ be the solution to \eqref{autonomousODE} given by Theorem \ref{Thm.ODE}. Then we have
	\begin{align}\label{equal condition for repulsivity property}
		1-w(x)>\max\{0,w'(x)\}+|\sigma(x)+\sigma'(x)|,\quad\forall\ x\in\R.
		\end{align}
\end{proposition}

\begin{figure}[htbp]
	\centering
	\includegraphics[width=1\textwidth]{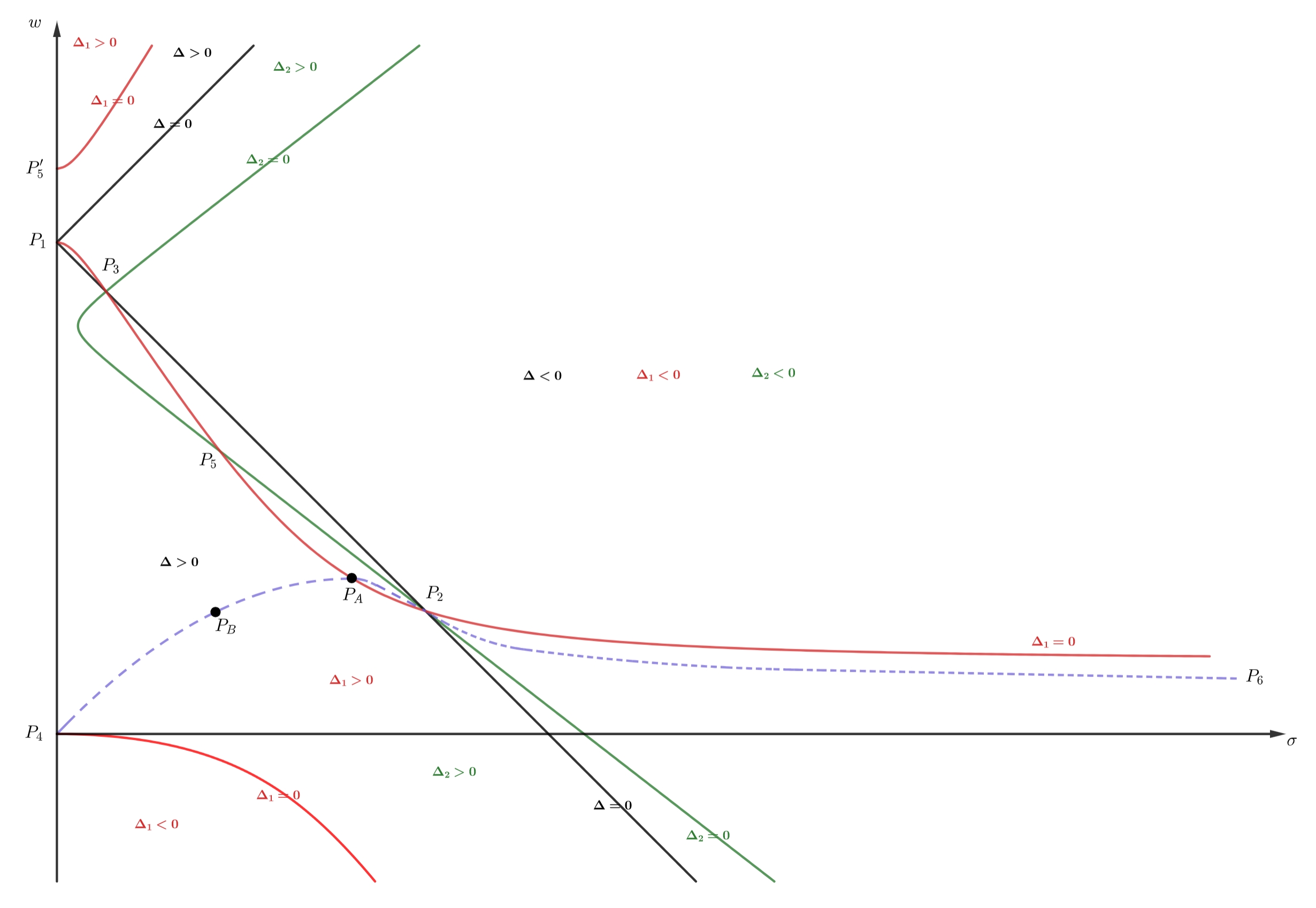}
	\caption{Phase portrait for the $\sigma-w$ system \eqref{ODE_w_sigma}: Positions of auxiliary points $P_A$ and $P_B$.}
	\label{Fig.w-sigma-repulsivity}
\end{figure}

The rest of this section is devoted to the proof of Proposition \ref{Prop.replusivity_reduce}. To this end, we introduce two auxiliary points $P_A$ and $P_B$, whose definitions are provided below; see also Figure \ref{Fig.w-sigma-repulsivity}. Recall from \eqref{Eq.P_1_2} and Subsection \ref{Subsec.Proof_Thm_ODE} (especially \eqref{Eq.simga-X_ODE}) that $(\sigma, w)(x=0)=P_2=(1-w_-, w_-)$. As $x$ increases, $\sigma(x)$ is strictly decreasing, and the trajectory $x\mapsto (\sigma(x), w(x))$ exits the region $\cR$ by crossing the red curve $(\sigma, w_2(\sigma))$. The intersection point is denoted by $P_A(\sigma_A=\sigma(x_A), w_A=w(x_A)$ for some $x_A>0$; moreover, we know that $x\mapsto w(x)$ is strictly increasing on $x\in[0, x_A]$. After crossing the red curve at $P_A$, as $x$ increases, due to the fact that $w(x)>a(1+a)\sigma(x)^2>0$ for all $x>0$ (recall Proposition \ref{Prop.global}), we know that $x\mapsto w(x)$ is strictly decreasing on $x\in[x_A, +\infty)$; since $\lim_{x\to+\infty}w(x)=0$, there exists a unique $x_B\in(x_A, +\infty)$ such that $w(x_B)=w(P_2)=w_-$. The point $(\sigma(x_B), w(x_B))$ is denoted by $P_B$. See Figure \ref{Fig.w-sigma-repulsivity} for the positions of $P_A$ and $P_B$. Under the re-normalization \eqref{Eq.renormalization}, the image of $P_A$ and $P_B$ are denoted by $Q_A(\tau_A, u_A)$ and $Q_B(\tau_B, u_B)$, respectively; see Figure \ref{Fig.u-tau-replusivity}.

\begin{figure}[htbp]
	\centering
	\includegraphics[width=1\textwidth]{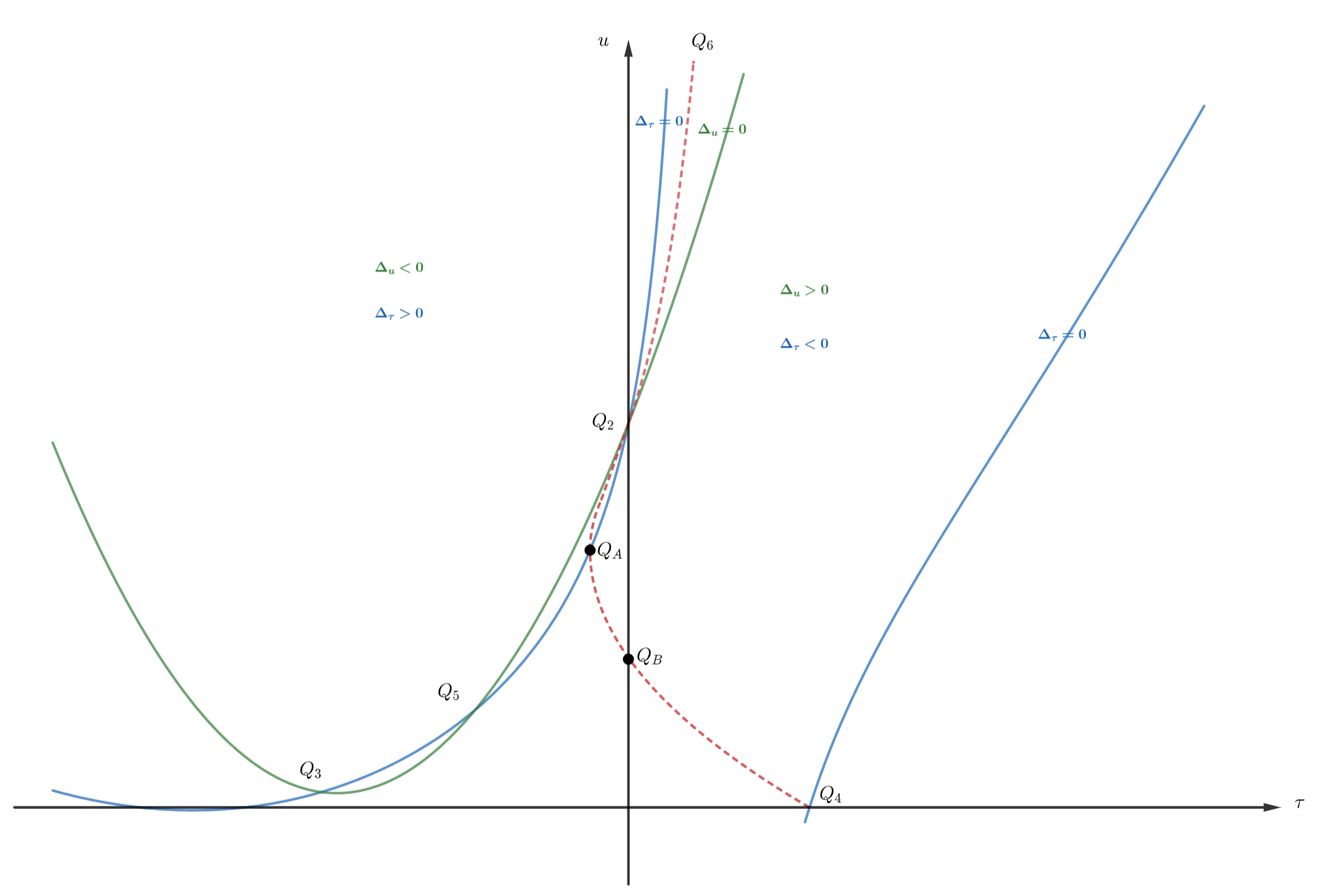}
	\caption{Phase portrait for the $\tau-u$ system \eqref{ODE_u_tau}: Positions of auxiliary points $Q_A$ and $Q_B$.}
	\label{Fig.u-tau-replusivity}
\end{figure}

We prove Proposition \ref{Prop.replusivity_reduce} by considering the four intervals $x\in(-\infty, 0]$, $x\in(0, x_A]$, $x\in(x_A, x_B]$, and $x\in(x_B, +\infty)$ separately. As a result, Proposition \ref{Prop.replusivity_reduce} is a direct consequence of Lemmas \ref{Lem.repulse1}--\ref{Lem.repulse4}.

\begin{lemma}\label{Lem.U_sigma1_compare}
    Let $d=\ell=3$. Let $\al\in(0,1)$ be such that the $Q_6-Q_2$ solution $u_F$ to \eqref{ODE_u_tau} passes through $Q_2$ smoothly (i.e. $u_F=u_L$, recalling Proposition \ref{Prop.sol_sonic_point}). We denote
    \[U_{\sigma_1}(\tau):=\frac19\left[\frac{2(\al+1)\tau+2(3\al+1)}{-2\tau+(3\al+1)}+1+\tau\right]^2,\quad\forall\ \tau\in(0,\al),\]
    then
    \begin{equation*}
        \cL\left[U_{\sigma_1}\right](\tau)<0\quad\text{and}\quad u_F(\tau)>U_{\sigma_1}(\tau),\quad\forall\ \tau\in(0,\al=\tau(Q_6)).
    \end{equation*}
\end{lemma}
\begin{proof}
    Recalling \eqref{Eq.Lu}, it is a brute force computation to obtain 
    \begin{align*}
        \cL\left[U_{\sigma_1}\right](\tau)=-\frac{16\tau(\tau+1-\al)}{81(-2\tau+3\al+1)^4}\left[\frac{2(\al+1)\tau+2(3\al+1)}{-2\tau+(3\al+1)}+1+\tau\right]\phi_1(\tau;\al),
    \end{align*}
    where
    \begin{align*}
        &\phi_1(\tau;\al):=\ 9(1-\al)^2(1+3\al)^2+6(1-\al)(1+3\al)(2+11\al-5\al^2)\tau\\
        &\quad+(1+3\al)(19+112\al-83\al^2)\tau^2+4(-9+2\al+67\al^2)\tau^3-4(9+29\al)\tau^4+16\tau^5>0
    \end{align*}
    for all $\al\in(0,1)$ and $\tau\in(0,\al)$. Here we have used the following facts:
    \begin{itemize}
        \item $2+11\al-5\al^2>0$ for $\al\in(0,1)$;
        \item the quadratic polynomial $\wt\phi_1(\tau;\al):=(1+3\al)(19+112\al-83\al^2)+4(-9+2\al+67\al^2)\tau-4(9+29\al)\tau^2$ satisfies $\wt\phi_1(0;\al)=(1+3\al)(19+112\al-83\al^2)>0$ and $\wt\phi(\al;\al)=19+133\al+225\al^2-97\al^3>225\al^2-97\al^3>0$ for all $\al\in(0,1)$, hence $\wt\phi_1(\tau;\al)>\min\left\{\wt\phi_1(0;\al), \wt\phi(\al;\al)\right\}>0$ for all $\tau\in(0,\al)$.
    \end{itemize}
    As a consequence, we obtain
    \begin{equation}\label{Eq.U_sigma1_compare}
        \cL\left[U_{\sigma_1}\right](\tau)<0, \quad\forall\ \tau\in(0,\al),\ \forall\ \al\in(0,1).
    \end{equation}
As $ U_{\sigma_1}$ is bounded on $(0,\al)$, the proof of $u_F(\tau)>U_{\sigma_1}(\tau)$ is similar to the proof of $u_{(2)}(\tau)<u_F(\tau)$ in Lemma \ref{Lem.u_2_compare}. 
\if0    At the sonic point $Q_2(0,1)$, we have $U_{\sigma_1}(0)=1=u_F(0)$ and
    \begin{align*}
        \frac{\mathrm dU_{\sigma_1}}{\mathrm d\tau}(0)-\frac{\mathrm du_{F}}{\mathrm d\tau}(0)=\frac{2(5\al+7)}{3(3\al+1)}-a_1=\frac{2(5\al+7)}{3(3\al+1)}-\frac{2\al+4\sqrt\al}{3\al}=\frac{4(1-\sqrt\al)^3}{3(3\al+1)\sqrt\al}>0
    \end{align*}
    for $\al\in(0,1)$. We assume for contradiction that there exists $\tau_*\in(0,\al)$ satisfying $u_F(\tau_*)=U_{\sigma_1}(\tau_*)$ and $u_F(\tau)>U_{\sigma_1}(\tau)$ for all $\tau\in(0,\tau_*)$, then $\frac{\mathrm dU_{\sigma_1}}{\mathrm d\tau}(\tau_*)\geq \frac{\mathrm du_{F}}{\mathrm d\tau}(\tau_*)$. Using \eqref{Eq.Lu} and $\Delta_\tau(\tau, u_F(\tau))<0$ for $\tau\in(0,\al)$, we obtain
    \begin{align*}
        \cL\left[U_{\sigma_1}\right](\tau_*)&=\Delta_\tau\left(\tau_*, U_{\sigma_1}(\tau_*)\right)\frac{\mathrm dU_{\sigma_1}}{\mathrm d\tau}(\tau_*)-\Delta_u\left(\tau_*, U_{\sigma_1}(\tau_*)\right)\\
        &\leq \Delta_\tau\left(\tau_*, u_F(\tau_*)\right)\frac{\mathrm du_{F}}{\mathrm d\tau}(\tau_*)-\Delta_u\left(\tau_*, u_F(\tau_*)\right)=0,
    \end{align*}
    which contradicts with \eqref{Eq.U_sigma1_compare}.\fi This completes the proof.
\end{proof}

\begin{lemma}[Repulsivity: $P_6-P_2$]\label{Lem.repulse1}
    Let $d=\ell=3$ and $\{r_n\}$ be given by Theorem \ref{Thm.ODE}. For each $r\in\{r_n\}$, let $(\sigma,w)$ be the solution to \eqref{autonomousODE} given by Theorem \ref{Thm.ODE}. Then $w'(x)>0$, $\sigma(x)+\sigma'(x)>0$ and $1-w(x)-w'(x)-\sigma(x)-\sigma'(x)>0$ for all $x\in(-\infty, 0]$.
\end{lemma}
\begin{proof}
    The proof for $x=0$, which corresponds to the sonic point $P_2$, is a direct computation. We introduce
    \begin{align*}
        e_1:&=\pa_\sigma\Delta_1(P_2)=-2(dw(P_2)-\ell(r-1))\sigma(P_2)<0,\\
        e_2:&=\pa_w\Delta_1(P_2)=\big(w(P_2)-w_1(\sigma(P_2))\big)\big(w(P_2)-w_{3}(\sigma(P_2))\big)<0,\\
        e_3:&=\pa_\sigma\Delta_2(P_2)=-2\sigma(P_2)^2<0,\\
        e_4:&=\pa_w\Delta_2(P_2)=\frac{(\ell+d-1)\sigma(P_2)}{\ell}\big(w(P_2)-w_2^+(\sigma(P_2))\big)<0.
    \end{align*}
    Then
    \[\frac{\mathrm dw}{\mathrm d\sigma}(\sigma(P_2))=c_-:=\frac{e_3-e_2-\sqrt{(e_3-e_2)^2+4e_1e_4}}{2|e_4|}\in\left(-\frac{e_3}{e_4}, -\frac{e_1}{e_2}\right)\subset (-1, 0),\]
    and
    \[w'(0)=\frac{e_1+e_2c_-}{2\sigma(P_2)(1+c_-)}>0, \quad \sigma'(0)=\frac{e_3+e_4c_-}{2\sigma(P_2)(1+c_-)}<0.\]
    Hence,
    \[\sigma(0)+\sigma'(0)=\frac{2\sigma(P_2)^2(1+c_-)+e_3+e_4c_-}{2\sigma(P_2)(1+c_-)}=\frac{(e_4-e_3)c_-}{2\sigma(P_2)(1+c_-)}>0,\]
    and
    \[1-w(0)-w'(0)-\sigma(0)-\sigma'(0)=-w'(0)-\sigma'(0)=-(1+c_-)\sigma'(0)>0.\]
    This completes the proof for $x=0$.

    Now we consider the case $x<0$. We define $F(x):=\sigma(x)+\sigma'(x)$ for $x\in\R$. A direct computation yields that
    \begin{equation}\label{Eq.F_expression}
        F(x)=-(d-1)\frac{\sigma(x)}{\ell\Delta(\sigma(x), w(x))}(w(x)-w_-)(w(x)-w_+),\quad\forall\ x\in\R\setminus\{0\},
    \end{equation}
    recalling \eqref{Eq.w+-}. See also \cite[Lemma 8.3]{MRRJ2}. For $x<0$, by the facts that $\Delta_1(\sigma(x), w(x))<0$ $\Delta(\sigma(x), w(x))<0$, $w'(x)=-\Delta_1/\Delta$ and $r-1<w(x)<w(P_2)=w_-<w_+$, we have $w'(x)>0$ and $F(x)=\sigma(x)+\sigma'(x)>0$. It remains to show that
    \begin{equation}\label{Eq.1-w-w'-F>0}
        1-w(x)-w'(x)-F(x)>0,\quad\forall\ x<0.
    \end{equation}
    Using \eqref{Eq.Delta}, \eqref{Eq.Delta_1} and \eqref{Eq.Delta_2}, we compute that
    \begin{equation}\label{Eq.1-w-w'-F_expression}
        1-w-w'-F=-\frac{A(w)\sigma^2-\frac1\ell B(w)\sigma-(1-w)[1-(2-r)w]}{\Delta(\sigma, w)},
    \end{equation}
    where\footnote{Here we abuse the notation $A$, which is different from the one in \eqref{Eq.A_def}.}
    \begin{align}
        &\qquad\qquad A(w):=(d-1)w+1-\ell(r-1),\label{Eq.A(w)}\\
        B(w):&=(d-1)w^2+(\ell+r-d-\ell r)w+\ell(r-1)=(d-1)(w-w_-)(w-w_+),\label{Eq.B(w)}
    \end{align}
    recalling \eqref{Eq.w+-}. We note that $A(w)(1-w)+B(w)=1-(2-r)w$, hence,
    \begin{equation*}
    	1-w-w'-F=-\frac{A(w)\sigma^2-\frac1\ell B(w)\sigma-(1-w)[A(w)(1-w)+B(w)]}{\Delta(\sigma, w)}.
    \end{equation*}
    For $x<0$, we have $\ell(r-1)/d<w(x)<w_-=w(P_2)<\min\{1, w_+\}$. Hence, $A(w(x))>0$, $B(w(x))>0$ and $(1-w(x))[A(w(x))(1-w(x))+B(w(x))]>0$. Let
    \[\bar\sigma(w):=\frac1{A(w)}\left(\frac{B(w)}{2\ell}+\sqrt{\frac{B(w)^2}{4\ell^2}+A(w)(1-w)[A(w)(1-w)+B(w)]}\right)\]
    be the unique positive root to the quadratic polynomial $\sigma\mapsto A(w)\sigma^2-\frac1\ell B(w)\sigma-(1-w)[A(w)(1-w)+B(w)]$, where $w\in(\ell(r-1)/d, w_-)$. Using $\ell>1$ gives that
    \[\bar \sigma(w){<}\frac1{A(w)}\left(\frac{B(w)}{2\ell}+A(w)(1-w)+\frac{B(w)}2\right):=\sigma_1(w),\quad\forall\ w\in\left(\frac{\ell(r-1)}{d}, w_-\right).\]
    To prove \eqref{Eq.1-w-w'-F>0}, it suffices to show that
    \begin{equation}\label{Eq.sigma>sigma_1}
    	\sigma(x)>\sigma_1(w(x)),\quad\forall\ x<0.
    \end{equation}

    Now we focus on $d=\ell=3$. In this case, by $A(w)(1-w)+B(w)=1-(2-r)w$, we have
    \[\sigma_1(w)=\frac23\frac{B(w)}{A(w)}+1-w=\frac23\frac{1-(2-r)w}{A(w)}+\frac{1-w}{3}.\]
    It follows \eqref{Eq.w+-} and \eqref{Eq.renormalization} that $w_-=a/(1+a)$ and $r=\frac{a^2+6a+3}{(a+1)(a+3)}$. It is a straightforward computation to verify that
    \begin{equation}\label{Eq.U_sigma1_relation}
        U_{\sigma_1}(\tau)=(1+a)^2\left(\sigma_1\left(-\frac{\tau}{1+a}+w_-\right)\right)^2,\quad\forall\ \tau\in(0,\al).
    \end{equation}

    Let $(\tau_\text s(x), u_\text s(x)):=\Psi(\sigma(x), w(x))=(-(1+a)(w(x)-w_-), (1+a)^2\sigma(x)^2)$ for $x\in\R$. Then $x\mapsto \tau_\text s(x)$ is strictly decreasing on $x\leq x_A$, and then the function $u(\tau):=u_\text s\left(\tau_\text s^{-1}(\tau)\right)$ is a smooth solution to \eqref{ODE_u_tau} connecting $Q_A$ and $Q_6$. By Lemma \ref{Lem.U_sigma1_compare}, we have $u_\text s(x)=u(\tau_\text s(x))>U_{\sigma_1}(\tau_\text s(x))$ for $x<0$. So, \eqref{Eq.U_sigma1_relation} implies that
    \[(1+a)^{2}\sigma(x)^2>(1+a)^2\left(\sigma_1\left(-\frac{\tau_\text s(x)}{1+a}+w_-\right)\right)^2
    =(1+a)^2\big(\sigma_1(w(x))\big)^2, \quad\forall\ x<0.\]
    This proves \eqref{Eq.sigma>sigma_1}.
\end{proof}

\begin{lemma}[Repulsivity: $P_2-P_A$]\label{Lem.repulse2}
    Let $d=\ell=3$ and $\{r_n\}$ be given by Theorem \ref{Thm.ODE}. For each $r\in\{r_n\}$, let $(\sigma,w)$ be the solution to \eqref{autonomousODE} given by Theorem \ref{Thm.ODE}. Then $w'(x)>0$, $\sigma(x)+\sigma'(x)>0$ and $1-w(x)-w'(x)-\sigma(x)-\sigma'(x)>0$ for all $x\in(0, x_A)$. Moreover, \eqref{equal condition for repulsivity property} holds for $x=x_A$.
\end{lemma}
\begin{proof}
    For $x\in(0, x_A)$, the solution curve $x\mapsto(\sigma(x), w(x))$ lies in the region $\{(\sigma, w): \Delta(\sigma, w)>0, \Delta_1(\sigma, w)<0, \Delta_2(\sigma, w)>0\}$, hence $w'(x)>0$ and $w_-=w(0)<w(x)<w(P_5)=r/2$, which implies that $F(x)>0$, recalling \eqref{Eq.F_expression}. It suffices to consider $1-w-w'-F$. We recall \eqref{Eq.1-w-w'-F_expression}, \eqref{Eq.A(w)} and \eqref{Eq.B(w)}. For $x\in(0, x_A)$, since $w_-<w(x)<r/2<1/(2-r)$, we have $A(w(x))>0$, $B(w(x))<0$ and $(1-w(x))[1-(2-r)w(x)]>0$. We still denote the unique positive root of the quadratic polynomial $\sigma\mapsto A(w)\sigma^2-\frac1\ell B(w)\sigma-(1-w)[A(w)(1-w)+B(w)]$ by $\bar \sigma(w)$, defined for $w\in(w_-, r/2)$. Note that
      \begin{align*}A(w)(1-w)^2-B(w)(1-w)/\ell-(1-w)[A(w)(1-w)+B(w)]\\=-(1+1/\ell)B(w)(1-w)>0,\end{align*}
    hence, $\bar\sigma(w)<1-w$ for $w\in(w_-, r/2)$, and then
    \begin{align*}
        \bar\sigma(w)^2&=\frac1{A(w)}\left(\frac{B(w)\bar\sigma(w)}\ell+(1-w)[A(w)(1-w)+B(w)]\right)\\
        &>\frac1{A(w)}\left(\frac{B(w)(1-w)}\ell+(1-w)[A(w)(1-w)+B(w)]\right)\\
        &=\frac{1-w}{A(w)}\left(A(w)(1-w)+\frac{\ell+1}{\ell}B(w)\right)=:\sigma_2(w)^2,\quad \forall\ w\in(w_-, r/2).
    \end{align*}
    Here we also require $\sigma_2(w)>0$ for $w\in(w_-, r/2)$. To prove that $1-w(x)-w'(x)-F(x)>0$ for all $x\in(0, x_A)$, it remains to show that
    \begin{equation}
        \sigma(x)<\sigma_2(w(x)),\quad \forall\ x\in(0, x_A).
    \end{equation}

    Now we focus on $d=\ell=3$. In this case, we have
    \begin{align*}
        \sigma_2(w)^2&=\frac{1-w}{A(w)}\left(A(w)(1-w)+\frac{4}{3}[1-(2-r)w-A(w)(1-w)]\right)\\
        &=\frac{4(1-w)}{A(w)}[1-(2-r)w]-\frac{(1-w)^2}{3}
    \end{align*}
    Using \eqref{Eq.w+-} and \eqref{Eq.renormalization}, we compute that
    \begin{align}\label{second barrier funtion}
        U_{\sigma_2}(\tau):&=(1+a)^2\left(\sigma_2\left(-\frac{\tau}{1+a}+w_-\right)\right)^2\\
        &=\frac{4(1+\tau)[(3\al+1)+(\al+1)\tau]}{3(-2\tau+3\al+1)}-\frac13(1+\tau)^2.\nonumber
    \end{align}
    Using an argument similar to the proof of Lemma \ref{Lem.repulse1}, we only need to prove
    \begin{equation}\label{Eq.u<U_sigma2}
        u(\tau)<U_{\sigma_2}(\tau),\quad\forall\ \tau\in(\tau_A, 0).
    \end{equation}

    Recall $P_5=(\sqrt 3r/6, r/2)$. Then the re-normalization \eqref{Eq.renormalization} yields that $$Q_5=\left(\frac{a^2-3}{2(a+3)}, \frac{(a^2+6a+3)^2}{12(a+3)^2}\right).$$
    Hence, $\tau_A\in\left(\frac{a^2-3}{2(a+3)},0\right)$. By the choice of $Q_A$, we know that $Q_A$ is the intersection point of the solution curve and $\tau\mapsto u_\text b(\tau)$; see Figure \ref{Fig.u-tau-replusivity}.  We claim that
    \begin{align}
        U_{\sigma_2}(\tau)>u_\text b(\tau),\quad\forall\ \tau\in\left(\frac{a^2-3}{2(a+3)},0\right),\label{Eq.U_sigma2>u_b}\\
        \cL\left[U_{\sigma_2}\right](\tau)>0, \quad\forall\ \tau\in\left(\frac{a^2-3}{2(a+3)},0\right).\label{Eq.U_sigma2_compare}
    \end{align}
    Assuming \eqref{Eq.U_sigma2>u_b} and \eqref{Eq.U_sigma2_compare}, we proceed to prove \eqref{Eq.u<U_sigma2} as follows. First of all, \eqref{Eq.U_sigma2>u_b} and $\tau_A\in\left(\frac{a^2-3}{2(a+3)},0\right)$ imply that $U_{\sigma_2}(\tau_A)>u_b(\tau_A)=u(\tau_A)$. We assume for contradiction that $U_{\sigma_2}(\tau_*)=u(\tau_*)$ for some $\tau_*\in(\tau_A, 0)$ and $U_{\sigma_2}(\tau)>u(\tau)$ for all $\tau\in(\tau_A, \tau_*)$. Then $\frac{\mathrm dU_{\sigma_2}}{\mathrm d\tau}(\tau_*)\leq \frac{\mathrm du}{\mathrm d\tau}(\tau_*)$. Using \eqref{Eq.Lu} and the fact that $\Delta_\tau(\tau, u(\tau))>0$ for $\tau\in(\tau_A, 0)$, we obtain
    \begin{align*}
        \cL\left[U_{\sigma_2}\right](\tau_*)&=\Delta_\tau\left(\tau_*, U_{\sigma_2}(\tau_*)\right)\frac{\mathrm dU_{\sigma_2}}{\mathrm d\tau}(\tau_*)-\Delta_u\left(\tau_*, U_{\sigma_2}(\tau_*)\right)\\
        &\leq \Delta_\tau\left(\tau_*, u(\tau_*)\right)\frac{\mathrm du}{\mathrm d\tau}(\tau_*)-\Delta_u\left(\tau_*, u(\tau_*)\right)=0,
    \end{align*}
    which contradicts \eqref{Eq.U_sigma2_compare}. This proves \eqref{Eq.u<U_sigma2}. So, it suffices to prove \eqref{Eq.U_sigma2>u_b} and \eqref{Eq.U_sigma2_compare}.

    \underline{\textit{Proof of \eqref{Eq.U_sigma2>u_b}}.} We compute that
    \[U_{\sigma_2}(\tau)-u_\text b(\tau)=-\frac{2\tau(1+\tau)(-2\tau+\al-1)(-\tau+\al-1)}{3(-2\tau+3\al+1)(\al-\tau)}.\]
    For $\tau\in\left(\frac{a^2-3}{2(a+3)},0\right)\subset(-1,0)$, by \eqref{Eq.renormalization}, we have
    \begin{align*}
        -\tau+\al-1<-2\tau+\al-1<\frac{3-a^2}{a+3}+\frac{a(a+1)}{a+3}-1=0,
    \end{align*}
    hence $U_{\sigma_2}(\tau)>u_\text b(\tau)$. This proves \eqref{Eq.U_sigma2>u_b}.

    \underline{\textit{Proof of \eqref{Eq.U_sigma2_compare}}.} Recalling \eqref{Eq.Lu}, one can obtain  through brute force computation    
    \[\cL\left[U_{\sigma_2}\right](\tau)=\frac{16\tau(\tau+1)(-1+\alpha-2\tau)(-1+\alpha-\tau)[-(1+\alpha)\tau^{2}+(3\alpha+1)(-1+\alpha-2\tau)]}{3(-2\tau+3\al+1)^3}.\]
    For $\tau\in\left(\frac{a^2-3}{2(a+3)},0\right)\subset(-1,0)$, since $-1+\al-\tau<-1+\al-2\tau<0$, we have $\cL\left[U_{\sigma_2}\right](\tau)>0$. This proves \eqref{Eq.U_sigma2_compare}.

    Finally, for $x=x_A$, we have $w'(x_A)=0$ and $F(x_A)>0$ thanks to \eqref{Eq.F_expression} and $w_-<w(x_A)<w_+$. Combining this with $w(x_A)+\sigma(x_A)<1$ and the strict decrease of $\sigma$, we obtain
    \begin{align*}
        1-w(x_A)-\max\{0, w'(x_A)\}-\left|\sigma(x_A)+\sigma'(x_A)\right|=1-w(x_A)-\sigma(x_A)-\sigma'(x_A)>0.
    \end{align*}
    Thus, \eqref{equal condition for repulsivity property} holds for $x=x_A$.
\end{proof}

\begin{lemma}[Repulsivity: $P_A-P_B$]\label{Lem.repulse3}
    Let $d=\ell=3$ and $\{r_n\}$ be given by Theorem \ref{Thm.ODE}. For each $r\in\{r_n\}$, let $(\sigma,w)$ be the solution to \eqref{autonomousODE} be given by Theorem \ref{Thm.ODE}. Then $w'(x)<0$, $\sigma(x)+\sigma'(x)>0$ and $1-w(x)-\sigma(x)-\sigma'(x)>0$ for all $x\in(x_A, x_B)$. Moreover, \eqref{equal condition for repulsivity property} holds for $x=x_B$.
\end{lemma}
\begin{proof}
    For $x\in(x_A, x_B)$, we have $w'(x)<0$, $w_-<w(x)<w(x_A)<r/2<w_+$ and $w(x)+\sigma(x)<1$. Hence, $1-w(x)-\sigma(x)-\sigma'(x)>-\sigma'(x)>0$ thanks to the strict decrease of $\sigma$. It also follows from \eqref{Eq.F_expression} that $\sigma(x)+\sigma'(x)=F(x)>0$ for all $x\in(x_A, x_B)$. If $x=x_B$, then $w'(x_B)<0$, $w(x_B)=w_-<1$ and $\sigma(x_B)+\sigma'(x_B)=F(x_B)=0$, and thus \eqref{equal condition for repulsivity property} holds for $x=x_B$.
\end{proof}

\begin{lemma}[Repulsivity: $P_B-P_4$]\label{Lem.repulse4}
    Let $d=\ell=3$ and $\{r_n\}$ given by Theorem \ref{Thm.ODE}. For each $r\in\{r_n\}$, let $(\sigma,w)$ be the solution to \eqref{autonomousODE} given by Theorem \ref{Thm.ODE}. Then $w'(x)<0$, $\sigma(x)+\sigma'(x)<0$ and $1-w(x)+\sigma(x)+\sigma'(x)>0$ for all $x\in(x_B, +\infty)$.
\end{lemma}
\begin{proof}
    For $x\in(x_B, +\infty)$, we have $w'(x)<0$, $0<a(1+a)\sigma(x)^2<w(x)<w(x_B)=w_-<w_+$ and $\Delta(\sigma(x), w(x))>0$. Hence, \eqref{Eq.F_expression} implies that $\sigma(x)+\sigma'(x)=F(x)<0$. It suffices to show that
    \begin{equation}\label{Eq.1-w+sigma+sigma'>0}
        1-w(x)+\sigma(x)+\sigma'(x)>0,\quad\forall\ x>x_B.
    \end{equation}
    Indeed, for $x>x_B$, we have $\sigma(x)<\sigma(0)=1-w_-$, so $\Delta(\sigma(x), w(x))=(1-w(x))^2-\sigma(x)^2>(1-w(x))^2-(1-w_-)^2=(2-w(x)-w_-)(w_--w(x))>0$, and hence,
    \begin{align*}
        &\ell\Delta(\sigma(x), w(x))\left(1-w(x)+\sigma(x)+\sigma'(x)\right)=\ell\Delta(\sigma(x), w(x))\left(1-w(x)+F(x)\right)\\
        =&\ {\ell}(1-w(x))\Delta(\sigma(x), w(x))-(d-1)\sigma(x)\left(w_--w(x)\right)\left(w_+-w(x)\right)\\
        >&\ {2}(1-w(x))(2-w(x)-w_-)(w_--w(x))-{2}(1-w(x))\left(w_--w(x)\right)\left(w_+-w(x)\right)\\
        =&\ 2(1-w(x))(w_--w(x))(2-w_--w_+),
        %\left[(d-2)w(x)+2-w_--(d-1)w_+\right],
    \end{align*}
    where we have used $\sigma(x)<1-w(x)$ and $\ell>d-1=2$ as $d=\ell=3$. Then \eqref{Eq.w+-} implies $2-w_--w_+=2-r>0$ for all $r<3-\sqrt3<2$. This proves \eqref{Eq.1-w+sigma+sigma'>0}.
    %, by taking $r\in(5/4, 3-\sqrt3)$, which is admissible, recalling that $r\in\{r_n\}$ and $r_n\uparrow 3-\sqrt 3>5/4$ as $n\to\infty$.
\end{proof}

\section*{Acknowledgments}
The authors would like to thank Gonzalo Cao-Labora and Jia Shi for clarifying the validity of Theorem \ref{Thm.stability} for  the general form of the isentropic compressible Navier-Stokes equations, with the general dissipation term $-\mu\Delta \mathbf u-(\la+\mu)\nabla\dive\mathbf u$. D. Wei is partially supported by the National Key R\&D Program of China under the grant 2021YFA1001500. Z. Zhang is partially supported by  NSF of China  under Grant 12288101.


\begin{thebibliography}{99}
	
%\bibitem{Carter} D. S. Carter, L'Hospital's rule for complex-valued functions. \textit{Amer. Math. Monthly},  \textbf{65} (1958), 264–266.


\bibitem{Speck2022} L. Abbrescia and J. Speck, The emergence of the singular boundary from the crease in 3D compressible Euler flow. arXiv:2207.07107, 2022.


\bibitem{AHS2023} C. Alexander, M. Had\v zi\'{c} and M. Schrecker, Supersonic Gravitational Collapse for Non-Isentropic Gaseous Stars. arXiv:2311.18795v2, 2024.


\bibitem{Alinhac1999} S. Alinhac, Blowup of small data solutions for a quasilinear wave equation in two space dimensions.
\textit{Ann. of Math.}, \textbf{149} (1999),  97--127.

\bibitem{Alinhac2001} S. Alinhac, The null condition for quasilinear wave equations in two space dimensions. II.
\textit{Amer. J. Math.},
\textbf{193} (2001), 1071--1101.

%\bibitem{Biasi2021} A. Biasi, Self-similar solutions to the compressible Euler equations and their instabilities. {\it
%Commun. Nonlinear Sci. Numer. Simul.}, {\bf103} (2021), Paper No. 106014, 28 pp.


\bibitem{Bourgain2000} J. Bourgain, Problems in Hamiltonian PDE’s.
\textit{Geom. Funct. Anal.},
Special volume, Part I (2000), 32--56.

\bibitem{BCLGS} T. Buckmaster, G. Cao-Labora and J. G\'omez-Serrano, Smooth imploding solutions for 3D compressible fluids. To appear in {\it Forum Math. Pi}; arXiv:2208.09445, 2022.


\bibitem{Buckmaster-Chen} T. Buckmaster and J. Chen, Blowup for the defocusing septic complex-valued nonlinear wave equation in $\mathbb R^{4+1}$. arXiv:2410.15619, 2024.


\bibitem{Buck-D-Shko-Vicol} T. Buckmaster, T. D. Drivas, S. Shkoller and V. Vicol, Simultaneous development of shocks and cusps for 2D Euler with azimuthal symmetry from smooth data. \textit{Ann. PDE}, \textbf{8} (2022), Paper No. 26, 199 pp.

\bibitem{BDSV-survey} T. Buckmaster, T. D. Drivas, S. Shkoller and V. Vicol, Formation and development of singularities for the compressible Euler equations. {\it ICM-International Congress of Mathematicians. Vol. 5. Sections 9-11}, 3636--3659. EMS Press, Berlin, 2023.

\bibitem{Buck-Iyer} T. Buckmaster and S. Iyer, Formation of unstable shocks for 2D isentropic compressible Euler. \textit{Comm. Math. Phys.}, \textbf{389} (2022), 197--271.

\bibitem{Buck-Shko-Vicol1} T. Buckmaster, S. Shkoller and V. Vicol, Formation of shocks for 2D isentropic compressible Euler. \textit{Comm. Pure Appl. Math.}, \textbf{75} (2022), 2069--2120.

\bibitem{Buck-Shko-Vicol2} T. Buckmaster, S. Shkoller and V. Vicol, Shock formation and vorticity creation for 3D Euler. \textit{Comm. Pure Appl. Math.}, \textbf{76} (2023), 1965--2072.

\bibitem{Buck-Shko-Vicol3} T. Buckmaster, S. Shkoller and V. Vicol, Formation of point shocks for 3D compressible Euler. \textit{Comm. Pure Appl. Math.}, \textbf{76} (2023), 2073--2191.

\bibitem{CLGSSS} G. Cao-Labora, J. G\'omez-Serrano, J. Shi and G. Staffilani, Non-radial implosion for compressible Euler and Navier-Stokes in $\mathbb T^3$ and $\mathbb R^3$. arXiv:2310.05325, 2023.

\bibitem{CLGSSS-2024} G. Cao-Labora, J. G\'omez-Serrano, J. Shi and G. Staffilani, Non-radial implosion for the defocusing nonlinear Schr\"odinger equation in $\mathbb T^d$ and $\mathbb R^d$. arXiv:2410.04532, 2024.

\bibitem{Chemin1990} J.-Y. Chemin, Dynamique des gaz \`a masse totale finie. {\it Asymptotic Anal.}, {\bf 3} (1990), 215--220.

\bibitem{Chen2023} J. Chen, Nearly self-similar blowup of the slightly perturbed homogeneous Landau equation with very
soft potentials. arXiv:2311.11511, 2023.

\bibitem{Chen2024} J. Chen, Vorticity blowup in compressible Euler equations in $\R^d$, $d\geq3$. arXiv:2408.04319, 2024.

\bibitem{CCSV2024} J. Chen, G. Cialdea, S. Shkoller and V. Vicol, Vorticity blowup in 2D compressible Euler equations. arXiv:2407.06455, 2024.

\bibitem{Chen2021} J. Chen and T. Y. Hou, Finite time blowup of 2D Boussinesq and 3D Euler equations with $C^{1,\alpha}$ velocity
and boundary.
\textit{Comm. Math. Phys.},
\textbf{383} (2021), 1559--1667.


\bibitem{Chen2022-1} J. Chen and T. Y. Hou, Stable nearly self-similar blowup of the 2D Boussinesq and 3D Euler equations
with smooth data I: Analysis. arXiv:2210.07191v3, 2023.

\bibitem{Chen2022-2} J. Chen and T. Y. Hou, Stable nearly self-similar blowup of the 2D Boussinesq and 3D Euler equations
with smooth data II: Rigorous numerics. {\it Multiscale Model. Simul.}, {\bf23} (2025), 25--130.

\bibitem{ChenHouHuang2021} J. Chen, T. Y. Hou and D. Huang, On the finite time blowup of the De Gregorio model for the 3D Euler equations. \textit{Comm. Pure Appl. Math.},
\textbf{74} (2021), 1282--1350.

\bibitem{CCK2004} Y. Cho, H. J. Choe and H. Kim, Unique solvability of the initial boundary value problems for compressible viscous fluids. \textit{J. Math. Pures Appl. (9)}, \textbf{83} (2004), 243--275.

\bibitem{CK2006} Y. Cho and H. Kim, On classical solutions of the compressible Navier–Stokes equations with nonnegative initial densities. \textit{Manuscripta Math.}, \textbf{120} (2006), 91--129.

\bibitem{CK2003} H. J. Choe and H. Kim, Strong solutions of the Navier–Stokes equations for isentropic compressible fluids. \textit{J. Differential Equations}, \textbf{190} (2003), 504--523.

\bibitem{Christ2007} D. Christodoulou, {\it The formation of shocks in 3-dimensional fluids.}  {EMS Monographs in Mathematics}. European Mathematical Society (EMS), Z\"urich, 2007. viii+992 pp.

\bibitem{Christodoulou2019} D. Christodoulou, {\it The shock development problem.} EMS Monographs in Mathematics. European Mathematical Society (EMS), Z\"urich, 2019. ix+920 pp.

\bibitem{Chris-Lisibach2016} D. Christodoulou and A. Lisibach, Shock development in spherical symmetry. {\it Ann. PDE}, {\bf2} (2016), Art. 3, 246 pp.

\bibitem{Christ2014} D. Christodoulou and S. Miao, {\it Compressible flow and Euler's equations.}
Surveys of Modern Mathematics, vol. \textbf 9.
International Press, Somerville, MA; Higher Education Press, Beijing, 2014.


%\bibitem{Dafermos2010} C. Dafermos, {\it Hyperbolic conservation laws in continuum physics. Third edition.}{Grundlehren der mathematischen Wissenschaften}\textbf{325} Springer-Verlag, Berlin, 2010. xxxvi+708 pp.


\bibitem{Danchin2001} R. Danchin, Local theory in critical spaces for compressible viscous and heat-conductive gases. \textit{Comm. Partial Differential Equations}, \textbf {26} (2001), 1183--1233.

\bibitem{Drivas-Elgindi} T. D. Drivas and T. M. Elgindi, Singularity formation in the incompressible Euler equation in finite and infinite time. \textit{EMS Surv. Math. Sci.}, \textbf{10} (2023), 1--100.

\bibitem{Elgindi2021-1} T. M. Elgindi, Finite-time singularity formation for $C^{1,\alpha}$ solutions to the incompressible Euler equations on $\mathbb R^3$. \textit{Ann. of Math.}, \textbf{194} (2021), 647--727.


\bibitem{Elgindi2021-2}  T. M. Elgindi, T. Ghoul and N. Masmoudi, On the stability of self-similar blow-up for $C^{1,\alpha}$ solutions to the incompressible Euler equations on $\mathbb R^3$.
\textit{Camb. J. Math.}, \textbf{9} (2021), 1035--1075.

\bibitem{Feireisl2001} E. Feireisl, A. Novotn\'y and H. Petzeltov\'a, On the global existence of globally defined weak solutions to the Navier–Stokes equations of isentropic
compressible fluids. \textit{J. Math. Fluid Mech.}, \textbf{3} (2001), 358--392.



%\bibitem{Germain2021} P. Germain and T. Iwabuchi, Self-similar solutions of the compressible Navier-Stokes equations. \textit{Arch. Ration. Mech. Anal.}, \textbf{240} (2021), 1645--1673.


\bibitem{Guderley1942} G. Guderley, Starke kugelige und zylindrische Verdichtungsst\"osse in der N\"ahe des Kugelmittelpunktes bzw. der Zylinderachse.
\textit{Luftfahrtforschung},
\textbf{19} (1942), 302--311.


\bibitem{GHJ2021-1} Y. Guo, M. Had\v zi\'{c} and J. Jang, Continued gravitational collapse for Newtonian stars.
\textit{Arch. Ration. Mech. Anal.}, \textbf{239} (2021), 431--552.


\bibitem{GHJ2021-2} Y. Guo, M. Had\v zi\'{c} and J. Jang, Larson-Penston self-similar gravitational collapse. \textit{Comm. Math. Phys.}, \textbf{386} (2021), 1551--1601.

\bibitem{GHJ2023} Y. Guo, M. Had\v zi\'{c} and J. Jang, Naked singularities in the Einstein-Euler system. \textit{Ann. PDE}, \textbf{9} (2023), Paper No. 4, 182 pp.

\bibitem{GHJS2022} Y. Guo, M. Had\v zi\'{c}, J. Jang and M. Schrecker, Gravitational collapse for polytropic  gaseous stars: self-similar solutions. \textit{Arch. Ration. Mech. Anal.}, \textbf{246} (2022), 957--1066.


%\bibitem{GJ2006} Z. Guo and S. Jiang, Self-similar solutions to the isothermal compressible Navier-Stokes equations. \textit{IMA J. Appl. Math.}, \textbf{71} (2006), 658--669.


\bibitem{HQWW2024} D. Huang, X. Qin, X. Wang and D. Wei, Self-similar finite-time blowups with smooth profiles of the
generalized Constantin-Lax-Majda model.
\textit{Arch. Ration. Mech. Anal.}, \textbf{248} (2024), Paper No. 22, 65 pp.


\bibitem{HQWW2023} D. Huang, X. Qin, X. Wang and D. Wei, On the exact self-similar finite-time blowup of the Hou-Luo
model with smooth profiles. arXiv:2308.01528v2, 2024.


\bibitem{HTW2023} D. Huang, J. Tong and D. Wei, On self-similar finite-time blowups of the de Gregorio model on the real line.
\textit{Comm. Math. Phys.}, \textbf{402} (2023), 2791--2829.


%\bibitem{HLX2011} X. Huang, J. Li and  Z. Xin, Blowup criterion for viscous baratropic flows with vacuum states. \textit{Comm. Math. Phys.}, \textbf{301} (2011), 23--35.
%
%\bibitem{HX2010} X. Huang and Z. Xin, Blowup criterion for viscous baratropic flows with vacuum states. \textit{Sci. China Math.}, \textbf {53} (2010), 671--686.



\bibitem{Itaya1976} N. Itaya, On the initial value problem of the motion of compressible viscous fluid, especially on the problem of uniqueness. \textit{J. Math. Kyoto Univ.}, \textbf{16} (1976), 413--427.

\bibitem{JZ2001} S. Jiang and P. Zhang, Global spherically symmetric solutions of the compressible isentropic Navier–Stokes equations. \textit{Comm. Math. Phys.}, \textbf{215} (2001), 559--581.


\bibitem{JZ2003} S. Jiang and P. Zhang, Axisymmetric solutions of the 3-D Navier–Stokes equations for compressible isentropic flows. \textit{J. Math. Pures Appl. (9)}, \textbf{82} (2003), 949--973.

\bibitem{Kato1975} T. Kato, The Cauchy problem for quasi-linear symmetric hyperbolic systems. {\it Arch. Ration. Mech. Anal.}, \textbf{58} (1975), 181--205.

\bibitem{Lax1964} P. D. Lax, Development of singularities of solutions of nonlinear hyperbolic partial differential equations.
{\it J. Mathematical Phys.}, {\bf 5} (1964), 611--613.

\bibitem{Lax1973} P. D. Lax, \textit{Hyperbolic systems of conservation laws and the mathematical theory of shock waves.} Conference Board of the Mathematical Sciences Regional Conference Series in Applied Mathematics, No. \textbf{11}. Society for Industrial and Applied Mathematics, Philadelphia, PA, 1973. v+48 pp.

\bibitem{Lebaud1994} M.-P. Lebaud, Description de la formation d'un choc dans le $p$-syst\`eme. {\it J. Math. Pures Appl. (9)}, {\bf 73} (1994), 523--565.

\bibitem{LWX2019} H. Li, Y. Wang and Z. Xin, Non-existence of classical solutions with finite energy to the Cauchy problem of the compressible Navier-Stokes equations. \textit{Arch. Ration. Mech. Anal.}, \textbf{232} (2019), 557--590.


%\bibitem{LCX2013} T. Li, P. Chen and J. Xie, Self-similar solutions of the compressible flow in one-space dimension. \textit{J. Appl. Math.}, (2013), Art. ID 194704, 5 pp.


\bibitem{Lions1998} P. L. Lions, {\it Mathematical Topics in Fluid Mechanics, vol. 2: Compressible Models.} Oxford Lecture Ser. Math. Appl., {\bf10}. Oxford Sci. Publ. The Clarendon Press, Oxford University Press, New York, 1998. xiv+348 pp.

\bibitem{Luk2018} J. Luk and J. Speck, Shock formation in solutions to the 2D compressible Euler equations in the presence of non-zero vorticity. \textit{Invent. Math.}, \textbf{214} (2018), 1--169.

\bibitem{Luk2024} J. Luk and J. Speck, The stability of simple plane-symmetric shock formation for three-dimensional compressible Euler flow with vorticity and entropy. \textit{Anal. PDE}, \textbf{17} (2024), 831--941.



\bibitem{HL2014}  G. Luo and T. Y. Hou, Toward the finite-time blowup of the 3D incompressible Euler equations: a numerical investigation. \textit{Multiscale Model. Simul.}, \textbf{12} (2014), 1722--1776.




\bibitem{Majda1984} A. J. Majda, {\it Compressible fluid flow and systems of conservation laws in several space variables.}
{Applied Mathematical Sciences}, \textbf{53}. Springer-Verlag, New York, 1984. viii+159 pp.

\bibitem{MUK1987} T. Makino, S. Ukai and S. Kawashima, On compactly supported solutions of the compressible Euler equation. In {\it Recent topics in nonlinear PDE, III (Tokyo, 1986)}, North-Holland Math. Stud. {\bf148}, Lecture Notes Numer. Appl. Anal. {\bf9}. North-Holland, Amsterdam, pp. 173--183.



\bibitem{MRRJ2} F. Merle, P. Raph\"ael, I. Rodnianski, and J. Szeftel, On the implosion of a compressible fluid I: Smooth self-similar inviscid profiles. \textit{Ann. of Math. (2)}, \textbf{196} (2022), 567--778.

\bibitem{MRRJ3} F. Merle, P. Raph\"ael, I. Rodnianski, and J. Szeftel, On the implosion of a compressible fluid II: Singularity formation. \textit{Ann. of Math. (2)}, \textbf{196} (2022), 779--889.


\bibitem{MRRJ4} F. Merle, P. Raph\"ael, I. Rodnianski, and J. Szeftel, On blow up for the energy super critical defocusing nonlinear Schrödinger equations. \textit{Invent. Math.}, \textbf{227} (2022), 247--413.

\bibitem{Nash1958} J. Nash,  Continuity of solutions of parabolic and elliptic equations. {\it Amer. J. Math.},  80 (1958), 931-954.

\bibitem{Nash1962}  J. Nash, Le probl\`eme de Cauchy pour les \'equations diff\'erentielles d'un fluide g\'en\'eral. \textit{Bull. Soc. Math. France}, \textbf{90} (1962), 487--497.

\bibitem{Rozanova2008}  O. Rozanova, Blow up of smooth solutions to the compressible Navier-Stokes equations with the data
highly decreasing at infinity. \textit{J. Differential Equations},
\textbf{245} (2008), 1762--1774.




\bibitem{Sedov1959} L. I. Sedov, {\it Similarity and dimensional methods in mechanics.} Translation by Morris Friedman. {Academic Press, New York-London}, 1959. xvi+363 pp.


\bibitem{Serrin1959}  J. Serrin, On the uniqueness of compressible fluid motions. \textit{Arch. Ration. Mech. Anal.}, \textbf{3} (1959), 271--288.

\bibitem{SWZ2024_1} F. Shao, D. Wei and Z. Zhang, Self-similar imploding solutions of the relativistic Euler equations. arXiv:2403.11471, 2024.

\bibitem{SWZ2024_2} F. Shao, D. Wei and Z. Zhang, On blow-up for the supercritical defocusing nonlinear wave equation. arXiv:2405.19674, 2024.

\bibitem{SV} S. Shkoller and V. Vicol, The geometry of maximal development and shock formation for the Euler equations in multiple space dimensions.
\textit{Invent. Math.},  \textbf{237}(2024), 871--1252.

\bibitem{Sideris1985} T. C. Sideris, Formation of singularities in three-dimensional compressible fluids. \textit{Comm. Math. Phys.}, \textbf{101}(1985), 475--485.

\bibitem{SWZ2011} Y. Sun, C. Wang and Z. Zhang, A Beale-Kato-Majda blow-up criterion for the 3-D compressible Navier-Stokes equations. \textit{J. Math. Pures Appl. (9)}, \textbf{95} (2011), 36--47.

\bibitem{Xin1998} Z. Xin, Blowup of smooth solutions to the compressible Navier-Stokes equation with compact density. \textit{Comm. Pure Appl. Math.}, \textbf{51} (1998), 229--240.

\bibitem{XY2013} Z. Xin and W. Yan, On blowup of classical solutions to the compressible Navier-Stokes equations. \textit{Comm. Math. Phys.}, \textbf{321} (2013), 529--541.

\bibitem{Yin2004} H. Yin, Formation and construction of a shock wave for 3-D compressible Euler equations with the spherical initial data. {\it Nagoya Math. J.}, {\bf 175} (2004), 125--164.


%	\bibitem{BR} D. Bowman and A. Regev, Counting symmetry classes of dissections of a convex regular polygon, \textit{Adv. Appl. Math.}, \textbf{56} (2014), 35-55.
	
%	\bibitem{SWZ} F. Shao, D. Wei and Z. Zhang, Self-similar algebraic spiral solution of 2-D incompressible Euler equations, arXiv: 2305.05182, 2023.
\end{thebibliography}
\end{document}